%% file: article.tex
\documentclass{article}[11pt]

\input{tex/0_preamble.tex}

\begin{document}

% COMMENT REMOVED

\input{tex/0_title.tex}

% COMMENT REMOVED
% COMMENT REMOVED
% COMMENT REMOVED
% COMMENT REMOVED

\section{Introduction}
\input{tex/intro.tex}\label{sec:intro}

\section{General Problem Statement}
\input{tex/statement.tex}\label{sec:statement}

\section{Mode Differences and Asymptotic Approximations}\label{sec:asymptotic}
% COMMENT REMOVED

\input{tex/asymptotic.tex}

\section{Numerical Method}\label{sec:Numeric}
\input{tex/algorithm.tex}

\section{Experimental Results} \label{sec:experiments}
\input{tex/experiments.tex}

\section{More Wind Directions and the Limiting Case}\label{sec:BM}
\input{tex/BM.tex}

\section{Conclusions}\label{sec:conclusions}
\input{tex/conclude.tex}

\begin{appendices}

% COMMENT REMOVED
% COMMENT REMOVED

\section{Upper Bound for Mode Differences}\label{append:upperbound}
\input{tex/appendixB.tex}

\end{appendices}

% COMMENT REMOVED

\bibliographystyle{plain}
\bibliography{tex/Ref}

\end{document}

%% file: tex/0_preamble.tex
\usepackage{amssymb}
\usepackage{amsmath}
\usepackage{amscd}
\usepackage{cancel}
\usepackage{amsthm}
\usepackage{hyperref}
\usepackage{graphics}
\usepackage{graphicx}
\usepackage{MnSymbol}
\usepackage{verbatim}
\usepackage{epsfig}
\usepackage{url}
\usepackage{array}
\usepackage{mdwlist}
\usepackage{enumerate}
\usepackage{multicol}
\usepackage{changepage}
\usepackage{float}
\usepackage{epstopdf}
\usepackage{appendix}
\usepackage{float}
\usepackage{listings}
\usepackage{color}
\definecolor{mygreen}{rgb}{0,0.6,0}
\definecolor{mygray}{rgb}{0.5,0.5,0.5}
\definecolor{mymauve}{rgb}{0.58,0,0.82}

\usepackage{algorithmic}
\usepackage[ruled,lined,linesnumbered]{algorithm2e}

\def\barint_#1{\mathchoice
{\mathop{\vrule width 6pt height 3 pt depth -2.5pt
\kern -8.8pt \intop}\nolimits_{#1}}%
{\mathop{\vrule width 5pt height 3 pt depth -2.6pt
\kern -6.5pt \intop}\nolimits_{#1}}%
{\mathop{\vrule width 5pt height 3 pt depth -2.6pt
\kern -6pt \intop}\nolimits_{#1}}%
{\mathop{\vrule width 5pt height 3 pt depth -2.6pt
\kern -6pt \intop}\nolimits_{#1}}}

\usepackage{pgf}
\usepackage{tikz}
\usepackage{pgfcore}
\usetikzlibrary{arrows,patterns,plotmarks,shapes,snakes,er,3d,automata,backgrounds,topaths,trees,petri,mindmap}
\usepackage{verbatim} % this is just for the multi line comments

\numberwithin{equation}{section}

\newcommand*{\CopyCounter}[2]{
  \expandafter\def\csname c@#2\endcsname{\csname c@#1\endcsname}
  \expandafter\def\csname p@#2\endcsname{\csname p@#1\endcsname}
  \expandafter\def\csname the#2\endcsname{\csname the#1\endcsname}}

\newtheorem{theorem}{Theorem}[section]

\newtheorem{assumption}{Assumption}[section]

\numberwithin{Theorem}{section}
\CopyCounter{Theorem}{Proposition}
\CopyCounter{Theorem}{ProposedProblem}
\CopyCounter{Theorem}{Property}
\CopyCounter{Theorem}{Claim}
\CopyCounter{Theorem}{Lemma}
\CopyCounter{Theorem}{Corollary}
\CopyCounter{Theorem}{Conjecture}
\CopyCounter{Theorem}{Definition}
\CopyCounter{Theorem}{Example}
\CopyCounter{Theorem}{Remark}
\CopyCounter{Theorem}{Question}
\CopyCounter{Theorem}{Condition}
\CopyCounter{Theorem}{Criterion}
\CopyCounter{Theorem}{Observation}
\theoremstyle{plain}%% needs amsthm.sty

\theoremstyle{definition}%% needs amsthm.sty

\newtheorem{remark}[Remark]{Remark}

\newif\ifnever\neverfalse

\newcommand{\be}{\mathbf{e}}

\newcommand{\bw}{\mathbf{w}}

\newcommand{\ba}{\mathbf{a}}
\newcommand{\balpha}{\boldsymbol\alpha}

\newcommand{\D}{\mathcal{D}}
\newcommand{\I}{\mathcal{I}}

\newcommand{\x}{\mathbf{x}}
\newcommand{\xhat}{\hat{\x}}
\newcommand{\xbar}{\bar{x}}
\newcommand{\y}{\mathbf{y}}
\newcommand{\z}{\mathbf{z}}
\newcommand{\Bf}{\mathbf{f}}
\newcommand{\zhat}{\hat{\z}}
\newcommand{\w}{\mathbf{w}}
\newcommand{\X}{\mathcal{X}}

\newcommand{\Z}{\mathbf{Z}}
\newcommand{\xtilde}{\tilde{\mathbf{x}}}
\newcommand{\ort}{\text{ort}}
\newcommand{\M}{\mathcal{M}}

\newcommand{\f}{\textbf{f}}
\newcommand{\tildeU}{\tilde{U}}
\newcommand{\bC}{\bar{C}}
\newcommand{\tildeu}{\tilde{u}}

\newcommand{\GroupNote}[1]{%
}

\newcommand{\marginfix}{
\setlength{\parskip}{0.01cm}
\setlength{\textwidth}{6.0in}
\setlength{\oddsidemargin}{-0.0 in}
\setlength{\evensidemargin}{0.0 in}
\setlength{\topmargin}{-0.5in}
\setlength{\textheight}{9.0 in}
}

\marginfix

\usepackage{eqparbox,array}

\usepackage[square,comma,numbers]{natbib}

  \renewenvironment{thebibliography}[1]{%
    \begin{oldthebibliography}{#1}%
      \setlength{\parskip}{.3ex}%
      \setlength{\itemsep}{.3ex}%
  }%
  {%
    \end{oldthebibliography}%
  }

%% file: tex/0_title.tex
\vspace*{1cm}
\centerline{\Large\textbf{Piecewise-Deterministic Optimal Path-Planning}}

\vspace*{.1in}

% COMMENT REMOVED
\renewcommand*{\thefootnote}{\fnsymbol{footnote}}

{\Large
\centerline{Z. Shen\footnotemark[1] and A. Vladimirsky\footnotemark[1]
\footnotetext[1]{\sc Supported in part by the National Science Foundation grant DMS-1016150.}
}}

% COMMENT REMOVED
% COMMENT REMOVED
\renewcommand*{\thefootnote}{\arabic{footnote}}

% COMMENT REMOVED
{\large
\vspace*{.1in}
\centerline{Center for Applied Mathematics and Department of Mathematics}
\vspace{-.05cm}
\centerline{Cornell University, Ithaca, NY 14853}
}
% COMMENT REMOVED

% COMMENT REMOVED
% COMMENT REMOVED
% COMMENT REMOVED
% COMMENT REMOVED
% COMMENT REMOVED
% COMMENT REMOVED
% COMMENT REMOVED
% COMMENT REMOVED
% COMMENT REMOVED
% COMMENT REMOVED
% COMMENT REMOVED
% COMMENT REMOVED
% COMMENT REMOVED
% COMMENT REMOVED
% COMMENT REMOVED
% COMMENT REMOVED

\vspace*{.1in}
\begin{abstract}
\noindent
We consider piecewise-deterministic optimal control problems in which the environment randomly switches among several deterministic modes, and the goal is to optimize the expected cost up to the termination while taking the likelihood of future mode-switches into account.  The dynamic programming approach yields a weakly-coupled system of Hamilton-Jacobi-Bellman PDEs satisfied by the value functions of individual modes.    We derive and implement semi-Lagrangian and Eulerian numerical schemes for that system.  We further recover simpler, ``asymptotically optimal'' controls and test their performance in non-asymptotic regimes.  Our approach is illustrated on a simple anisotropic path-planning problem: the time-optimal control for a boat affected by randomly switching winds.
\end{abstract}

\vspace*{.1in}

%% file: tex/intro.tex
Dynamic programming is a popular approach for solving continuous optimal control problems with a moderate number of degrees of freedom $d$.  The value function is defined on a $d$-dimensional domain to encode the minimum-cost-till-termination for each initial configuration.  Bellman's optimality principle %\cite{}
allows to recover the value function by solving the corresponding Hamilton-Jacobi-Bellman (HJB) partial differential equation \cite{bardi2008optimal}.  If the controlled process is deterministic, that PDE involves only the gradient of the value function.  On the other hand, the most common setting for the stochastic optimal control is based on the assumption that the deterministic dynamics is perturbed by the additive Brownian motion, yielding second partial derivatives in the equation \cite{fleming2006controlled}.

Piecewise-deterministic controlled systems \cite{davis1986control, verms1985optimal} occupy an intermediate niche and are useful for encoding non-diffusive stochastic perturbations.  In this framework, the assumption is that both the cost and dynamics remain deterministic except for occasional transitions to different deterministic modes.  The value functions of different modes satisfy a weakly-coupled system of first-order HJB equations, which is more expensive computationally.
Nevertheless, this formulation is extensively used to model
the optimal control of manufacturing processes
(e.g., \cite{akella1986optimal, bielecki1988optimality, olsder1980time, sethi2012hierarchical}), production/maintenance planning \cite{boukas1990optimal},
economic growth \& global climate change modeling \cite{haurie2006stochastic},
as well as multi-generational games \cite{haurie2005multigenerational}.
More generally, they can be used to optimize the expected performance of any {\it stochastic switching system}, provided the total number of possible environment configurations is moderate, the environment changes are random in nature, and the planner becomes aware of the exact change as soon as it happens.  We believe that they can be particularly useful in path-planning problems, where describing stochastic perturbations by additive Brownian motion is often unsatisfactory \cite{yershov2013simplicial}.
Another related problem is the deterministic optimal control of randomly-terminated processes recently considered in \cite{andrews2013deterministic}.

For a concrete example, consider a generalization of the classical ``Zermelo's navigation problem'' \cite{zlobec2001zermelo}.
We will assume that, in the absence of winds, a rowboat can move with a (location-dependent) speed $s(\x)$.
In this case, the total time to the target $Q$ can be minimized by using the gradient descent in the value function $v$, defined by the Eikonal equation
$\|\nabla v\| s(\x) = 1$ with $v=0$ on $Q$.
If we now add wind which moves that boat with velocity $\w(\x)$, the boat dynamics becomes anisotropic.
For a unit vector $\ba$ encoding the direction in which we are currently rowing, the resulting boat velocity is $\Bf(\x,\ba) = s(\x) \ba + \w(\x)$, and the corresponding HJB equation is
$$\|\nabla v\| s(\x) - \w(\x) \cdot \nabla v = 1.$$
This equation can be solved by a variety of fast numerical methods for static HJB developed in the last 15 years \cite{sethian2001ordered, sethian2003ordered, alton2012ordered, mirebeau2014efficient, boue1999, kao2005fast, qian2007fast, bornemann2006, jeong2007interactive, gillberg2012new}.

But what if we have several different ``wind maps''  ( $\w_i(\x), \; i=1,\ldots,n$) commonly observed on this reservoir?
It is easy to define $\Bf_i$ for each $\w_i$ and then solve for each corresponding $v_i$ separately.  The relevant wind map $\w_i$ is selected when the boat starts moving, and the path is planned based on $v_i$. However, the wind direction might change before we reach the target.
An intuitive strategy is to react by switching to a new path (based on the new wind map) as soon as this change occurs.
But if we have statistics on how often such switches happen, this  can be used to improve our path planning.
Suppose $\lambda_{ij}  \geq 0$ characterizes the rate of switching from $\w_i$ to $\w_j$.
The new value function $u_i$ is defined as the minimal expected time to target if we start from $\x$
when the wind map $\w_i$ is in effect.  It is easy to show that such value functions must satisfy
 a {\it weakly-coupled} system:
\begin{equation}\label{eq:coupledEik}
\| \nabla u_i \| s(x) \; - \;
\w_i(\x) \cdot \nabla u_i \; = \; 1
\; - \;
\sum\limits_{j=1}^n \lambda_{ij} (u_i-u_j),
\qquad
i=1,\ldots,n.
\end{equation}
More general piecewise-deterministic systems are described in detail in section \ref{sec:statement} and the numerical methods for them are discussed in section \ref{sec:Numeric}.

Since solving a coupled system is more expensive,
the natural question is whether path-planning based on $v_i$'s would result in a significantly larger expected time of arrival.
We note that the controls based on $v_i$'s can be viewed as asymptotically optimal if all transition rates tend to zero.
A different approach to reduce the computational cost is based on a {\em single} value function $u^{\infty}$ describing the optimal behavior of a controller
who does not keep track of the current mode of dynamics \cite{yershov2013simplicial, bujorianu2004general}.  In section \ref{sec:asymptotic}, we explain why this strategy can be also viewed as ``asymptotically optimal'' when all rates of transition tend to infinity.  We then test the performance of both asymptotically optimal strategies with finite, positive
$\lambda_{ij}$ values.  In section \ref{sec:experiments}, we show that both of them result in a noticeable performance degradation, with the strategy based on $u^{\infty}$ yielding collisions with obstacles in a significant portion of Monte Carlo simulations.

% COMMENT REMOVED
Our example used for this benchmarking is particularly simple: a single rectangular obstacle and only two space-homogeneous wind-maps.  Yet, we show that the resulting dynamics based on the correct $u_i$'s has several subtle features.
E.g., it is well-known that in deterministic path planning, optimal trajectories are piecewise straight lines if the velocity is the same throughout the domain.
Our tests in section \ref{sec:experiments} show that this is no longer the case in a piecewise-deterministic setting.
In section \ref{sec:BM} we consider a different limiting case, with the number of modes/wind directions tending to infinity, and show the connection to degenerate elliptic HJB PDEs on $\mathbb{R}^{d+1}$.
We conclude by discussing several directions for future work in section \ref{sec:conclusions}.

% COMMENT REMOVED

% COMMENT REMOVED

% COMMENT REMOVED

%% file: tex/statement.tex
% COMMENT REMOVED
% COMMENT REMOVED

Consider the following feedback control process in a bounded domain $\Omega \subset \mathbb{R}^d$:
\begin{align}
\left\{
\begin{array}{ccl}
\y'(t)&=&\textbf{f}_{m(t)}\Big( \y(t), \, \balpha \big( \y(t), m(t) \big) \Big),\\
\y(0)&=&\x,\\
m(0)&=& i. \label{eq:ODE}
\end{array}
\right.
\end{align}
The full state of the process is specified by its continuous component $\y(t)$ and the current ``mode'' of the dynamics
$m(t) \in \M=\{1,\cdots,n\}.$
% COMMENT REMOVED
The continuous component changes according to a differential equation, while
the mode $m(t)$ randomly switches as a
continuous-time Markov process on $\M$.
The right hand side in the ODE for the continuous component depends on the current state $(\y, m)$ and the current control
$\balpha ( \y, m )$
chosen from
a compact set of control values %$A\subset \mathbb{R}^{d'}$.
$A.$
We note that $\y(t)$ is a random variable which also depends on the initial conditions $(\x, i)$, and the control $\balpha \in \mathcal{A},$
but we % COMMENT REMOVED
omit these dependencies from the notation
for the sake of readability.

As usual, the goal is to find an {\em optimal feedback control function} from the admissible
set\\ $\mathcal{A} = \left\{ \text{ measurable } \balpha: \Omega\times \M \mapsto A \right\}$, but before we specify the optimization criteria,
% COMMENT REMOVED
we need to describe the mode-switching process $m(t)$.
It can be characterized by the transition rate matrix $\Lambda= \left(\lambda_{ij}\right)$, where $\lambda_{ij}$ is the transition rate from mode $i$ to mode $j$:
\begin{align}\label{eq:transition_rate}
\lim_{\tau\to 0}\frac{Pr(m(\tau)=j \, | \, m(0) = i) \, - \, \delta_{ij}}{\tau} = \lambda_{ij},
\qquad \; \delta_{ij} = \left\{
\begin{array}{cc}
1,&i=j\\
0,&i\neq j
\end{array}
\right.
\end{align}

It is easy to verify that the transition rate matrix should satisfy
\begin{align}\label{eq:transition_condition}
\left.
\begin{array}{cl}
\lambda_{ij}\geq 0, &\qquad \forall j\neq i;\\
\lambda_{ii} = -\sum\limits_{j\neq i}\lambda_{ij}, &\qquad \forall i.
\end{array}
\right.
\end{align}

We also recall several standard properties of continuous-time Markov processes useful in our context:
\begin{itemize}
\item At any time $t$, conditional on $m(t)=i$, the probability that the system stays in current mode for at least time $s$ is
$\exp(\lambda_{ii}s) = \exp(-\sum_{j\neq i}\lambda_{ij}s)$.
\item If we define
$p_{ij}(t)=Pr\left(m(t)=j \, | \, m(0)=i\right)$, the evolution of the probability matrix $P(t)=\left(p_{ij}(t)\right)$ is then described by
\[\frac{d}{dt}P(t) = \Lambda P(t) = P(t) \Lambda, \qquad P(0)=I.\]
Therefore, $P(t)=\exp(\Lambda t)$.
\end{itemize}

Given a closed ``target set'' $Q\subset \bar{\Omega},$
we define the \emph{stopping time} $T=\min_{t\geq 0}\{t \, \mid \, \y(t)\in Q\}$.
The total (random) cost of starting from $(\x, i)$ and using the control $\balpha(\cdot)$ is defined as
\begin{align}
J\left(\x,i, \balpha \right) \; = \; \int_{0}^{T}C_{m(t)}\left(\y(t),\balpha(\y(t), m(t)) \right)dt \, + \, q(\y(T)),
\end{align}
where $C$ is the positive {\em running cost} defined on $\left( \Omega \backslash Q \right) \times \M \times A$
and $q$ is the {\em terminal cost} defined on $Q \cup \partial \Omega$.
To constrain the dynamics to $\Omega$, we impose a prohibitive terminal cost ($q = + \infty$) on $\partial \Omega \backslash Q$.
% COMMENT REMOVED
% COMMENT REMOVED
% COMMENT REMOVED
% COMMENT REMOVED
We can now define the {\em value function} as the minimum expected cost starting from %$z=(\x,i)$:
$\x$ with initial mode $i$:
\begin{align}
% COMMENT REMOVED
u(\x, i) = \inf_{\balpha \in \mathcal{A}} \mathbb{E} \left[ J \left( \x,i,\balpha \right) \right]
\end{align}
Throughout this paper we will use $u_i(\x)$ and $u(\x,i)$ notations interchangeably.
% COMMENT REMOVED

A standard formal argument can be used to derive the Hamilton-Jacobi-Bellman (HJB) equation for $u$.
The optimality of $\balpha\in\mathcal{A}$ yields
\begin{align}
u(\x,i)=\min_{\balpha\in\mathcal{A}} \left\{
\mathbb{E}\left[\int_0^{\tau} C_{m(t)}(\y(t), \balpha(\y(t),m(t)))dt\right]
\, + \,
\mathbb{E}\left[u(\y(\tau), m(\tau))\right]
\right\}
\end{align}
for $\tau$ less than the stopping time. As $\tau\to 0$, we have
\begin{align}\label{eq:CoupledDiscretize}
u(\x,i)=\min_{\textbf{a}\in A} \left\{ \tau C_{i}(\x,\ba) \, + \, \sum_{j=1}^n p_{ij}(\tau) u(\xtilde_{\ba},j) \right\} \, + \, o(\tau),
\end{align}
where $\xtilde_{\ba}=\x+\tau\f_{i}(\x,\ba)$.

The first-order approximations of $p_{ij}(\tau)$ and $u(\xtilde_{\ba}, i)$ are:
\begin{align*}
p_{ij}(\tau) \, = & \, 1-\exp(-\lambda_{ij}\tau) + o(\tau) \, = \, \lambda_{ij}\tau + o(\tau), \qquad \text{ if } j\neq i ,\\
p_{ii}(\tau)\, = & \, 1-\sum_{j\neq i} p_{ij}(\tau) \, = \, 1 - \sum_{j\neq i}\lambda_{ij}\tau + o(\tau),\\
u_i(\xtilde_{\ba}) \, = & u_i(\x) + \tau \f_i(\x,\ba)\cdot \nabla u_i(\x) + o(\tau).
\end{align*}
Using these approximations in \eqref{eq:CoupledDiscretize} we obtain
\begin{align}\label{eq:CoupledSimplify}
% COMMENT REMOVED
% COMMENT REMOVED
% COMMENT REMOVED
u_i(\x)=\min_{\textbf{a}\in A} \left\{ \tau C_{i}(\x,\ba) \, + \, \sum_{j \neq i} \lambda_{ij}\tau u_j(\x) \, + \, \left( u_i(\x) + \tau \f_i(\x,\ba)\cdot \nabla u_i(\x)
 \, - \, \sum_{j \neq i} \lambda_{ij} \tau  u_i(\x) \right)
\right\} \, + \, o(\tau).
\end{align}
Simplifying and letting $\tau \to 0$, we obtain the HJB equation:
\begin{equation}\label{coupledHJB}
\min_{\ba\in A}\{\nabla u_i(\x)\cdot \f_{i}(\x,\ba)+C_{i}(\x,\ba)\}-\sum\limits_{j\neq i}\lambda_{ij} (u_i(\x)-u_j(\x))=0,
\end{equation}
with $u_i(\x) = q(\x)$ on $Q \cup \partial \Omega$.
Since this has to hold for every $i \in \M$, it is more natural to think of this as a system of $n$ HJB PDEs.  We note that this system is {\em weakly coupled} as long as $\lambda$'s are non-zero.

The above derivation is formal since these PDEs typically don't admit classical solutions even for $n=1$. To obtain a derivation for the most general case, we need the concept of viscosity solution, which was first introduced in \cite{crandall1981viscosity}.
The viscosity solution of similar stochastic hybrid control problems is discussed in
\cite{bensoussan2000stochastic}.
% COMMENT REMOVED

% COMMENT REMOVED
% COMMENT REMOVED
% COMMENT REMOVED
% COMMENT REMOVED

% COMMENT REMOVED
% COMMENT REMOVED
% COMMENT REMOVED
% COMMENT REMOVED
% COMMENT REMOVED
% COMMENT REMOVED
% COMMENT REMOVED

%% file: tex/asymptotic.tex
The coupling in PDEs \eqref{coupledHJB} presents significant challenges for the numerical computation of the value function and implementation of optimal controls.  The mode-switching results in the interdependence of value functions for specific modes, increasing the number of iterations needed to solve the discretization of these PDEs.
To alleviate these difficulties, it can be useful to consider simplified/asymptotic versions of the above system.  If all rates of switching approach zero, the system becomes uncoupled (with many efficient numerical methods available).  Alternatively,
we also consider the infinite-transition-rate case, in which the switches happen so often that it is not necessary (or is no longer possible in practice) to keep track of the current mode.

Both of these asymptotic approximations can be used instead of the correct value function $u$ (based on the actual switching rates matrix $\Lambda$)
to simplify the computation of optimal controls.
The performance implications of this approach are explored in the simulations of section \ref{sec:experiments}.

\subsection{Uncoupled Planning}\label{ss:uncoupled_plan}

If there is no switching between different modes (i.e., $\Lambda=0$), the corresponding value function, defined as $u^0_i(\x)$,
must satisfy the system of HJBs
\begin{align}\label{uncoupledHJB}
\min_{\ba\in A}\{\nabla u^0_i(\x)\cdot \f_{i}(\x,\ba)+C_{i}(\x,\ba)\}=0.
\end{align}
Since this system is {\em uncoupled},  $u^0$ can be computed for each mode $i$ separately\footnote{An example of this
kind is the description of one-wind-direction-only value functions $v_i$ in the Introduction.}.
This is simply a collection of $n$ standard deterministic optimal control problems, for which many efficient numerical methods are already available.
Methods based on Fast Sweeping \cite{boue1999, zhao2005fast, qian2007fast, kao2005fast}, generalizations of Fast Marching \cite{sethian1996fast, sethian2001ordered, sethian2003ordered, alton2012ordered, mirebeau2014efficient}, hybrid two-scale methods \cite{chacon2012fast, chacon2015parallel}, or other label-correcting-type techniques \cite{bornemann2006, bak2010some, jeong2007interactive, gillberg2012new}
might be advantageous depending on what is known about the running cost and the dynamics.
Crucially, many of these take advantage of the direction of characteristics, which are typically quite different for each mode $i$.
This is not a problem in the uncoupled case, but makes such techniques much less efficient when we are interested in a general $\Lambda$.

\subsection{Infinite-transition-rate Planning}\label{ss:Inf-Trans}
For the purposes of this section, we assume that the switching process is irreducible; i.e. it is possible to eventually get to any mode starting from any other mode.
Given the transition rate matrix $\Lambda$, it is logical to ask what portion of time the process spends in each mode.
This gives rise to an invariant distribution on $\M$.  We recall several of its standard properties
(e.g., see \textbf{Theorem 3.5.3} and \textbf{Theorem 3.6.2} in \cite{norris1998markov}):

\noindent
For an irreducible transition rate matrix $\Lambda$ and its probability matrix $P(t)$,
\vspace*{-2mm}
\begin{itemize}
\item[(1)] There exists a unique invariant distribution $\pi^* \in \mathbb{R}^n$ satisfying $\pi^*=\pi^* P(t)$ for all $t>0$;
\item[(2)] $\lim\limits_{t\to +\infty}||\pi P(t) - \pi^*||=0$ for all initial distributions $\pi.$
\end{itemize}
If transition rate matrix $\Lambda$ is replaced by $\Lambda_c = c\Lambda$ for some $c>0$,
the probability matrix becomes $P_c(t) = \exp(ct\Lambda)$.  It is interesting to study the {\em infinite transition rate limit}: the behavior of the corresponding value functions $u^c_i(\x)$'s as $c \to \infty$.
% COMMENT REMOVED
% COMMENT REMOVED
% COMMENT REMOVED
% COMMENT REMOVED
% COMMENT REMOVED
% COMMENT REMOVED
% COMMENT REMOVED
% COMMENT REMOVED
Since every probability distribution $\pi^c(t)$ will satisfy
$$
\pi^c(t) = \pi^c(0) P_c(t) = \pi^c(0) \exp(ct\Lambda) = \pi^c(0) P_1(ct) = \pi^1(ct), \qquad
\lim_{c\to \infty} \pi^c(t) = \lim_{t\to\infty}\pi^1(t) = \pi^*,
$$
in the limiting case the distribution is always $\pi^*$ at any $t> 0$ regardless of the initial mode.
% COMMENT REMOVED
% COMMENT REMOVED
% COMMENT REMOVED
% COMMENT REMOVED
% COMMENT REMOVED
% COMMENT REMOVED
% COMMENT REMOVED
% COMMENT REMOVED
% COMMENT REMOVED
% COMMENT REMOVED
% COMMENT REMOVED
% COMMENT REMOVED
Additional technical assumptions (made precise in Theorem \ref{thm:upperbound} in Appendix), yield an upper bound for mode differences:
\begin{equation}
\label{eq:upper_rate}
\|u_i^c(\x) - u_j^c(\x)\|_{\infty} = O(c^{-1}).
\end{equation}
% COMMENT REMOVED
% COMMENT REMOVED
% COMMENT REMOVED
% COMMENT REMOVED
% COMMENT REMOVED
% COMMENT REMOVED
% COMMENT REMOVED
% COMMENT REMOVED
% COMMENT REMOVED
Thus, in the infinite-transition-rate limit, the value functions in all modes will be the same, and
we can simply use $u^{\infty}(\x)$ instead of $u^{\infty}_i(\x)$ to represent the value function.
% COMMENT REMOVED
% COMMENT REMOVED
% COMMENT REMOVED
Additional controllability assumptions about $\f_i$'s will yield the local Lipschitz-continuity of all $u^c_i$'s
on $\Omega \backslash Q$.  An argument similar to the one in \cite{evans1989perturbed, barron2008infinity}
can be then used to prove that their convergence to $u^{\infty}$ is uniform.
The latter function can be recovered as the viscosity solution of the HJB PDE
% COMMENT REMOVED
\begin{equation}\label{infHJB}
\min_{\ba\in A}\{\nabla u^{\infty}(\x)\cdot \mathbb{E}_{i\sim\pi^*}[ \f_{i} (\x,\ba)]+\mathbb{E}_{i\sim\pi^*} [C_{i}(\x,\ba)]\}=0,
\end{equation}
where
$\quad \mathbb{E}_{i\sim\pi^*}[ \f_{i} (\x,\ba)] = \sum_{i=1}^n \pi^*_{i}\f_{i} (\x,\ba)
\quad \text{ and } \quad
\mathbb{E}_{i\sim\pi^*}[ C_{i}(\x,\ba)] = \sum_{i=1}^n \pi^*_{i} C_{i} (\x,\ba)$.

Due to the irreducibility of the switching process\footnote{
For a more general case of possibly reducible switching processes, the value function does depend on the {\em communicating class} of the starting mode
% COMMENT REMOVED
\cite{norris1998markov}, and
a PDE similar to \eqref{infHJB} would have to be solved for each communicating class.
% COMMENT REMOVED
},
the choice of optimal controls in the limit becomes independent of the current mode $i$ and can be determined based on $\pi^*$.
We note that the optimal controls derived from \eqref{infHJB} can be applied even for finite $c$ as a way of planning for a {\em partially observable}
process.  (I.e., what is the best way to reach the target if we do not know the current mode $m(t)$?) This approach is advocated in several prior papers
\cite{yershov2013simplicial, bujorianu2004general}.
In section \ref{sec:experiments} we show that it has significant disadvantages for robotic path planning problems.

% COMMENT REMOVED
% COMMENT REMOVED
\subsection{Using Miscalculated Transition Rates}
\label{ss:approx}
Another interesting question is the expected total cost of path planning if we base it on ``approximately correct'' (rather than the real) transition rates.  This arises naturally if the rates we use come from the statistics accumulated in previous runs.   Alternatively, even if the real transition rate matrix $\Lambda^r = (\lambda_{ij}^r)$ is known explicitly, one might prefer to replace it with $\Lambda^p = (\lambda_{ij}^p)$ used for path planning purposes just for the sake of computational efficiency.
% COMMENT REMOVED
% COMMENT REMOVED
% COMMENT REMOVED
% COMMENT REMOVED
% COMMENT REMOVED
% COMMENT REMOVED
% COMMENT REMOVED
% COMMENT REMOVED
Assuming that $\Lambda^r$ and $\Lambda^p$ are related by simple scaling as in subsection \ref{ss:Inf-Trans}, we will use the corresponding $c$ values (including the limiting case $c = \infty$) in the superscript.
If $\balpha^p \in \mathcal{A}$ is the ``optimal'' control based on $\Lambda^p$, we define the expected cost $u^{r,p}$ of using that control when $\Lambda^r$ is in effect as
\begin{align}
u^{r, p}_i(\x) \; = \; \mathbb{E}[ J(\x, i, \balpha^p) \, | \, \Lambda^r].
\end{align}
The case $r=p$ corresponds to the equations in sections \ref{sec:statement}-\ref{ss:Inf-Trans}; otherwise,
$u^{r,p}$ can be recovered by solving a coupled system of linear HJB PDEs
\begin{align}\label{eq:urp}
&\nabla u^{r,p}_i(\x) \cdot \f_i(\x, \balpha^p(\x, i)) - \sum_{j \neq i} \lambda_{ij}^{r} (u^{r, p}_i(\x) - u^{r,p}_j(\x)) + C_i(\x, \balpha^p(\x, i)) = 0,
\end{align}
where $ \balpha^p(\x, i) = \arg\min\limits_{\ba\in A} \{ \nabla u^{p,p}_i(\x) \cdot \f_i(\x, \ba) + C_i(\x, \ba) \}$ are assumed to be already known.

% COMMENT REMOVED

% COMMENT REMOVED
% COMMENT REMOVED
% COMMENT REMOVED
% COMMENT REMOVED
% COMMENT REMOVED
% COMMENT REMOVED
% COMMENT REMOVED
% COMMENT REMOVED
% COMMENT REMOVED
% COMMENT REMOVED
% COMMENT REMOVED
% COMMENT REMOVED
% COMMENT REMOVED
% COMMENT REMOVED
% COMMENT REMOVED
% COMMENT REMOVED
% COMMENT REMOVED
% COMMENT REMOVED
% COMMENT REMOVED
% COMMENT REMOVED

%% file: tex/algorithm.tex
To simplify the notation, we will focus on the time-optimal control problems (i.e., $C_{i}(\x,\ba) = 1$ for all $\x, \ba,$ and $i$).
The numerical solution is sought on a Cartesian grid with grid spacing $h$ imposed over $\Omega$.
We will use $\X$ to denote the set of all gridpoints.
% COMMENT REMOVED
% COMMENT REMOVED
% COMMENT REMOVED
A separate copy of the grid is used for each mode $i \in \M$.
The PDE in a continuous domain is replaced by a coupled
system of discretized equations (one for each $\z=(\x,i)$).  So, the approximation
method consists of two components:
\begin{enumerate}
\item
a formula for approximating each individual $u(\x, i)$ if the value function for each neighboring gridpoint in $\X \times \M$ is already known;\\
\item
an algorithm for solving the entire coupled system.
\end{enumerate}
We start with the latter, which is largely ``discretization-neutral'', and postpone the former until subsection \ref{ss:discretization}.
The following notations will be used throughout the section:
\begin{center}
\begin{tabular}{| p{2.5cm} | p{10cm} |}
  \hline

  $h$ & the gridline spacing;\\
  \hline
  $\X$ & the set of all gridpoints in $\Omega$\\
  \hline
  $\mathcal{Z}$ & the set of all gridpoints in $\Omega\times\M$\\
  \hline
  $X(\z)$& the projection onto $\X$ for all $\z\in \mathcal{Z}$\\
  \hline
  $M(\z)$& the projection onto $\M$ for all $\z\in \mathcal{Z}$\\
  \hline
  $U(\z)$& the approximate value function at $\z$\\
  \hline
  $\mathcal{I}(\z)$ &
   the (discretization-dependent) set of ``possibly influenced'' gridpoints;
   i.e., $\mathcal{I}(\z) = \left\{ \zhat \in \mathcal{Z} \, | \,
   U(\zhat) \text{ might depend on } U(\z) \right\}$\\
  \hline
  $\mathcal{D}(\z)$ &
   the (discretization-dependent) set of ``possibly influencing'' gridpoints;
   i.e., $\mathcal{D}(\z) = \left\{ \zhat \in \mathcal{Z} \, | \, \z \in \I(\zhat) \right\}$\\
  \hline
  $\mathcal{Q}$ & the discretized target set; i.e., $\mathcal{Q} = \X \cap Q$\\
  \hline
  $q(\x)$ & the boundary condition for $\x\in \X$\\
  \hline
  $\mathcal{N}(\x)$ & the set of neighboring gridpoints; i.e., \newline
  $\mathcal{N}(\x) \, = \, \left\{ \xbar \in \X
  \, | \,  \| \x - \xbar \| = h \right\}$\\
  \hline
  $ActiveFlag(\z)$ & a boolean flag indicating whether $u(\z)$ should be recomputed\\
  \hline
\end{tabular}
\end{center}
Whenever the correspondence $\z=(\x,i)$ is clear from the context, we will use both the $\z$-based and $(\x,i)$-based notations
interchangeably; e.g.,  $U(\x,i)=U(\z),$ and $\f_{i}(\x,\ba)=\f(\z,\ba).$
% COMMENT REMOVED

We will focus on discussing algorithms for the general weakly-coupled planning.  The problems described in sections \ref{ss:uncoupled_plan} and \ref{ss:Inf-Trans} can be handled similarly, but are also covered by more efficient non-iterative methods; e.g., \cite{sethian2003ordered, alton2012ordered, mirebeau2014efficient}.

\subsection{Value Iterations}

Our overall approach is iterative -- a Gauss-Seidel relaxation of the {\em value iterations method}.
This is essentially an extension of {\em fast sweeping methods}, originally introduced for single-mode problems \cite{boue1999, zhao2005fast}.
We alternate through a set of predefined ``geometric orderings'' to sweep through gridpoints in $\X$ (e.g., in 2D, the four sweeping directions are
from the south-west, from the south-east, from the north-east, and from the north-west).
% COMMENT REMOVED
In each sweep, we attempt to update the value at each gridpoint by solving the discretized HJB equations (see section %\ref{sec:asymptotic} and
\ref{ss:discretization});
whenever a specific
$\x \in \X$ is processed, we re-compute $U(\x, i)$ for all $i \in \M$.
The newly-computed value is only used if it is smaller than the previous version of $U(\x,i)$.  This relies on the {\em monotonicity} of our discretization
(discussed in section \ref{ss:discretization})
and the initialization $U(\x, i) = +\infty$ for all $\x \not \in \mathcal{Q}.$

Our implementation also relies on one further speed-up technique: {\em the active flags}, first used in \cite{bak2010some} for the Eikonal PDE, are employed to identify the gridpoints that might need to be updated in each sweep.  (There is no point in updating $\z$ if none of the values on $\D(\z)$ have changed since the last time $U(\z)$ was computed.)
% COMMENT REMOVED
% COMMENT REMOVED
% COMMENT REMOVED
% COMMENT REMOVED
% COMMENT REMOVED
% COMMENT REMOVED
% COMMENT REMOVED
% COMMENT REMOVED
Finally, the iterations are terminated once the maximal change in the value function observed in the last sweep
falls below the predefined tolerance parameter
$\varepsilon > 0$.

\begin{remark}\label{rem:loop_struct}
% COMMENT REMOVED
In the current loop structure (iterate over $\M$ in the inner loop), none of the marching-type methods (e.g., \cite{sethian1996fast, sethian2003ordered}) are directly applicable since characteristic directions are usually quite different in different modes and the value functions for different modes are coupled.
One could also use an alternative loop structure, iterating over $\M$ in the outer loop.
 This technique iteratively ``freezes'' the value function for all modes but one, with the latter then solved by the same techniques available for decoupled problems.
 All previously developed methods (marching, sweeping, hybrid, label-setting, etc) are available as solvers for this ``one mode updated at a time'' problem.  However, our computational experiments showed that this alternative loop structure usually requires a much larger number of iterations.
 This determined our implementation choices described in Algorithm \ref{agDyn}.
% COMMENT REMOVED
\end{remark}
% COMMENT REMOVED

\begin{center}\label{agDyn}
\begin{algorithm}[!ht]
% COMMENT REMOVED
\#\texttt{Initilization:}\\%\LONGCOMMENT{Initialization}\\
\For {$\z \in \mathcal{Z}$}
{
    $ActiveFlag(\z)\leftarrow$ \textbf{false}\\
}
\For {$\x\in \X$}
{
    \eIf {$\x \in \mathcal{Q}$}
    {
        \For {$i \in \M$}
        {
            $U(\x, i) \leftarrow q(\x)$\\
            \For {$\zhat \in \mathcal{I}(\x, i)$}
            {
            	$ActiveFlag(\zhat) \leftarrow$ \textbf{true}\\
            }
        }
% COMMENT REMOVED
    }
    {
        \For {$i \in \M$}
        {
            $U(\x, i) \leftarrow \infty$\\
        }
    }
}
\ \\
\#\texttt{Main Loop:}\\%\LONGCOMMENT{Main Loop}\\
 $converged \leftarrow$ \textbf{false}\\
\While {not $converged$}
{
    $max\_change \leftarrow 0$\\
    \For  {$\x\in \X \backslash \mathcal{Q}$ { enumerated in the current sweep order}}
    {
            \For {$i\in \M$}
            {
                $\z \leftarrow (\x,i)$\\
                \If {($ActiveFlag(\z) ==$ \textbf{true})}
                {
                    $ActiveFlag(\z) \leftarrow $ \textbf{false} \\	
                    $U^* \leftarrow \textbf{update}(\z)$  \hfill\#\ \texttt{subsection \ref{ss:discretization}}\\%\SHORTCOMMENT{subsection \ref{ss:discretization}}\\
                    $change\_z \leftarrow U(\z) - U^*$ \\
                    $max\_change \leftarrow \max\{change\_z, max\_change\}$\\
                    $U(\z) \leftarrow U^*$  \\
                    \For {$\zhat \in \mathcal{I}(\z)$}
                    {
                        \If {$X(\zhat) \not\in \mathcal{Q}$}
                        {
                    	   $ActiveFlag(\zhat) \leftarrow \textbf{true}$\\
                        }
                }
            }
   	 }
    	\eIf {$max\_change < \varepsilon$}
    	{
         	$converged \leftarrow$ \textbf{true}
   	 }
    	%\Else
    	{
        		switch to the next sweep order
        }
    }
}
% COMMENT REMOVED
\caption{Pseudocode for solving the dynamic programming equations on a grid.}
\end{algorithm}
\end{center}

\subsection{Discretizations on the Grid}\label{ss:discretization}
We discuss several first-order accurate update formulas used to compute $U(\z)$
if $U(\zhat)$ are already known for all $\zhat \in \D(\z).$
We start with the most general semi-Lagrangian discretization and then show how additional
assumptions about the dynamics can be leveraged to obtain more efficient algorithms.
We introduce the notation used throughout this section:
\begin{center}
\begin{tabular}{| p{2.5cm} | p{12cm} |}
  \hline
  $(\textbf{e}_1, \ldots, \textbf{e}_d)$ & the canonical basis in $\mathbb{R}^d$\\
  \hline
 ort $= (\epsilon_1, \cdots, \epsilon_d)$ & a vector encoding a particular orthant in $\mathbb{R}^d$:  $\; \epsilon_k = \pm 1, \quad k = 1, \cdots, d$\\
 \hline
 $\mathcal{O}$ & the set of all orthants in $\mathbb{R}^d$\\
 \hline
 $\xi = (\xi_1, \cdots, \xi_d)$ & a point in the unit $(d-1)$-simplex: $\; \sum_{k=1}^d \xi_k = 1$, $\xi_k \geq 0, \quad  k = 1, \cdots, d$ \\
 \hline
 $\Xi$ & the unit $(d-1)$-simplex\\
  \hline
\end{tabular}
\end{center}

% COMMENT REMOVED
% COMMENT REMOVED
% COMMENT REMOVED
% COMMENT REMOVED
% COMMENT REMOVED
% COMMENT REMOVED
% COMMENT REMOVED
% COMMENT REMOVED

% COMMENT REMOVED

\subsubsection{Semi-Lagrangian Method}\label{sss:semiLag}
Semi-Lagrangian discretizations are based on following the characteristic for a short time and then using
the interpolated value function at the resulting point.  A comprehensive overview for the fully deterministic case can be found in \cite{falcone2013semi}.

\noindent
The idea is to use the dynamic programming equation \eqref{eq:CoupledDiscretize} to approximate the value function:
\begin{align}\label{eq:semiLag}
U(\x, i) \approx \min_{\ba\in A}
\left\{
\tau + \sum_{j \in \M} {p}_{ij}(\tau) {U}(\tilde{\x}_{\ba, \tau}, j)
\right\},
\qquad
\tilde{\x}_{\ba, \tau} = \x + \tau \f_i(\x,\ba).
\end{align}
$U(\tilde{\x}_{\ba, \tau}, j)$ is computed by an interpolation procedure, since $\tilde{\x}_{\ba, \tau}$ is generally not a gridpoint.
Here, we can use a first order approximation for ${p}_{ij}(\tau)$.

% COMMENT REMOVED
% COMMENT REMOVED
% COMMENT REMOVED
% COMMENT REMOVED
% COMMENT REMOVED
% COMMENT REMOVED
% COMMENT REMOVED
% COMMENT REMOVED
% COMMENT REMOVED
% COMMENT REMOVED
% COMMENT REMOVED
% COMMENT REMOVED
% COMMENT REMOVED
% COMMENT REMOVED
% COMMENT REMOVED
% COMMENT REMOVED
% COMMENT REMOVED

The discretization also implicitly depends on the choice of the (pseudo-)timestep $\tau$.
% COMMENT REMOVED
One popular choice is to take $\tau>0$ to be a constant proportional to $h$ \cite{falcone1994minimum}.
% COMMENT REMOVED
% COMMENT REMOVED
 In this approach, it is common to approximate $U(\tilde{\x}_{\ba, \tau}, j)$ through bi-linear interpolation using all $2^d$ vertices of the grid cell containing $(\tilde{\x}_{\ba, \tau}, j)$.
% COMMENT REMOVED
% COMMENT REMOVED

Another choice is to use a control-dependent $\tau_{\ba} = h / \|\f_i(\x,\ba)\|$ to ensure that $\tilde{\x}_{\ba, \tau}$ lies in a ($d-1$)-dimensional simplex formed by the gridpoints adjacent to $\x$. In this case, we can interpolate ${U}(\tilde{\x}_{\ba, \tau}, j)$ linearly using the values at the vertices of that simplex.

% COMMENT REMOVED
% COMMENT REMOVED
% COMMENT REMOVED
% COMMENT REMOVED
% COMMENT REMOVED

% COMMENT REMOVED

In our implementation we use the latter approach, which has important computational advantages (discussed at the end of this section),
particularly for the {\em small-time controllable problems}, where the controller can move the system in every direction regardless of the current mode.
I.e., for the rest of this section we will assume that, given any $\forall \z \in \mathcal{Z}$ and any
$(\ort, \xi) = \big((\epsilon_1, \cdots, \epsilon_d), \, (\xi_1, \cdots, \xi_d)\big) \in \mathcal{O} \times \Xi$,
there exists a unique control $\ba = \ba(\z, \ort, \xi)\in A$ and $\tau = \tau(\z, \ort, \xi) > 0$ such that
\begin{align}\label{eq:depend_tau}
\tau \f(\z, \ba) \, = \, (h \epsilon_1 \xi _1, \cdots, h \epsilon_d \xi_d).
\end{align}

% COMMENT REMOVED
Starting from $\z\in\mathcal{Z}$ with the control $\ba=\ba(\z,\ort,\xi)$ determined by $(\ort,\xi)$, we can compute the corresponding timestep
\begin{equation}
\label{eq:tau_for_simplex}
\tau = \tau(\z,\ort,\xi) = \frac{\sqrt{\sum_{k=1}^d \xi_k^2}}{\| {\f}(\z,\ba)\|}h,
\end{equation}
and define the set of contributing neighbors
\[\x_k = \x+\epsilon_k h\textbf{e}_k, \qquad k = 1, \cdots, d.\]

\noindent
If $(\ort, \xi)$ encodes the optimal direction, the approximate value function at $\z$ is
% COMMENT REMOVED
\begin{equation}
\label{eq:1_simplex_update}
\tildeU(\z, \ort, \xi)=
\tau+\sum_{k=1}^d \xi_k \left( \, p_{ii}(\tau)  U(\x_k,i) \, + \, \sum_{j\neq i} p_{ij}(\tau)U(\x_k,j) \, \right).
\end{equation}

% COMMENT REMOVED
% COMMENT REMOVED
% COMMENT REMOVED
% COMMENT REMOVED
% COMMENT REMOVED

Therefore, instead of searching for the optimal control $\ba^*\in A$,
we can search for the optimal $(\ort, \xi)^*\in\mathcal{O}\times\Xi$.
The minimization is efficiently performed on an orthant-by-orthant basis; i.e.,
% COMMENT REMOVED
% COMMENT REMOVED
\begin{align}
\label{eq:multi_simplex_update}
U(\z)=\min_{(\ort,\xi)\in \mathcal{O}\times\Xi}\tildeU(\z,\ort,\xi).
\end{align}
This update strategy is summarized in Algorithm \ref{agSemiLag}.
Since $\xi_k$ and $p_{ij}(\tau)$ are always non-negative, (\ref{eq:1_simplex_update}-\ref{eq:multi_simplex_update}) ensure that $U(\z)$ is a monotone non-decreasing function of $U(\x_k,j)$ for all $k$ and $j$.
% COMMENT REMOVED
% COMMENT REMOVED
% COMMENT REMOVED
As a result, the updates in Algorithm \ref{agDyn} are also performed in a monotone fashion; i.e., all gridpoint values are decreasing.

% COMMENT REMOVED

\begin{center}
\begin{algorithm}[!ht]\label{agSemiLag}
\texttt{function} $U^*=$ \textbf{update$(\z)$}:\\
$U^* \leftarrow U(\z)$\\
\For {$\ort = (\varepsilon_1, \varepsilon_2, \cdots, \varepsilon_d)\in \mathcal{O}$}
{
    $U' \leftarrow \min_{\xi\in\Xi} \tildeU(\z, \ort, \xi)$ \\
    (see formula \eqref{eq:1_simplex_update})\\
    \ \\
    \If {$U' < U^*$}
    {
        $U^* \leftarrow U'$\\
    }
}
\Return $U^*$\\

\caption{A semi-Lagrangian update function with a control-dependent timestep.}
\end{algorithm}
\end{center}

% COMMENT REMOVED
% COMMENT REMOVED

\subsubsection{Eulerian Discretization}\label{sss:Eulerian}
In general, the minimization in \eqref{eq:1_simplex_update} has to be performed numerically,
but for special types of dynamics it might be possible to find an analytic formula for the optimal direction, and Eulerian numerical schemes become preferable.
To illustrate the latter, we will focus on the dynamics of the rowboat affected by changing winds (as described in section \ref{sec:intro}).  The boat velocity in mode $i \in \M$ is $\f_{i}(\x, \ba) \; = \; s(\x) \ba + \bw_{i}(\x)$,
where $\ba \in \mathbb{S}^1$ is our rowing direction, $s(\x)$ is the rowing speed,
and $\bw_i$ is the velocity component due to the $i$-th wind map.
% COMMENT REMOVED
% COMMENT REMOVED
% COMMENT REMOVED
% COMMENT REMOVED
% COMMENT REMOVED
% COMMENT REMOVED
In this case,
\begin{align}
\min_{\ba\in A}\{\nabla u(\x,i) \cdot \f_{i}(\x,\ba)\}
 \; = \;
 -s(\x) \|\nabla u(\x,i)\| + \nabla u(\x,i)\cdot \bw_{i}(\x).
\end{align}
Therefore, the HJB equation for $u(\x,i)$ becomes
\begin{align}
\label{eq:Eik_version}
s(\x) \|\nabla u(\x,i)\| \; = \;  \nabla u(\x,i)\cdot \bw_{i}(\x) - \sum_{j\neq i}\lambda_{ij}(u(\x,i)-u(\x,j)) + 1.
\end{align}
% COMMENT REMOVED
% COMMENT REMOVED
To simplify the notation, we will only describe the discretization for $d=2$.
% COMMENT REMOVED
Having chosen a particular quadrant $\ort = (\epsilon_1, \epsilon_2)$,
we can use $$ D^h U(\x,i) = \left[ \frac{U(\x+\epsilon_1 h\be_1, i)-U(\x, i)}{\epsilon_1 h}, \frac{U(\x+\epsilon_2 h\be_2, i)-U(\x, i)}{\epsilon_2 h} \right]^T$$
to approximate $\nabla u(\x, i)$.
% COMMENT REMOVED
Plugging in this approximation into \eqref{eq:Eik_version}, reduces the PDE to
\begin{equation} \label{eq:Eik_discr}
s^2(\x) \left\| D^h U(\x,i) \right\|^2
\quad = \quad
\left[
D^h U(\x,i) \cdot \bw_{i}(\x) \; - \;
\sum\limits_{j\neq i}\lambda_{ij} \left(U(\x,i)-U(\x,j) \right) \; + \; 1
\right]^2.
\end{equation}
Assuming the values of $U$ at all neighboring gridpoints in $\Z$ are known,
this is simply a quadratic equation for $U(\x,i)$.

Note that $\ba^* = \frac{-D^h U}{\|D^h U\|}$ approximates the optimal control value at $(\x, i)$.
If the quadratic equation \eqref{eq:Eik_discr} has real roots, we need to satisfy the additional {\em upwinding condition}, by choosing the smallest root for which
the velocity vector points from the same quadrant used to approximate $\nabla u(x,i)$; i.e.,
\begin{align}\label{eq:upwinding}
% COMMENT REMOVED
% COMMENT REMOVED
% COMMENT REMOVED
% COMMENT REMOVED
% COMMENT REMOVED
% COMMENT REMOVED
{\f}_i(\x,\ba^*)\cdot \epsilon_k\be_k\geq 0,
\qquad k = 1, 2.
\end{align}
A straightforward computation shows that, if the quadratic equation \eqref{eq:Eik_discr} has more than one real root, the smaller root is always inconsistent with the upwinding condition \eqref{eq:upwinding}, and only the larger of them is relevant.

If the quadratic equation does not have real roots, or the upwinding condition \eqref{eq:upwinding} is not satisfied,
we default to ``one-sided'' %semi-Lagrangian
updates along the directions $\epsilon_k\be_k$ for $k=1,2$:
\begin{align}\label{eq:onesided}
U(\x,i) \; = \; \min\limits_{k=1,2} \tildeU_k,
\qquad \quad \text{with } \quad
\tildeU_k \; = \;
\frac{ \tau_k + U( \x + h \epsilon_k \be_k, i ) +  \tau_k \sum\limits_{j\neq i} \lambda_{ij} U(\x,j) }
{1 + \tau_k \sum\limits_{j\neq i} \lambda_{ij} },
\end{align}
where $\tau_k=\frac{h}{\f_{i}(\x,\ba)}$ and $\ba\in A$ is such that $\f_{i}(\x,\ba)=c \epsilon_k \be_k$ for some $c>0$.
This approximation is derived by taking functions $u(\x,j)$ to be known for all $j \neq i$ and writing down the characteristic ODEs for $u(\x,i)$ with the assumption that the characteristic direction at $\x$ is $(\epsilon_k \be_k).$

Finally, since the entire discretization was described for one quadrant, we need to minimize the update over all $\ort \in \mathcal{O}$.
The complete update procedure is summarized in Algorithm \ref{agEulerian}.

 \begin{center}
\begin{algorithm}[!ht]\label{agEulerian}
\texttt{function} $U^*=$ \textbf{update$(\z)$}:\\
$U^*\leftarrow U(\z)$\\
\For {$\ort = (\epsilon_1, \epsilon_2)\in \mathcal{O}$}
{
    \eIf {\eqref{eq:Eik_discr} has real roots AND the larger root $\tildeU$ satisfies the upwinding condition \eqref{eq:upwinding}}
    {
        $U^* \leftarrow \min(U^*, \tildeU)$\\
    }
    %\Else
    {
    	\For {$k=1,2$}
    	{
	       	compute $\tildeU_k$ by formula \eqref{eq:onesided}\\
		$U^*\leftarrow \min(U^*, \tildeU_k$)\\
% COMMENT REMOVED
% COMMENT REMOVED
% COMMENT REMOVED
% COMMENT REMOVED
% COMMENT REMOVED
    	}
    }
}
\Return $U^*$\\

\caption{An Eulerian update function.}
\end{algorithm}
\end{center}
Even though the approach described here is fairly similar to what is done in the Eikonal/single-mode case, there are several important subtleties worth noting.
\begin{enumerate}
\item
This discretization is first order accurate in $h$, consistent, and monotone.
I.e., it is easy to verify that $U(\x,i)$ is a monotone non-decreasing function of all $U(\x,j)$ and
$U(\x+h \epsilon_k \be_k,i)$.  When $u(\x,j)$ is known for all $j \neq i$, the standard argument \cite{barles1991convergence} shows the convergence
of $U(\x,i)$ to the viscosity solution $u(\x,i).$  This is also the reason why the update is only used if it is below the previously computed value of $U(\x,i)$.
(Recall that all $U$ values outside of $\mathcal{Q}$ are initialized to $\infty$ in Algorithm \ref{agDyn}.)
\item
The upwinding condition \eqref{eq:upwinding} is quite different from the one imposed in the simple Eikonal case.  In particular,
$U(\x,i)$ generally {\em is not} larger than the neighboring values from the quadrant used to compute it.  This means that the Fast Marching Method \cite{sethian1996fast} is not directly applicable.  An attempt to extend it to this case has been presented in \cite{dahiya2013characteristic} using a two-sided update equivalent to \eqref{eq:Eik_discr}; however, the lack of causality prevents the convergence for a wide range of examples.  The causal properties can be restored by sufficiently extending the stencil (as in Ordered Upwind Methods) and the corresponding examples were presented in \cite{sethian2003ordered}.  Similar problems were also handled by Fast Sweeping in \cite{cristiani2015modeling}.
\item
An extension of \eqref{eq:Eik_discr} and \eqref{eq:upwinding} to higher dimensional problems is quite straightforward.
However, when the upwinding condition is not satisfied, it becomes necessary to solve a sequence of quadratic equations corresponding to characteristics lying in lower dimensional faces of the update orthant.
% COMMENT REMOVED
\end{enumerate}

% COMMENT REMOVED
% COMMENT REMOVED
% COMMENT REMOVED
% COMMENT REMOVED

% COMMENT REMOVED

 \subsubsection{Comparison of discretizations}\label{sss:comparisions}
 We finish this section by summarizing the properties of three discretizations discussed above:
(a) the semi-Lagrangian method with a constant timestep $\tau$, (b) the semi-Lagrangian method with a control-dependent $\tau$, and (c) the Eulerian method.  All of them are first-order accurate, but their computational costs are quite different.
% COMMENT REMOVED
% COMMENT REMOVED
% COMMENT REMOVED
% COMMENT REMOVED
% COMMENT REMOVED
% COMMENT REMOVED
% COMMENT REMOVED
\begin{itemize}
\item Method (a) works on the most general problems, but is more computationally expensive than (b) and (c)
because (1) the number of iterations required for convergence is large when $\tau$ is relatively small;
(2) a larger stencil makes it more difficult to impose the boundary conditions and treat exit sets with empty interior;
(3) the non-local influence/dependence sets (i.e., $\mathcal{I}(\z)$ and $\mathcal{D}(\z)$) make it far more difficult to implement ActiveFlags efficiently.
\item Method (b) is more computationally efficient and often more accurate %than method (a)
due to its stencil locality
($\mathcal{I}(\x,i)=\mathcal{D}(\x,i)=\mathcal{N}(\x) \times \M$).
However, it still requires numerical minimization and is only suitable for speed profiles containing the origin in the interior.
\item Method (c) is the most efficient of these three, but it is only applicable when the optimal direction can be found explicitly as a function of the gradient.
Its computational stencil is also local but slightly different;
i.e., $\mathcal{I}(\x,i)=\mathcal{D}(\x,i)=\left(\mathcal{N}(\x) \times \{i\}\right) \bigcup \left(\{\x\} \times \M\right)$.
An argument based on Kuhn-Tucker optimality conditions can be used to show its equivalence
to a related semi-Lagrangian scheme, %similar to \eqref{eq:1_simplex_update},
and its output is at most $O(\tau^2)$ different from
the numerical solution produced by the method (b).
We omit the proof for the sake of brevity; see the Appendix in \cite{andrews2013deterministic} for a similar approach.
\end{itemize}

%% file: tex/experiments.tex
Our experiments are based on a simple model problem described in section \ref{sec:intro}: a rowboat trying to reach its destination in the presence of  (randomly changing) wind.
We also compare the performance of optimal controls with those found under simplifying assumptions
(i.e., zero or infinite wind-switching rate).

\subsection{Experimental setting and parameter values}
% COMMENT REMOVED
% COMMENT REMOVED
We minimize the expected time for the boat
% COMMENT REMOVED
whose velocity is defined by
$$\f_i(\x, \ba) = s(\x)\ba + \w_i(\x), \quad \ba \in \mathbb{S}^1, \quad i = 1, \cdots, n,$$
with the assumption that $\|\w_i\| < s(\x)$, for all $(\x, i)$.
If the transition between the modes are possible, this results in a weak-coupling
and yields the system of equations \eqref{eq:Eik_version},
which we solve iteratively using the Eulerian discretization from section \ref{sss:Eulerian}
on a cartesian grid with $h = \frac{1}{320}$.
The iterations are terminated when the changes in grid values become lower than $\varepsilon=10^{-6}.$

% COMMENT REMOVED
% COMMENT REMOVED
% COMMENT REMOVED
% COMMENT REMOVED
% COMMENT REMOVED
% COMMENT REMOVED
% COMMENT REMOVED
% COMMENT REMOVED
% COMMENT REMOVED
% COMMENT REMOVED
% COMMENT REMOVED
% COMMENT REMOVED

The emphasis of our experiments is on the influence of mode-switching rates on the optimal controls.
To this end, we further simplify the dynamics by taking the constant boat-speed in still water ($s(\x) = 2$),
and two constant wind map versions:$$n=2, \quad \w_1(\x) = [1.5, 0]^T,  \quad \w_2(\x) = [-1.5, 0]^T.$$
We restrict the dynamics to a square domain with a rectangular obstacle
$$
\Omega \quad = \quad (0, 1) \times (0,1) \quad \backslash \quad [0.1, 0.85] \times [0.1, 0.15]
$$
and choose the target at the location $Q = \{(0.5, 0.05)\}.$  The boundary conditions are specified by
$q=0$ on $Q$ and $q=+\infty$ on $\partial \Omega \backslash Q$.

Finally, we also make the mode switching symmetric by taking $\lambda_{12} \; = \; \lambda_{21} \; = \; \lambda$
and study the change in the value function as we vary this $\lambda$.

\subsection{Value Functions}
Figure \ref{figValueFunctions} shows the value functions computed for $\lambda=0,1,10, 50$ and $\infty$
with a magenta disk used to indicate the target.
When $\lambda=0$, the problem corresponds to the uncoupled planning problem in \eqref{uncoupledHJB}. When $\lambda=\infty$, the problem becomes the infinite-transition-rate planning problem in \eqref{infHJB}.
The computational cost of our experiments is very dependent on the strength of coupling:
the number of iterations/sweeps needed up to convergence is 6, 19, 35, 87 and 6 for the respective $\lambda$ values listed above.
% COMMENT REMOVED
% COMMENT REMOVED
% COMMENT REMOVED
\input{Figures/Fig_ValueFunctions.tex}
% COMMENT REMOVED
In each subfigure there is a clearly visible ``shockline'' above the obstacle, where the gradient of the value function is undefined.  The optimal trajectory runs clockwise or counterclockwise around the obstacle depending on whether our starting position is to the right or to the left of the shockline.  The optimal direction is not unique on the shockline with both approaches yielding the same (expected) time to the target for these starting locations.
It is important to note that the locations of shocklines are mode and $\lambda-$dependent.
% COMMENT REMOVED
When $\lambda = 0$, the shocklines of both modes are very sharp and almost in the center of the figure.
As $\lambda$ increases, the shocklines move left or right (depending on the mode) away from the center and  become less pronounced.  (E.g., for $\lambda=10$ and most starting positions sufficiently far North of the obstacle, it appears to be optimal to row South, directly toward the obstacle, instead of aiming for one of its corners.   The decision on whether to go clockwise or counterclockwise is postponed until we are closer to the obstacle.)
But when $\lambda$ becomes even higher, the shocklines are sharp again and return to the center.

We also use these computations to verify that the value functions of different modes become more similar as $\lambda$ increases.  Figure \ref{fig:ModeDifference} is fully in agreement with the asymptotic rate \eqref{eq:upper_rate} and the upper bound \eqref{eq:upper_bound}.
% COMMENT REMOVED
\input{Figures/Fig_Diff_UpperBound.tex}
% COMMENT REMOVED
% COMMENT REMOVED
% COMMENT REMOVED
\subsection{Path Planning and Simulations}\label{ss:simulations}
% COMMENT REMOVED
\input{Figures/Fig_Traj_NoSwitch.tex}
% COMMENT REMOVED
\input{Figures/Fig_Traj_Switch.tex}

% COMMENT REMOVED
% COMMENT REMOVED
We perform a series of Monte-Carlo trials for trajectories starting from the same position/mode,
but using randomly generated sequences of mode-switching times.

Once the rate of switching $\lambda>0$ is fixed, we can compute the value function and
define the optimal feedback policy $\balpha_*(\x,i)$ by using the minimizing value $\ba$ in equation \eqref{coupledHJB}.
However, our current goal is to explore the performance changes due to the use of asymptotic approximations:
both $u^\infty(\x)$ and $u^0_i(\x)$ are cheaper to compute than $u_i(\x)$ and define corresponding
``optimal'' feedback policies ($\balpha_*^{\infty}$ and $\balpha_*^0$ respectively).  In the path planning simulations,
these policies can be used despite the fact that the random mode-switching times correspond to our specific $\lambda$.
Since we are discussing the discrepancy of ``assumed'' and ``real'' $\lambda$ values, it is natural to adopt the notation of section \ref{ss:approx}:
we will refer to $u$ as $u^r$, while $u^0$ and $u^{\infty}$ correspond to $u^p$.
Equation \eqref{eq:urp} can be used to recover the expected time-to-target $u^{r,p}$ resulting from
this approximate planning.
However, we prefer to conduct Monte-Carlo simulations instead, since they provide more information about the result:  approximating the actual time-of-arrival
distribution rather than just the expectation.

In our experiments, the starting point is always at ${\xhat} = (0.5,0.8)$, and
the trajectories are color coded based on
the current wind direction: black is for eastward while white is for westward.
Given the current position and mode, ${(\x, i)} \in \Omega \times \M$,
we use a constant control $\balpha_*({\x, i})$ over the next $\Delta t$ seconds.
(We have used $\Delta t=10^{-3}$, which is sufficiently small for the $\lambda$ values 1 and 10 considered in these tests.)
The position is shifted by
$\Delta t \f_i(\x, \balpha_*({\x, i}))$
and the mode possibly changes based on $\lambda$.
The process is repeated until we reach the target.
% COMMENT REMOVED
\input{Figures/Fig_Traj_data.tex}
% COMMENT REMOVED
% COMMENT REMOVED
Figure \ref{figSinglePathNoTrans} shows the trajectories without any mode switches based on the
$\balpha_*^0$ planning and $\balpha_*$ planning with $\lambda = 1$.
Since the Hamiltonians are homogeneous in $\x$,
the trajectories resulting from $\balpha_*^0$ are piecewise-straight lines.
For this particular starting location,
% COMMENT REMOVED
it is optimal to go clockwise around the obstacle even if the wind blows eastward -- since the counterclockwise alternative path is slightly longer (the obstacle is off center).
However, if we take the correct switching rate into account, use $\balpha_*$ planning and start with the wind blowing eastward, it makes sense to go counterclockwise.
The intuitive reason is that the wind is likely %expected
to switch toward West before we reach the target %with a positive probability,
making our motion on the final stretch much faster.
Interestingly, the trajectory we follow in anticipation of this switch is not piecewise straight.
This is entirely due to a weak coupling of the PDEs since the Hamiltonians remain the same.
However, this logic backfires in the particular scenario illustrated by Figure \ref{figSinglePathNoTrans}:
since the wind never changes, the actual time along the trajectory based on $\balpha_*$ with the wind blowing eastward is
$T^{\lambda} = 1.179$, which is larger than our expected $E[T^{\lambda}] = u^{\lambda}_2(\xhat) = 0.915$
and certainly larger than the no-switch optimal $u^{0}_2(\xhat) = 1.073$.
Luckily, this outcome is not the most likely: the probability of the wind direction remaining the same that long is
$\exp(-1.179) \approx 31\%.$
% COMMENT REMOVED
% COMMENT REMOVED
% COMMENT REMOVED
% COMMENT REMOVED

Of course, there can be many mode switches happening before we reach the target, particularly when $\lambda$ is larger.
Figure \ref{figSinglePath} illustrates the results of a single simulation with many mode switchings based on $\balpha_*$, $\balpha_*^{0}$, and $\balpha_*^{\infty}$ path planning.
Using $\lambda = 10$, we have generated a sequence of switching times: (\SwitchTime, ...).
With the correct feedback optimal controls $\balpha_*$, the target is reached at the time $T^{\lambda} = \TCoupled{}$.
With the uncoupled ``optimal'' control $\balpha_*^0$, two more mode switches occur by the new arrival time $T^0 = \TUncoupled{}$.
Finally, if we use $\balpha_*^{\infty}$,
the target is not reached at all since a collision with an obstacle happens at the time $t = \TInfTrans{}$.
This outcome is not uncommon when using the infinite-transition-rate planning despite the fact that
the real value of $\lambda$ is finite.
In our experiments based on $\balpha_*^{\infty}$,  collisions occurred in $22.5\%$ of simulations for $\lambda=1$
and in $42.5\%$ of simulations for $\lambda=10$.

\input{Figures/Fig_Simulation_Low.tex}
% COMMENT REMOVED
% COMMENT REMOVED
\input{Figures/Fig_Simulation_High.tex}
% COMMENT REMOVED
% COMMENT REMOVED
While Figure \ref{figSinglePath} is based on one particular sequence of switching times,
Figures \ref{figCostLow} and \ref{figCostHigh} show the time-of-arrival statistics
gathered in 200 Monte-Carlo simulations
% COMMENT REMOVED
with $\lambda = 1$ and $\lambda = 10$ respectively.
% COMMENT REMOVED
% COMMENT REMOVED
% COMMENT REMOVED
% COMMENT REMOVED
% COMMENT REMOVED
In Figure \ref{figCostLow} different markers are used to indicate the number of mode switches experienced in each simulation.
% COMMENT REMOVED
The green line shows the observed sample average, an approximation of $u^{r,p}$, and the red line corresponds to the planner's naive expectations (based on $u^0 ({\x, i})$ and $u^{\infty}({\x, i})$ respectively).
For reference, the purple line shows the correct attainable minimum average time based on $u^r(\z)=u^{\lambda}({\z})$.
% COMMENT REMOVED
% COMMENT REMOVED
% COMMENT REMOVED
Figure \ref{figCostHigh} presents the same information for $\lambda = 10$,
but uses the x-axis to represent the number of mode switches.
% COMMENT REMOVED
These scatter plots report the arrival times only for those $\balpha_*^{\infty}$-based simulations that reach the target safely.

% COMMENT REMOVED
We note that for the ``coupled planning'' our observed sample average is the lowest
and quite close to the optimal $u^{\lambda}$.
Since $m(t)$ is based on that $\lambda$ value, this is hardly surprising.  As expected, when $\lambda$ is low,
$u^{\lambda}$ is closer to $u^0$, and $u^{\infty}$ becomes a better approximation for larger $\lambda.$  But from a practical point of view, a much more relevant question is the relative difference between $u^{r,p}$ and $u^r$.
For example, $(u^{\lambda,0} - u^{\lambda})/u^{\lambda}$ can be used to measure the percentage of performance degradation due to using the uncoupled planning.  This quantity is only 1\% when $\lambda=1$, but already
13.2\% when $\lambda = 10$.  (The corresponding numbers for the $\balpha_*^{\infty}$-based planning are
18\% and 8.7\%, but their practical relevance is lower since they are based on a subset of non-colliding trajectories.)
% COMMENT REMOVED
% COMMENT REMOVED
% COMMENT REMOVED
% COMMENT REMOVED
% COMMENT REMOVED
% COMMENT REMOVED
% COMMENT REMOVED
% COMMENT REMOVED
% COMMENT REMOVED
% COMMENT REMOVED
% COMMENT REMOVED
% COMMENT REMOVED

% COMMENT REMOVED
% COMMENT REMOVED
% COMMENT REMOVED
% COMMENT REMOVED
% COMMENT REMOVED
% COMMENT REMOVED
% COMMENT REMOVED
% COMMENT REMOVED
% COMMENT REMOVED
% COMMENT REMOVED
% COMMENT REMOVED

%% file: Figures/Fig_ValueFunctions.tex
\newcommand{\Figonetype}{png}
\def \FIGONETYPE {pdf}

\begin{figure}[!ht]
\centerline{
\includegraphics[width=0.25\textwidth, height=0.25\textwidth]{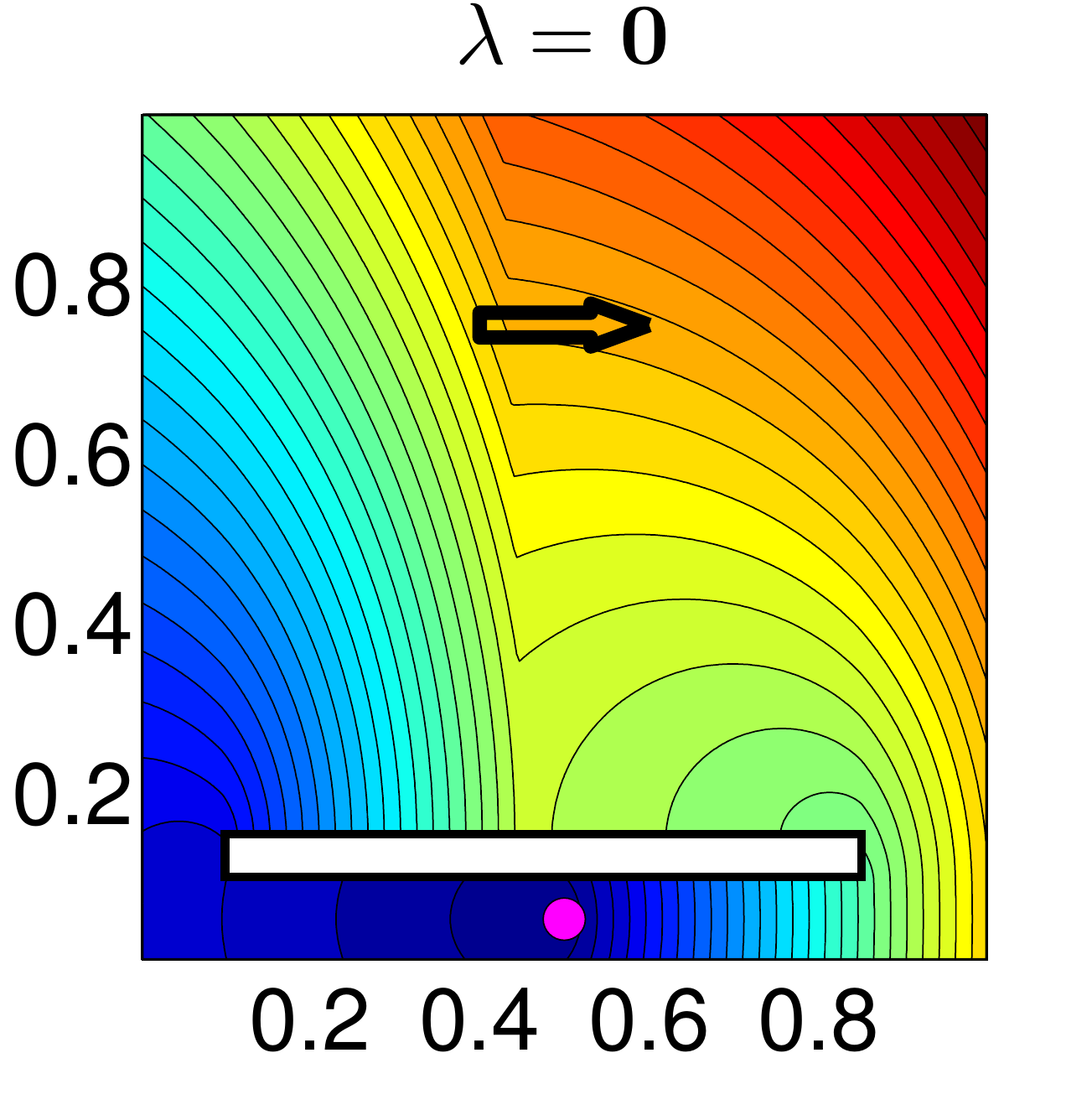}
\includegraphics[width=0.25\textwidth, height=0.25\textwidth]{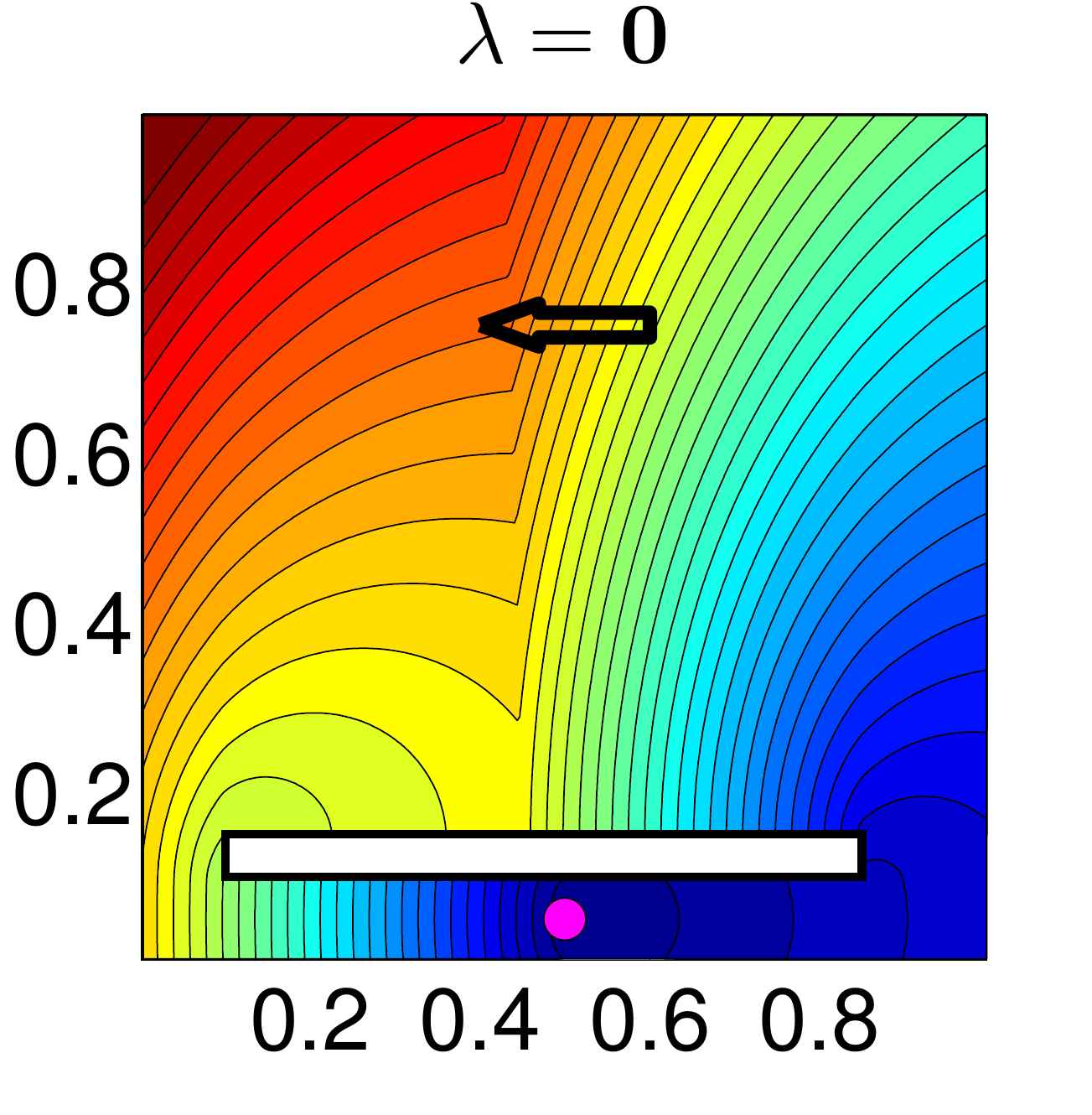}
\includegraphics[width=0.25\textwidth, height=0.25\textwidth]{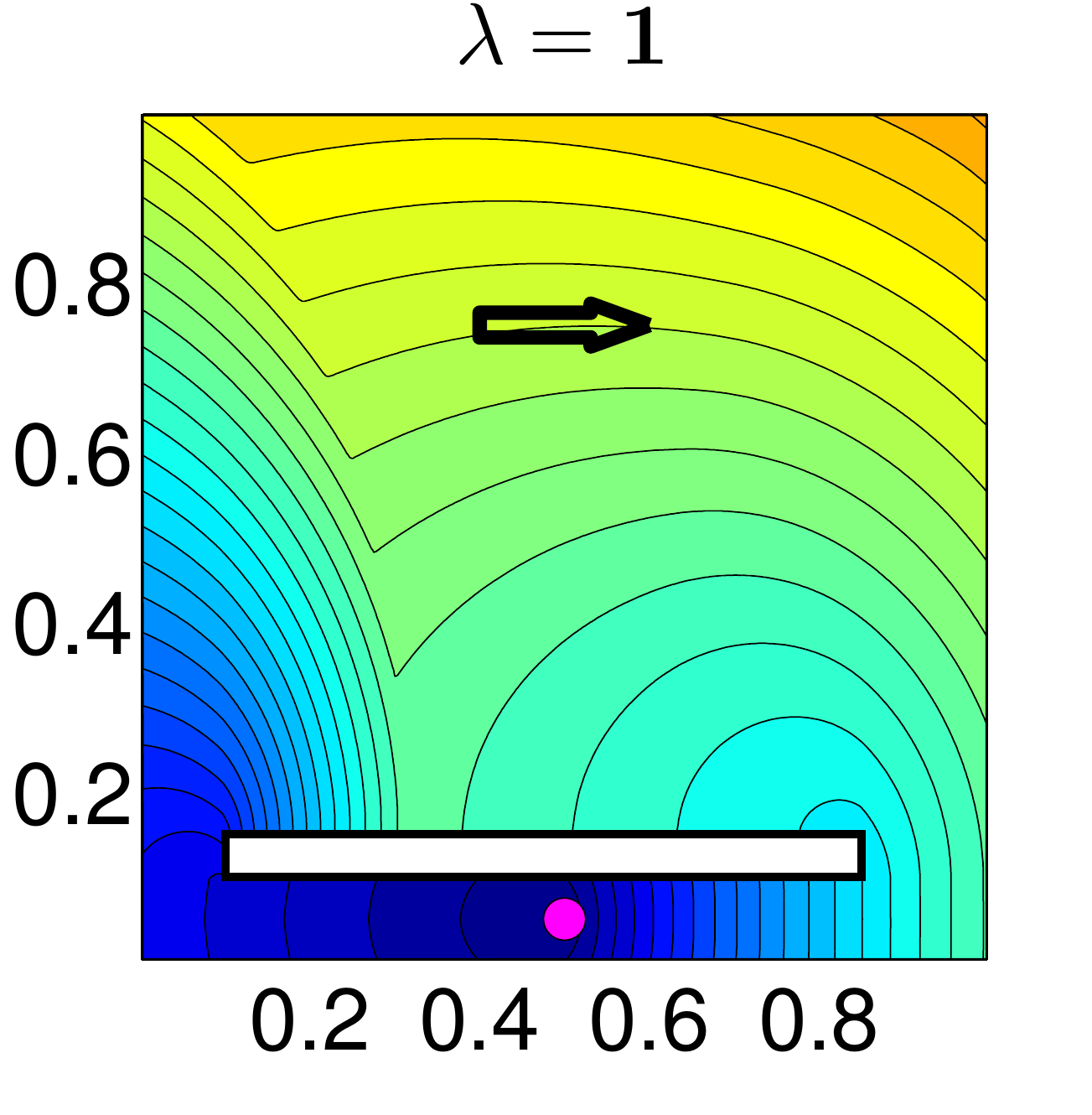}
\includegraphics[width=0.25\textwidth, height=0.25\textwidth]{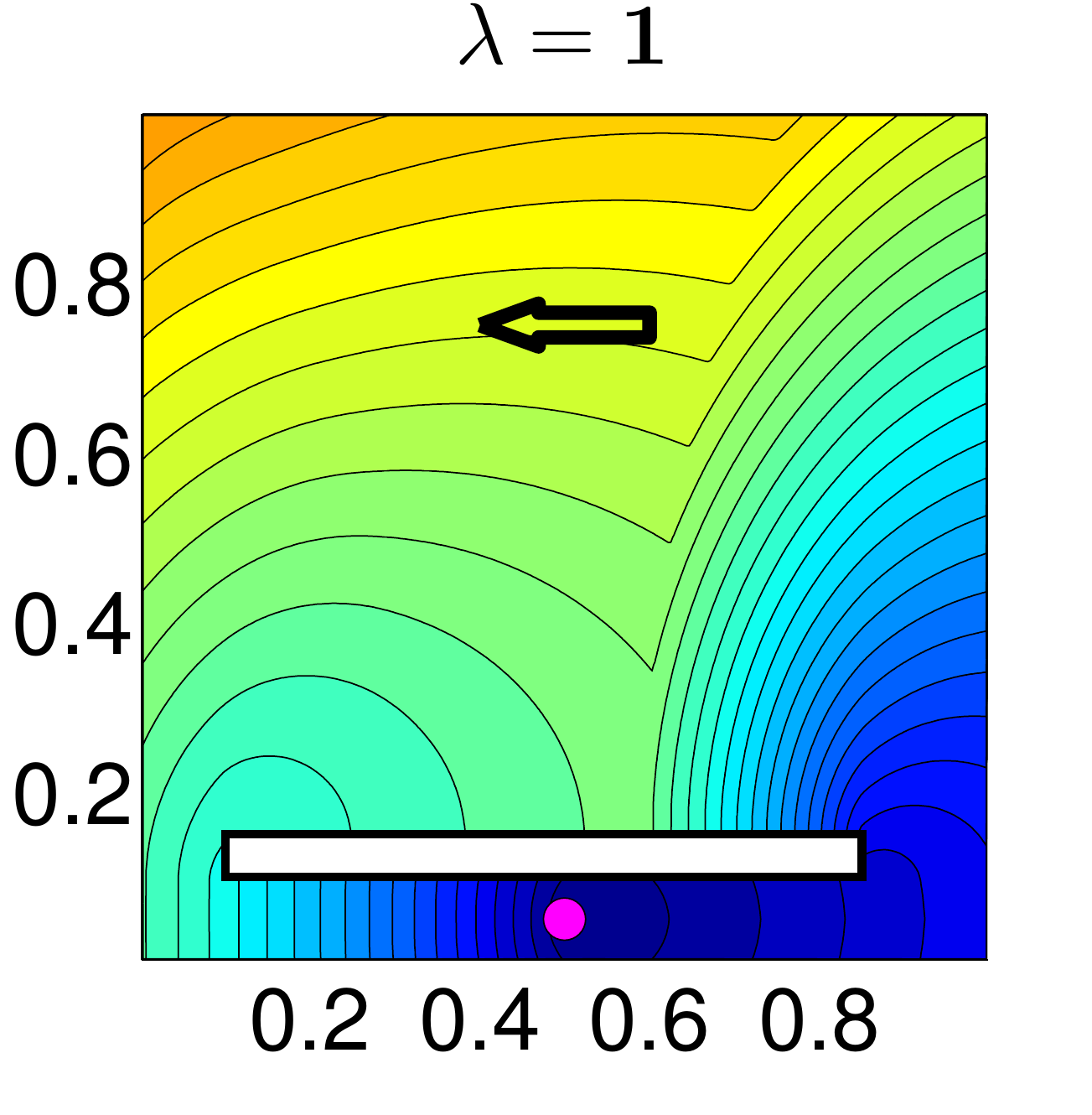}
}

\centerline{
\includegraphics[width=0.3\textwidth, height=0.01\textwidth]{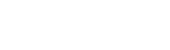}
}
\centerline{
\includegraphics[width=0.25\textwidth, height=0.25\textwidth]{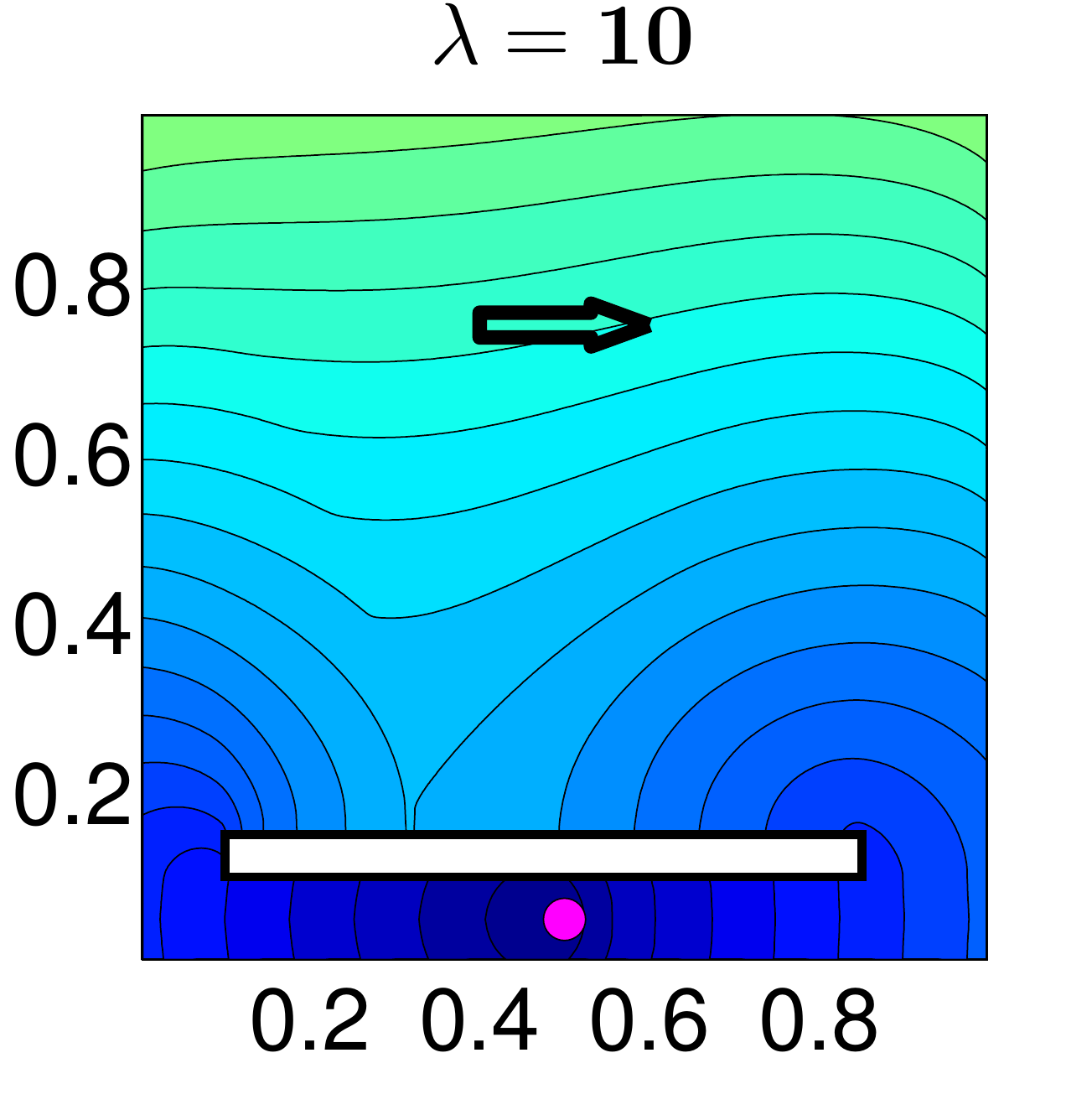}
\includegraphics[width=0.25\textwidth, height=0.25\textwidth]{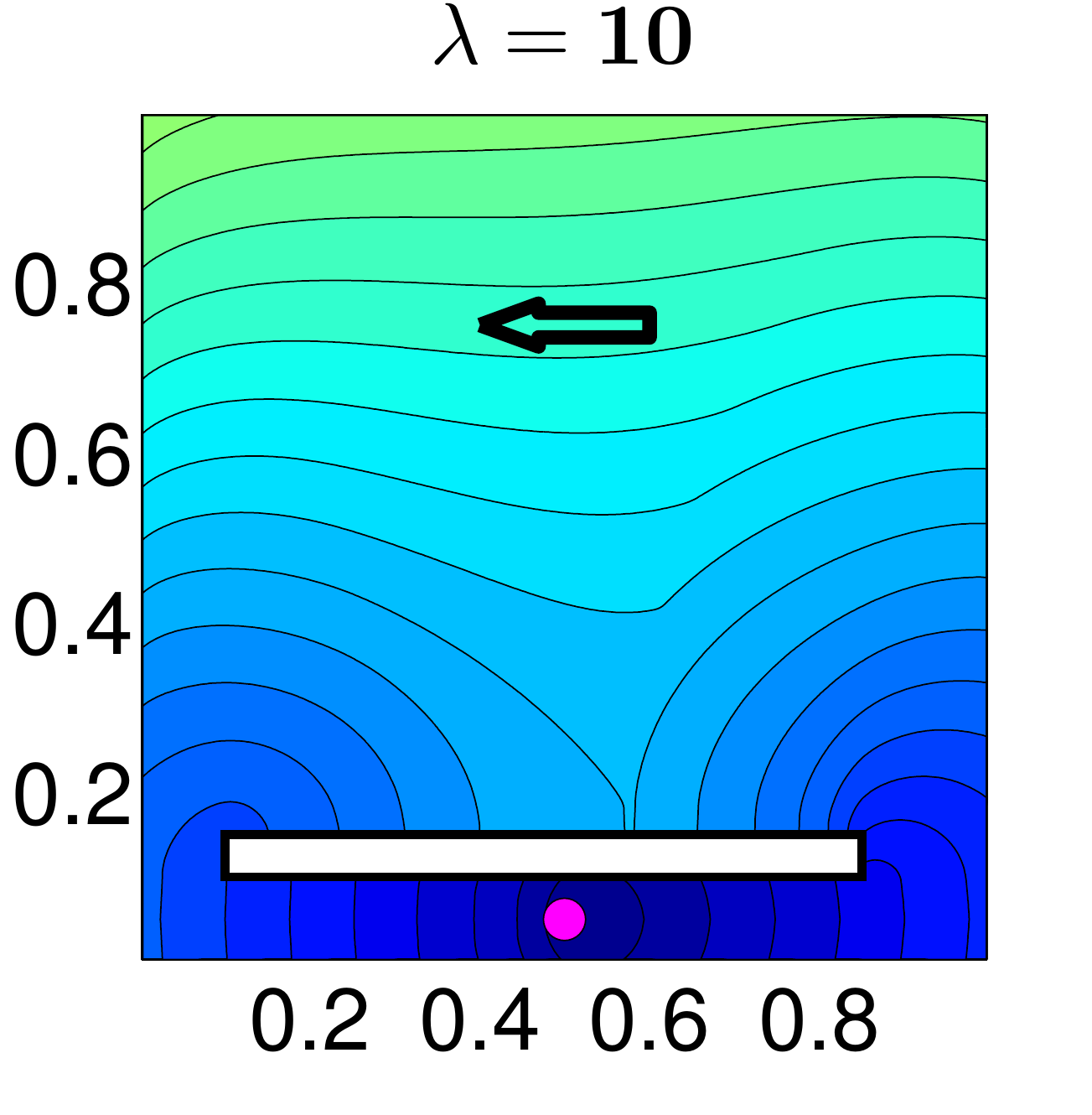}
\includegraphics[width=0.25\textwidth, height=0.25\textwidth]{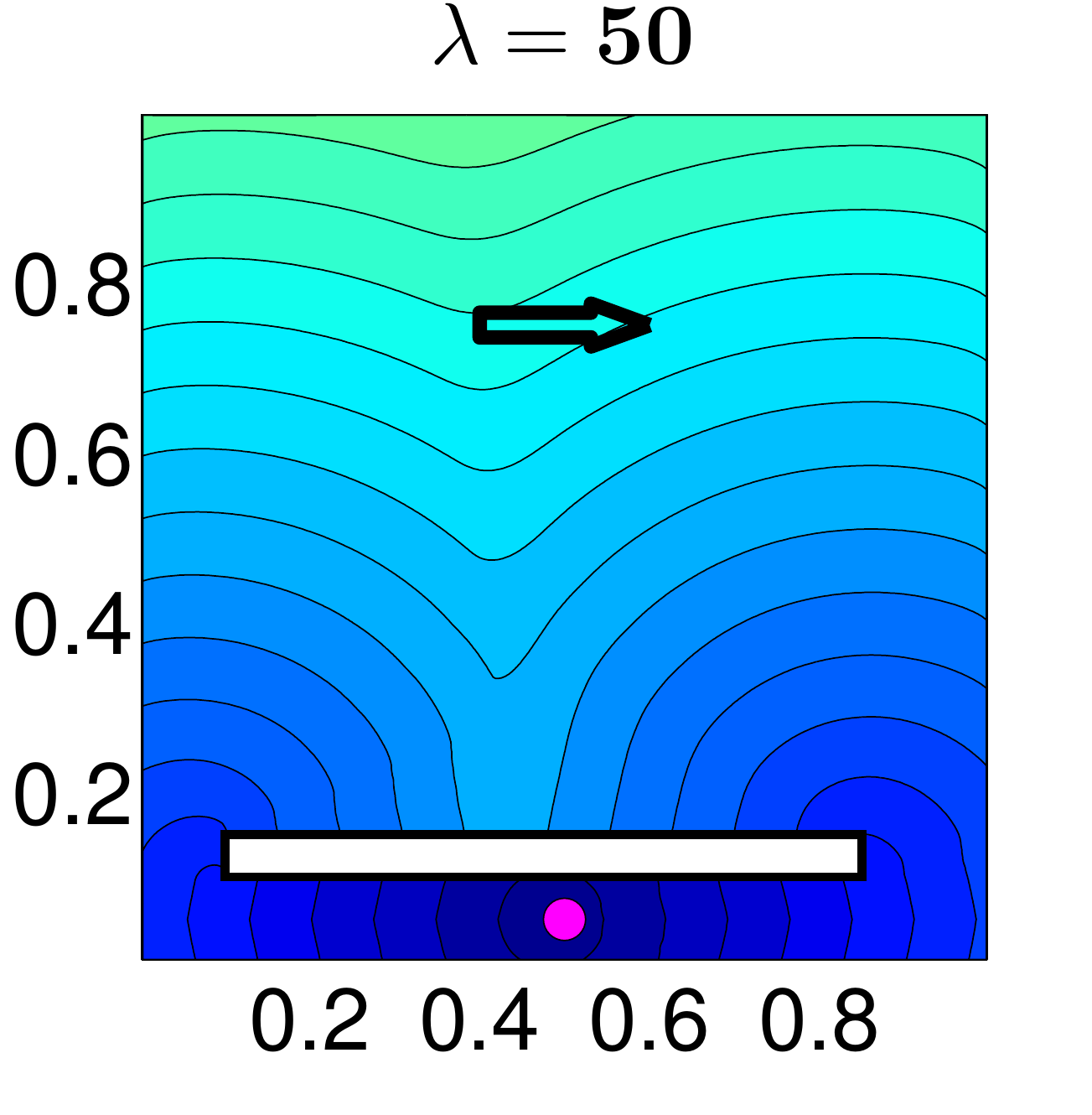}
\includegraphics[width=0.25\textwidth, height=0.25\textwidth]{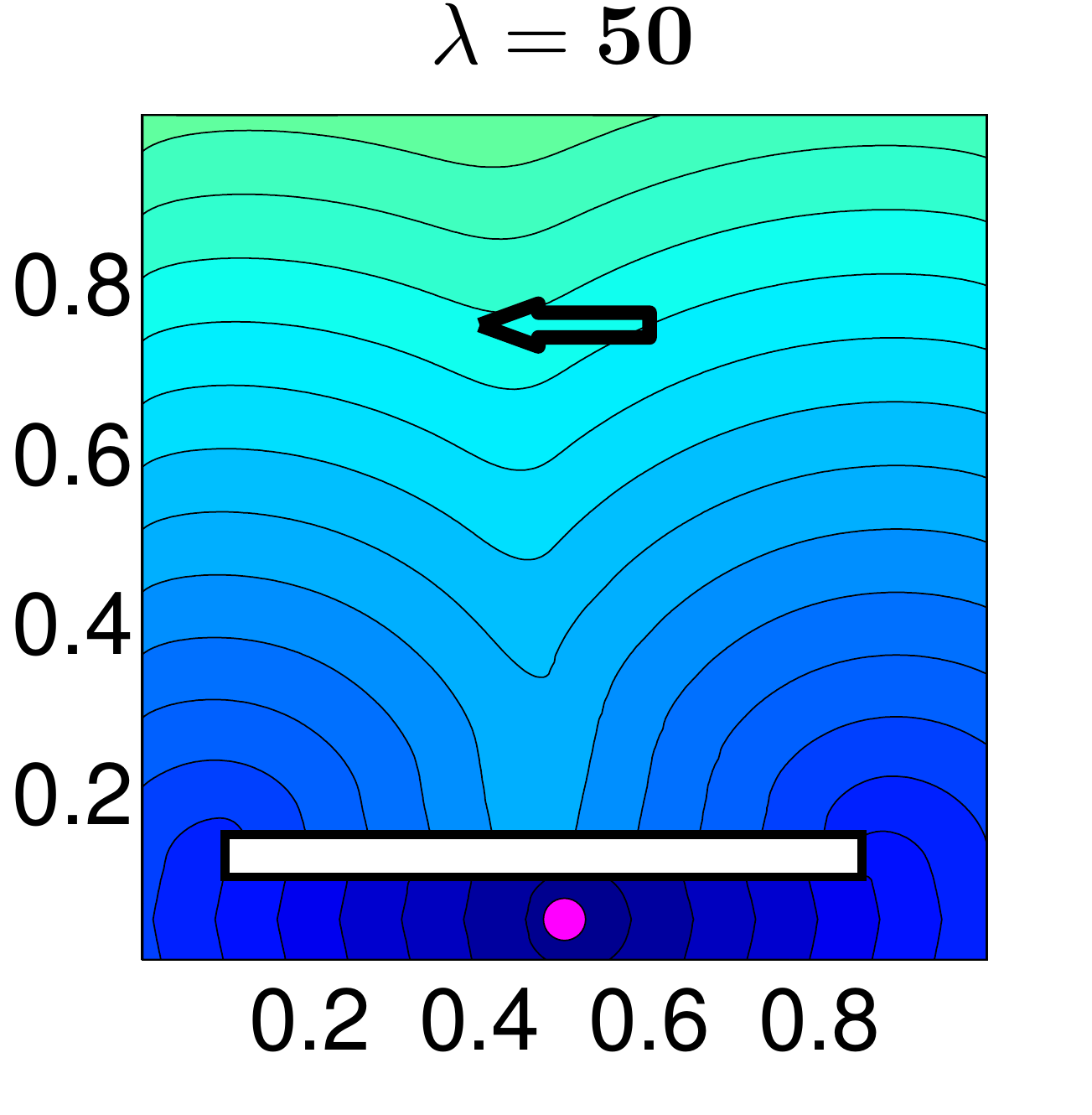}
}

\centerline{
\includegraphics[width=0.25\textwidth, height=0.01\textwidth]{Figures/blank.png}
}

\centerline{
\includegraphics[width=0.55\textwidth, height=0.2\textwidth]{Figures/colorbar.\FIGONETYPE}
\includegraphics[width=0.25\textwidth, height=0.25\textwidth]{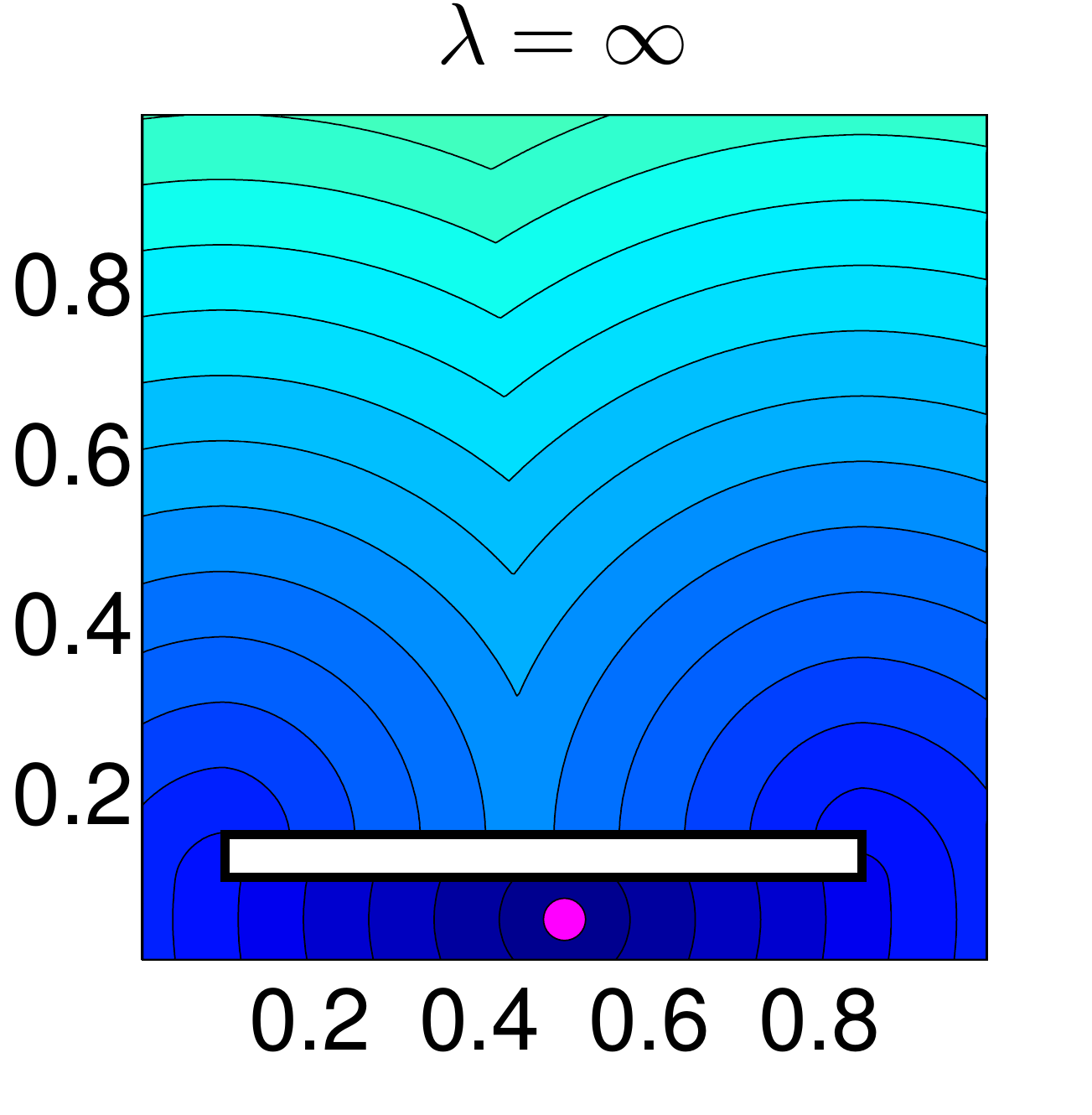}
}
\caption{
The value functions for $\lambda=0, \, \lambda=1, \, \lambda=10, \, \lambda = 50 \text{ and } \lambda=\infty$.
   Wind directions represented by arrows.
}
% COMMENT REMOVED
% COMMENT REMOVED
% COMMENT REMOVED
% COMMENT REMOVED
% COMMENT REMOVED
% COMMENT REMOVED
% COMMENT REMOVED
% COMMENT REMOVED
% COMMENT REMOVED
\label{figValueFunctions}
\end{figure} 

%% file: Figures/Fig_Diff_UpperBound.tex
\begin{figure}[!ht]
\centerline{
\includegraphics[width=0.5\textwidth, height=0.35\textwidth]{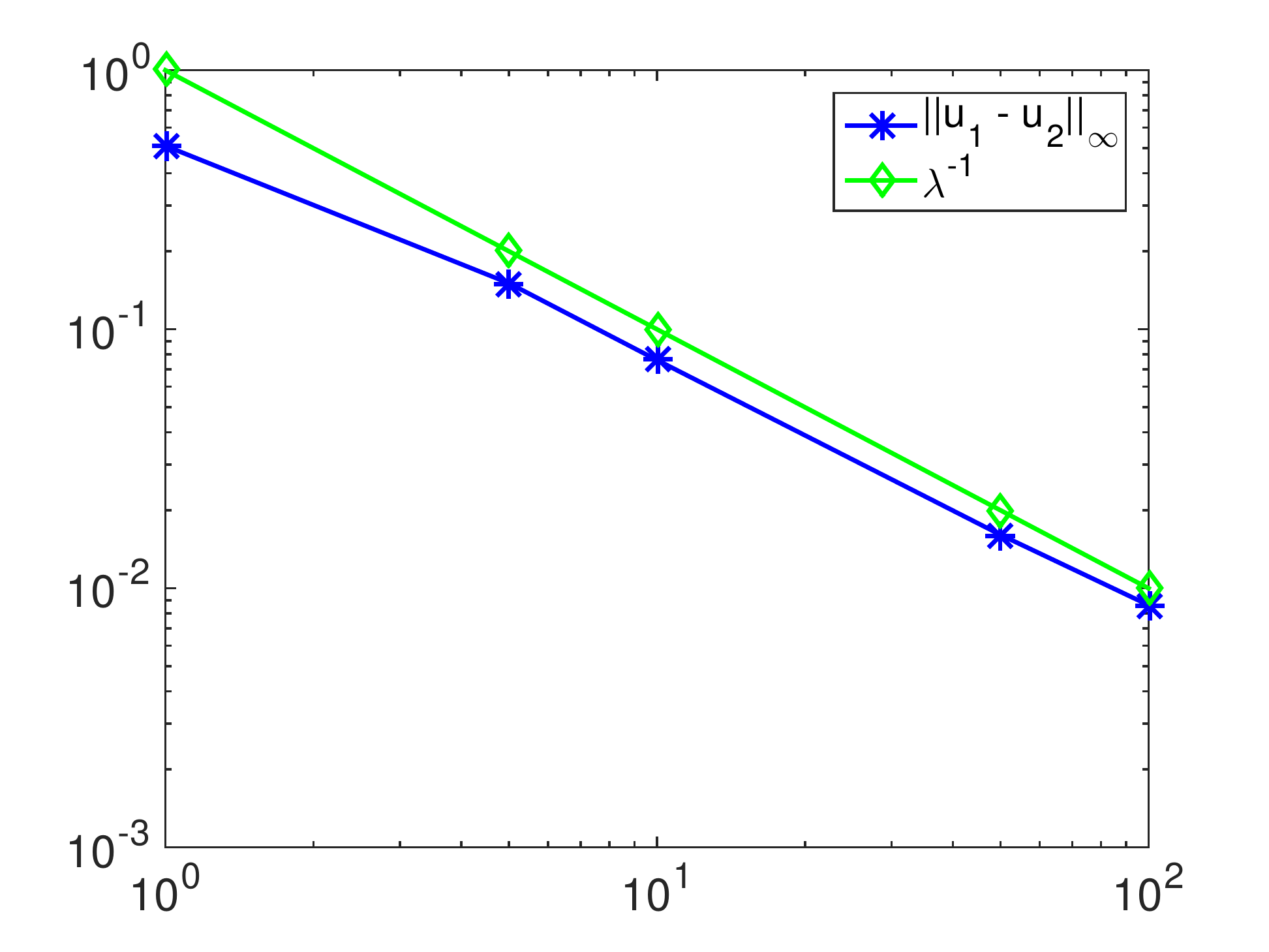}
}
\caption{
The mode differences $\|u_1 - u_2\|_{\infty}$ plotted for $\lambda = 1, 5, 10, 50, 100.$ 
The case $\lambda=0$ is not plotted, but $\|u^0_1 - u^0_2 \|_{\infty} \approx 0.8518.$
}
\label{fig:ModeDifference}
\end{figure}

%% file: Figures/Fig_Traj_NoSwitch.tex
\def \FIGTHREETWOTYPE{eps}
\begin{figure}[!ht]
\centerline{
\includegraphics[width=0.45\textwidth, height=0.33\textwidth]{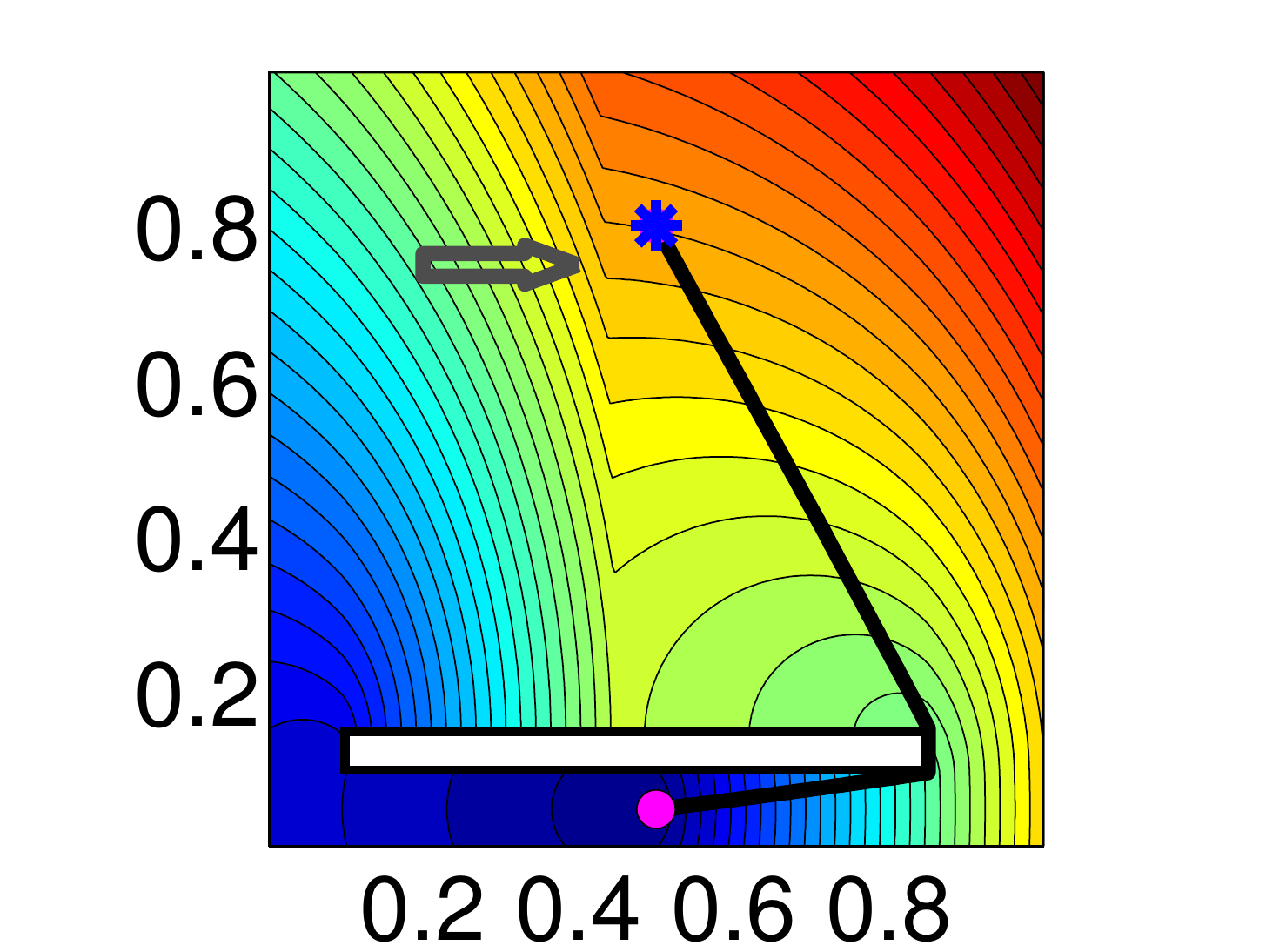}
\includegraphics[width=0.45\textwidth, height=0.33\textwidth]{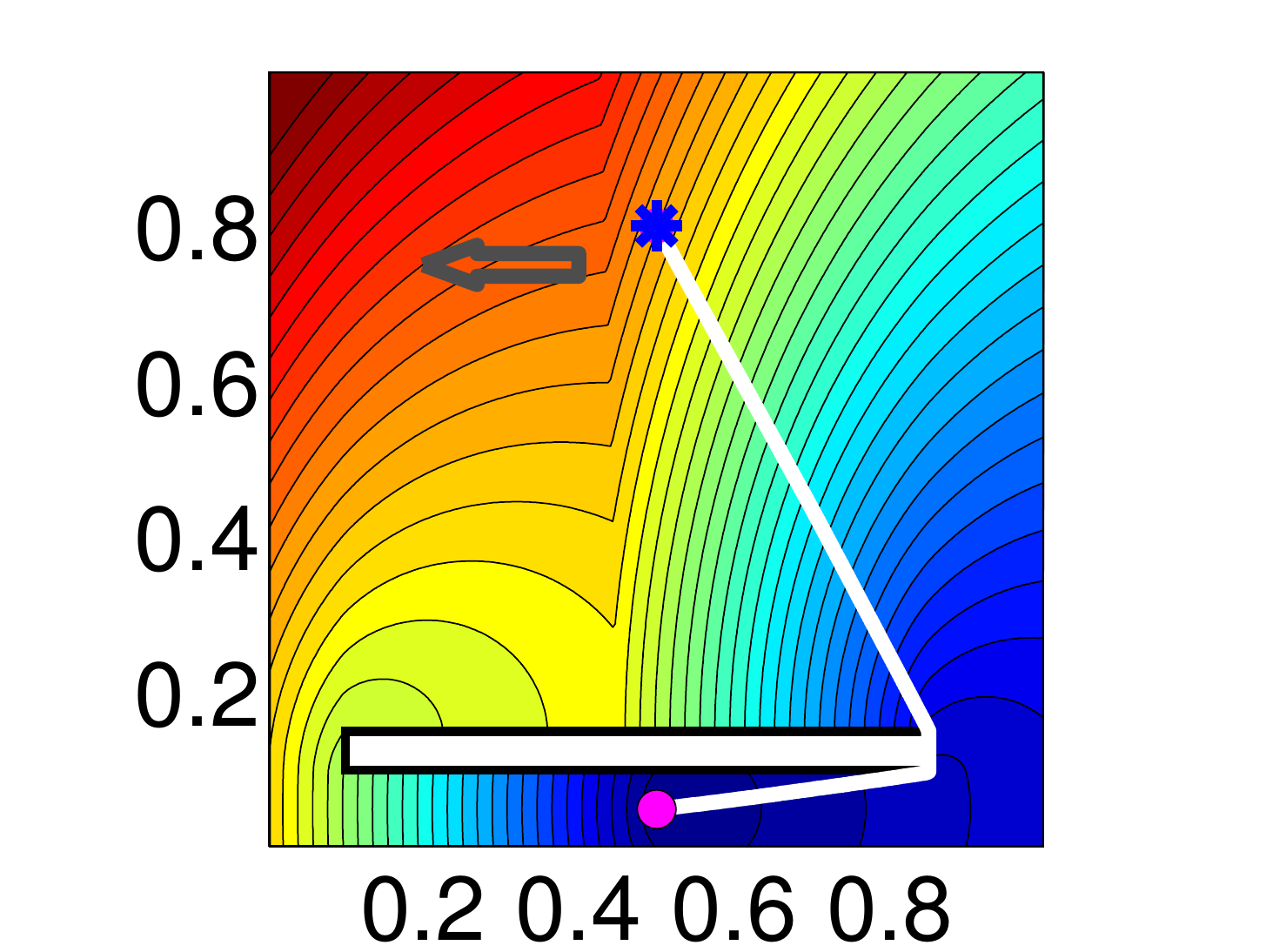}
}
\centerline{
\includegraphics[width=0.45\textwidth, height=0.33\textwidth]{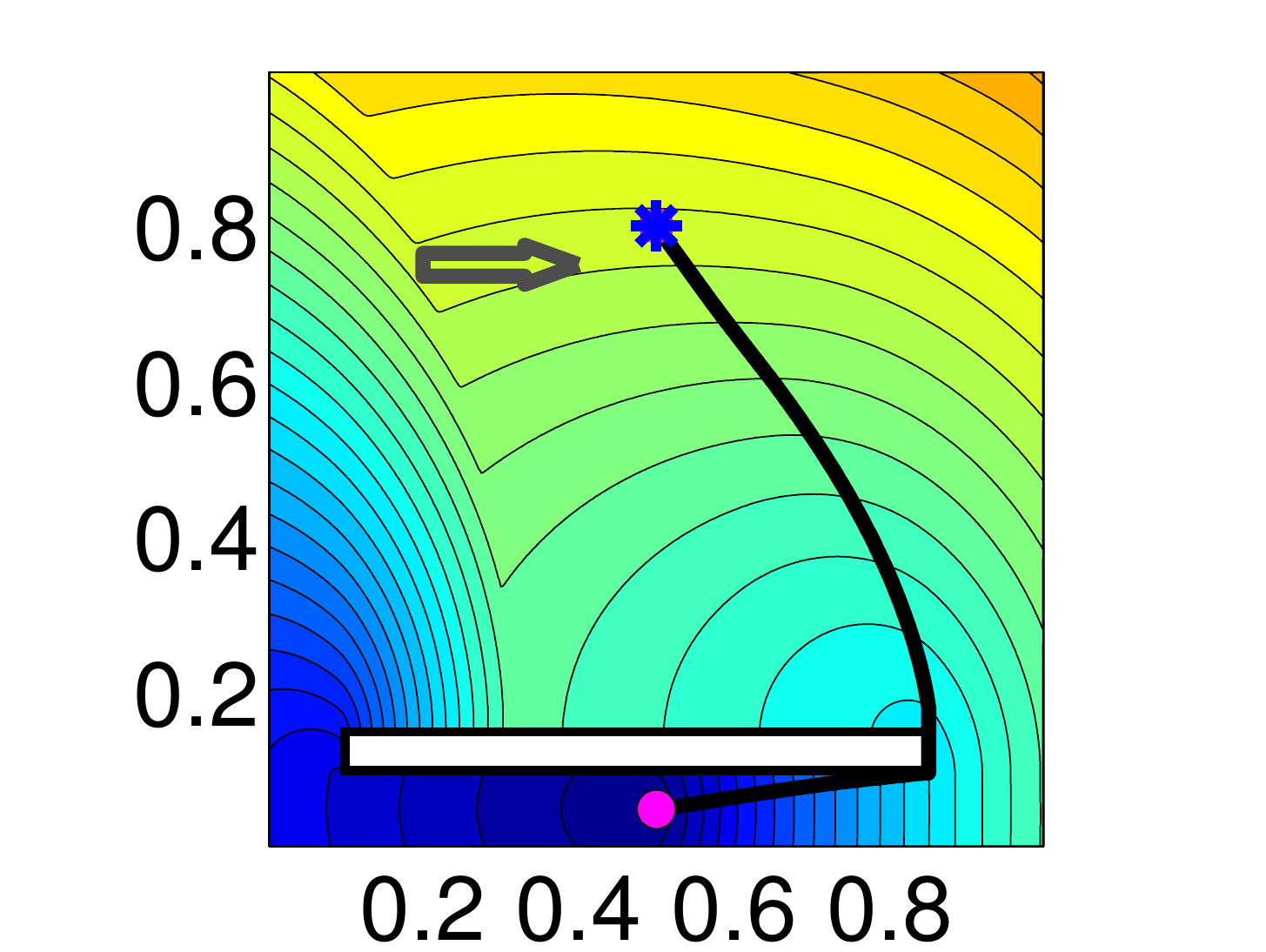}
\includegraphics[width=0.45\textwidth, height=0.33\textwidth]{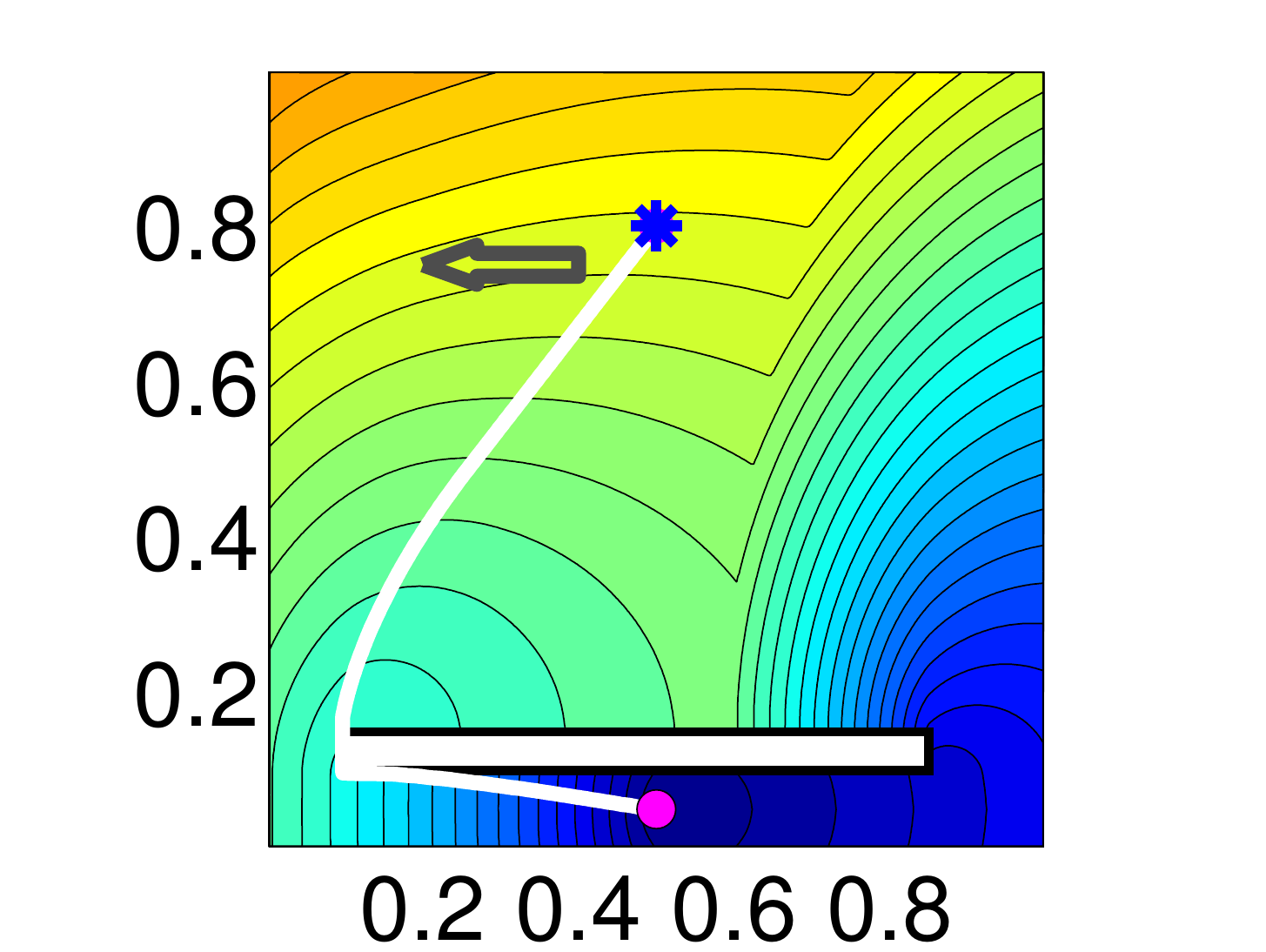}
}
\centerline{
\includegraphics[width=0.5\textwidth, height=0.2\textwidth]{Figures/colorbar.\FIGONETYPE}
}
\caption{
``Coupled'' vs ``uncoupled'' path planning.
A reference trajectory shown for the case when no wind transitions happen before we reach the target.
The color coding of trajectories corresponds to the current wind direction: black is to the East while white is to the West.
The first row: uncoupled planning (based on $u^0_i(\x)$).
The second row: coupled planning  under the assumption that $\lambda_{12}=\lambda_{21}=1$.
}
% COMMENT REMOVED
% COMMENT REMOVED
% COMMENT REMOVED
% COMMENT REMOVED
% COMMENT REMOVED
% COMMENT REMOVED
% COMMENT REMOVED
% COMMENT REMOVED
\label{figSinglePathNoTrans}
\end{figure}

%% file: Figures/Fig_Traj_Switch.tex
\def \FIGONETYPE {pdf}
\def \FIGTHREETYPE {eps}

\begin{figure}[!htbp]
\centerline{
\includegraphics[width=0.45\textwidth, height=0.33\textwidth]{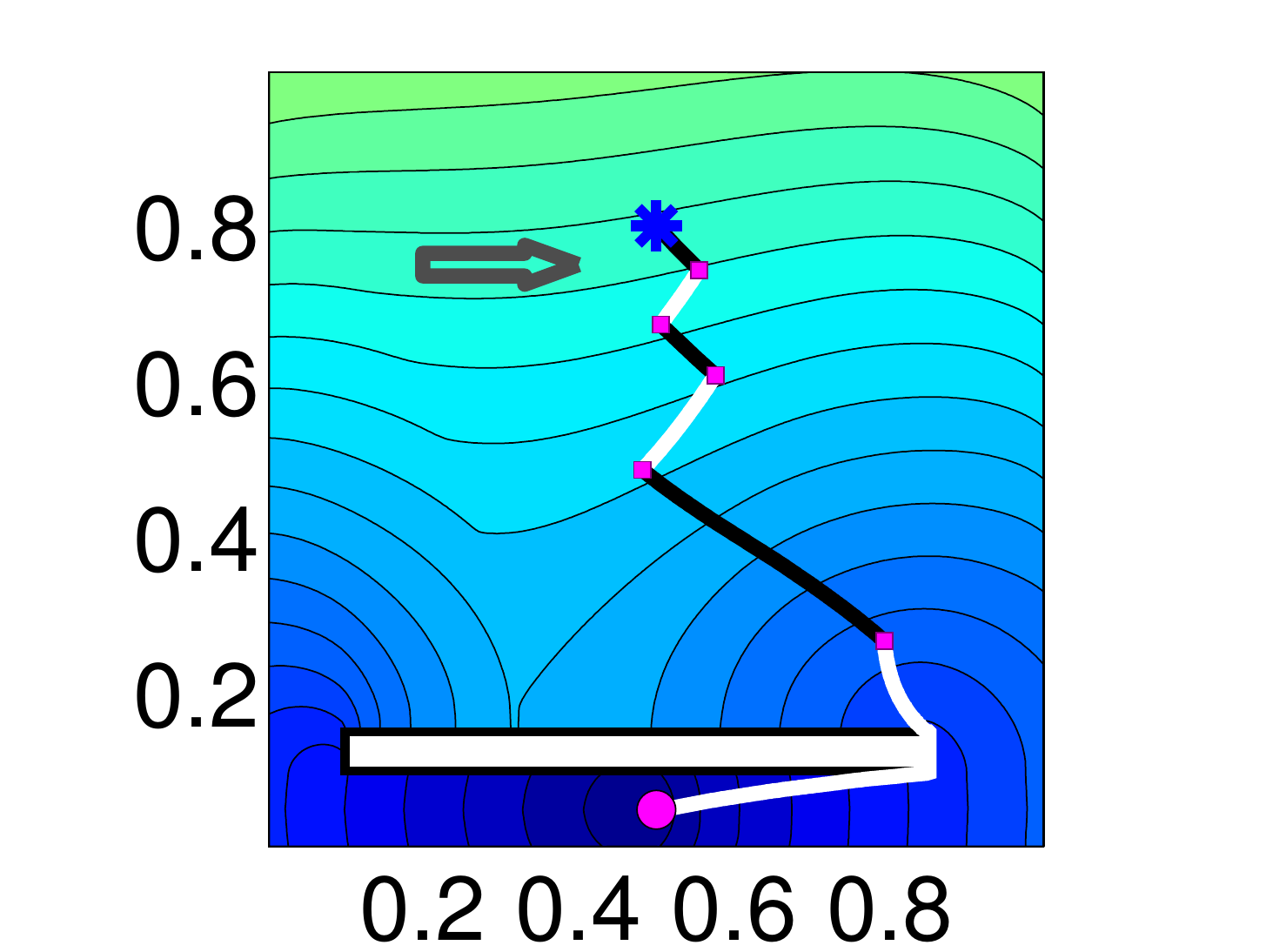}
\includegraphics[width=0.45\textwidth, height=0.33\textwidth]{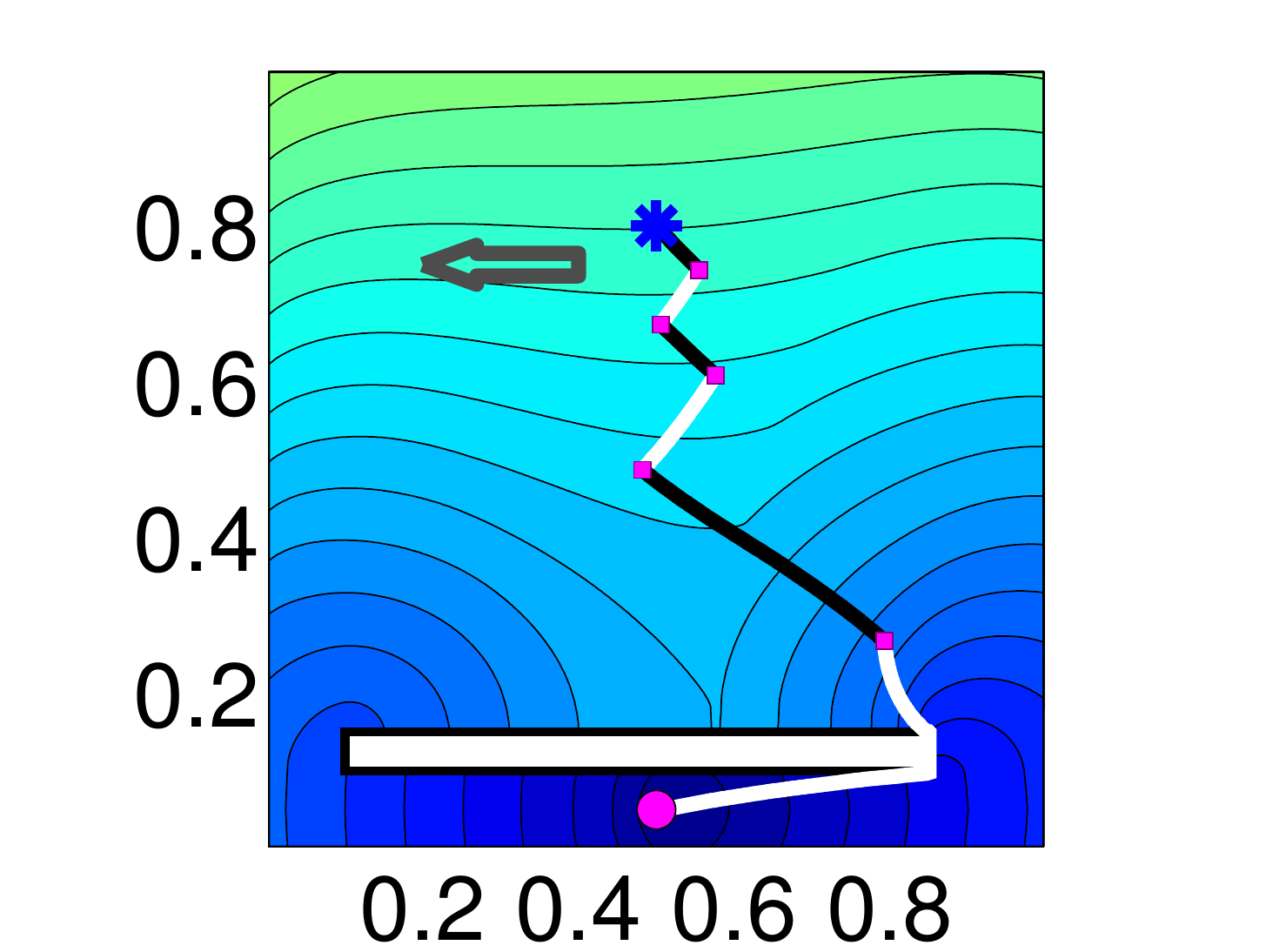}
}
\centerline{
\includegraphics[width=0.45\textwidth, height=0.33\textwidth]{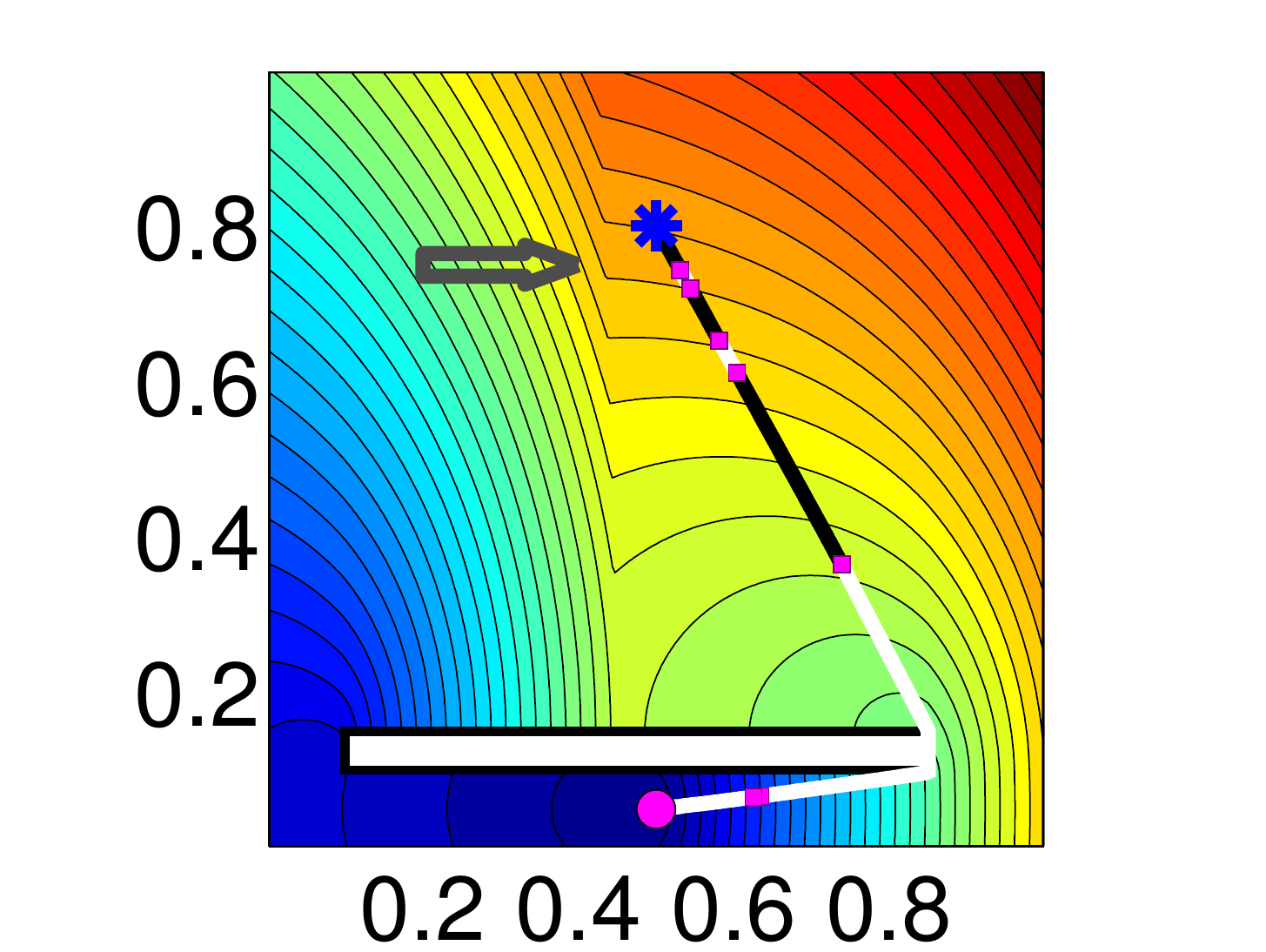}
\includegraphics[width=0.45\textwidth, height=0.33\textwidth]{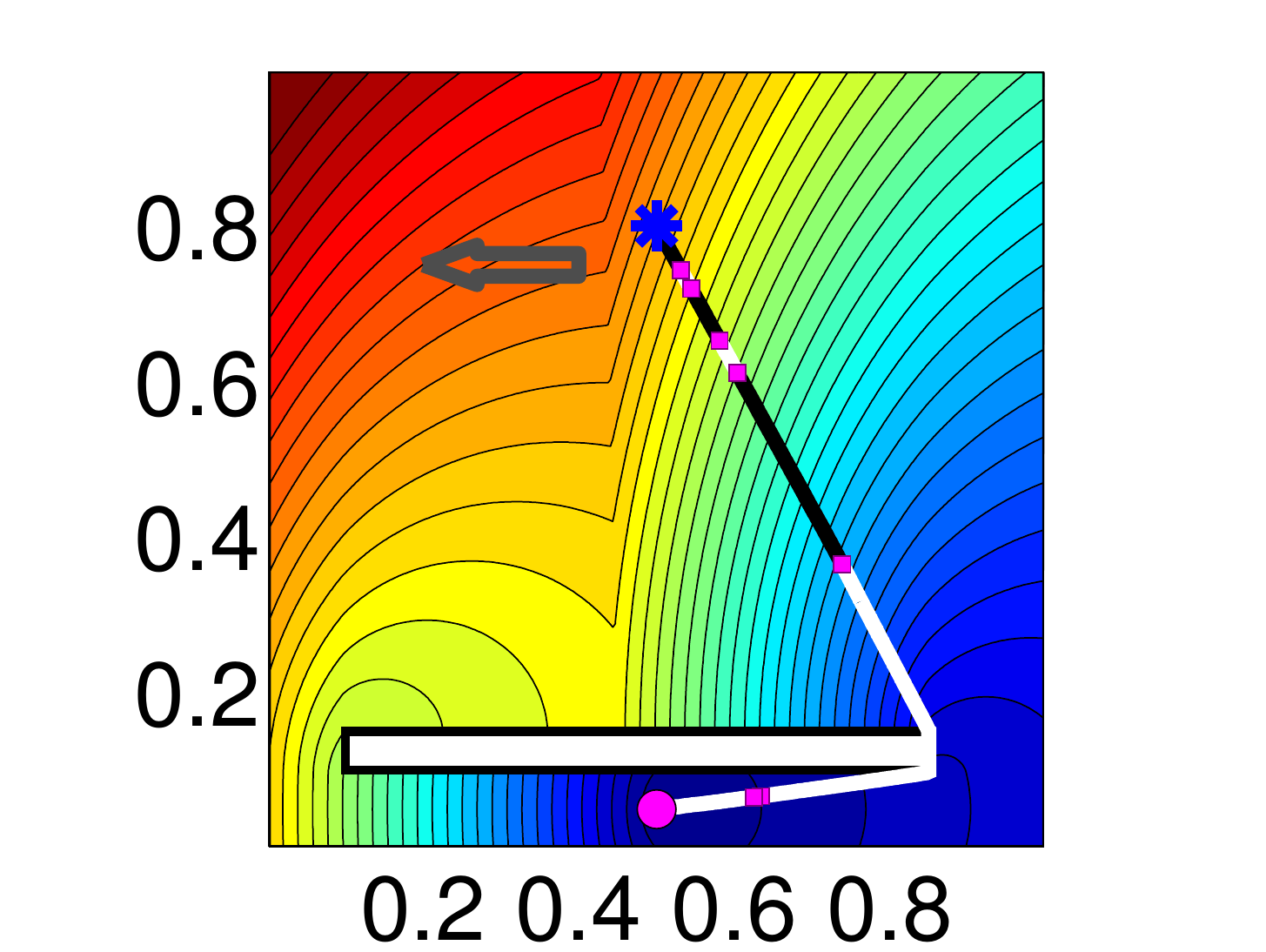}
}
\centerline{
\includegraphics[width=0.45\textwidth, height=0.33\textwidth]{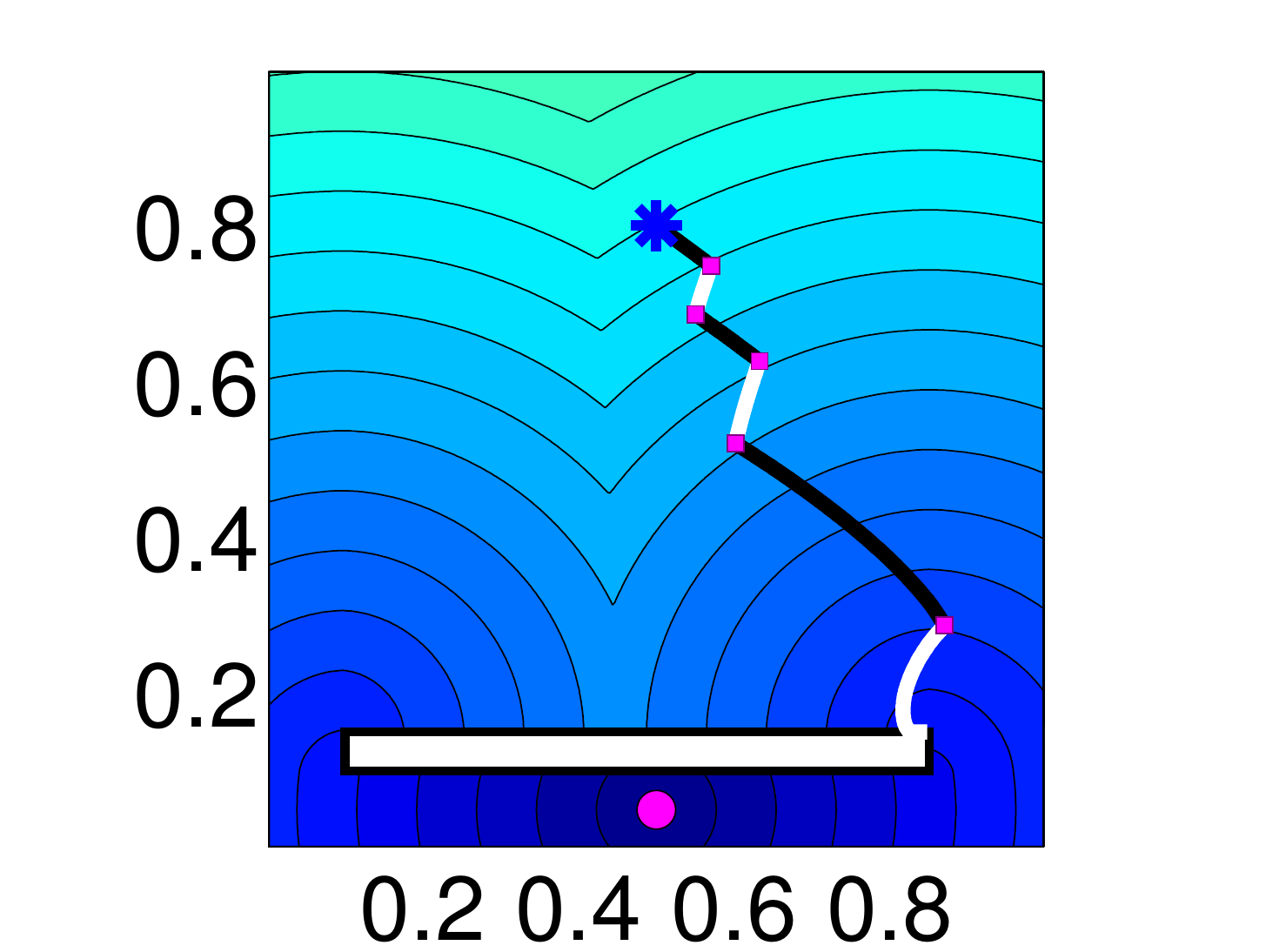}
% COMMENT REMOVED
% COMMENT REMOVED
\includegraphics[width=0.45\textwidth, height=0.2\textwidth]{Figures/colorbar.\FIGONETYPE}
}
\caption{
Trajectories based on coupled, uncoupled and infinite transition rate planners.
   The blue asterisk is the starting position $\xhat$, and the initial wind direction is to the East.
   The magenta disk is the target.
   Wind directions represented by grey arrows.
   % COMMENT REMOVED
   The same sequences of wind-switching times (based on the ``real'' transition rate $\lambda_{12}=\lambda_{21}=10$) used in all cases.
   The wind switch locations are shown by purple squares.
   The color coding of trajectories corresponds to the current wind direction: black is to the East while white is to the West.
   The first row: trajectory planning based on the correct/real transition rates.
   The second row: ``uncoupled'' planning (based on a no-switches assumption).
   The third row: infinite-transition-rate planning (we no longer keep track of the current mode) -- the target is never reached since we collide with an obstacle.
}
% COMMENT REMOVED
% COMMENT REMOVED
% COMMENT REMOVED
% COMMENT REMOVED
% COMMENT REMOVED
% COMMENT REMOVED
% COMMENT REMOVED
% COMMENT REMOVED
% COMMENT REMOVED
% COMMENT REMOVED
% COMMENT REMOVED
% COMMENT REMOVED
% COMMENT REMOVED
\label{figSinglePath}
\end{figure}

%% file: Figures/Fig_Traj_data.tex
\newcommand{\SwitchTime}{0.029, 0.064, 0.098, 0.159, 0.285, 0.689, 0.706}
\newcommand{\TCoupled}{0.589}
\newcommand{\TUncoupled}{0.741}
\newcommand{\TInfTrans}{0.423}

%% file: Figures/Fig_Simulation_Low.tex
\begin{figure}[!htbp]
\centerline{
\includegraphics[width=0.35\textwidth, height=0.265\textwidth]{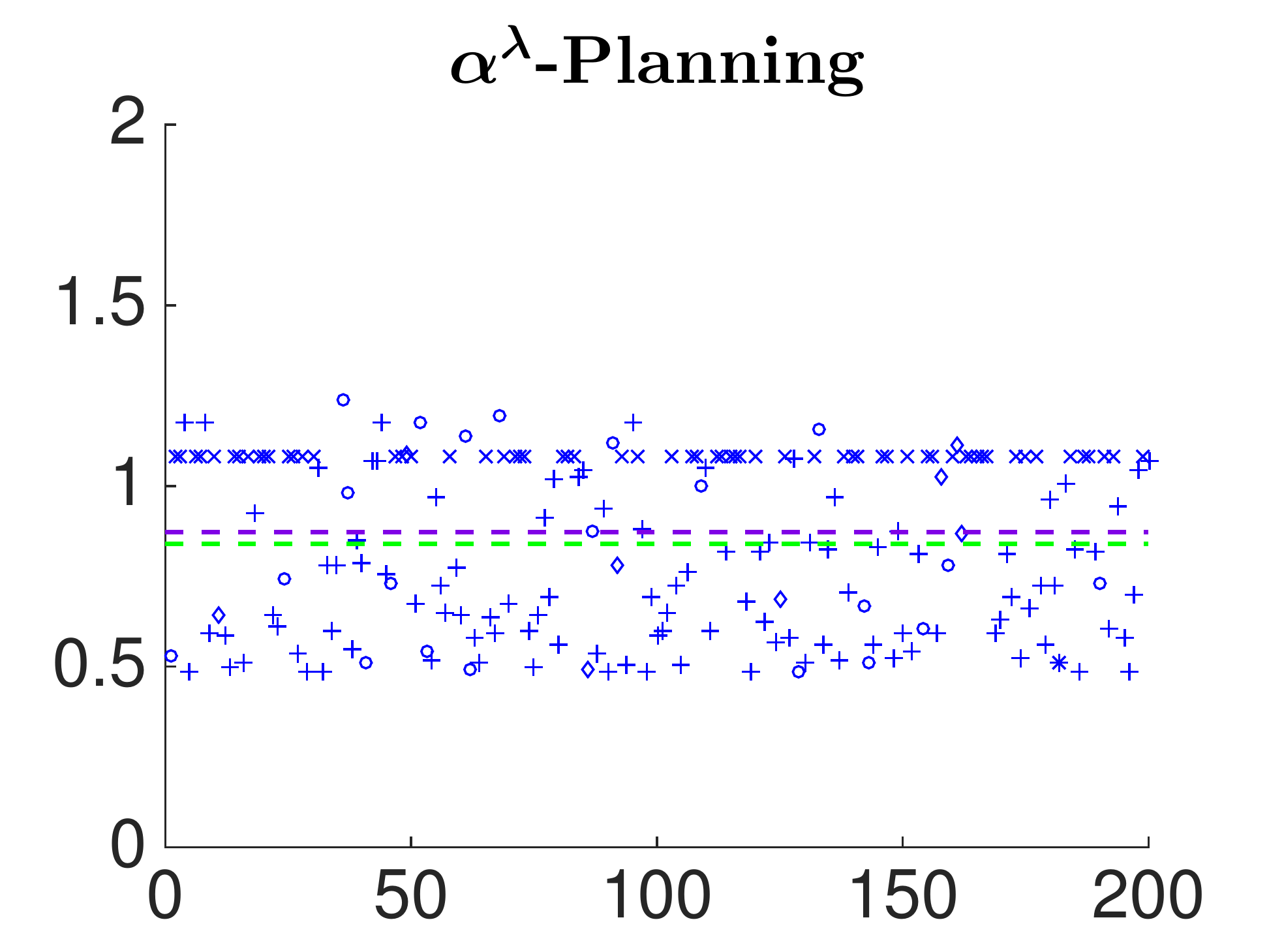}
\includegraphics[width=0.35\textwidth, height=0.265\textwidth]{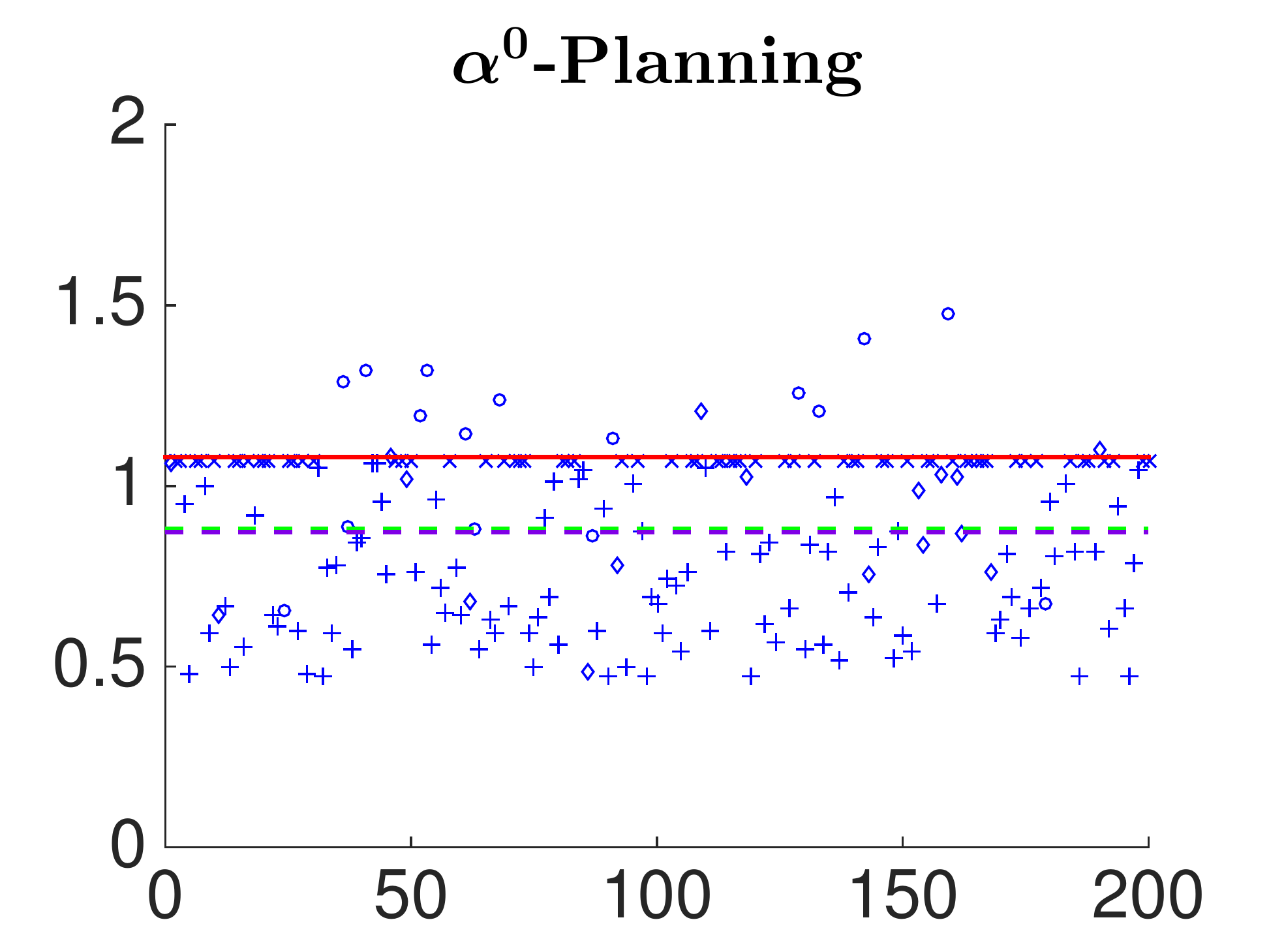}
\includegraphics[width=0.35\textwidth, height=0.265\textwidth]{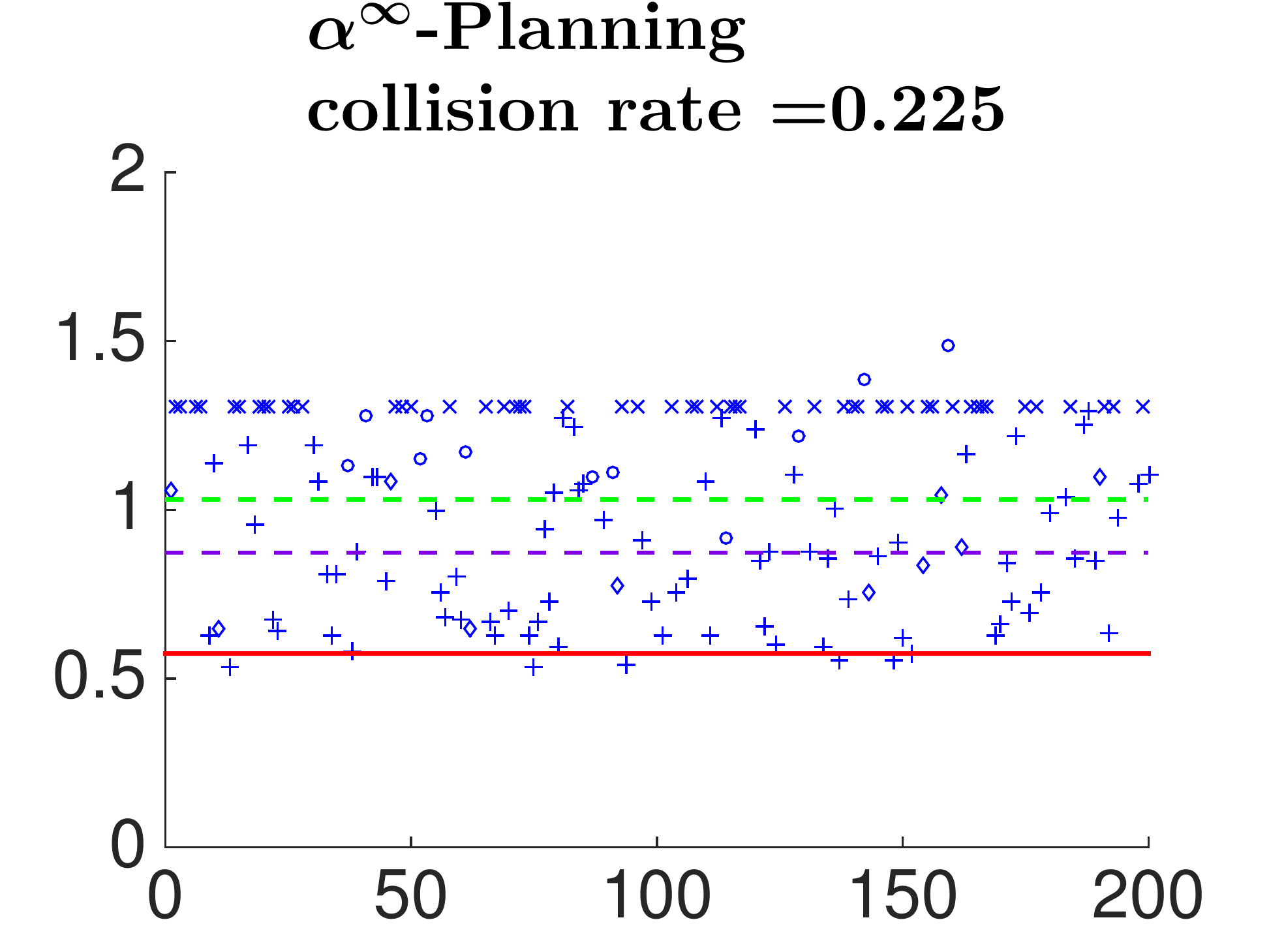}
}
% COMMENT REMOVED
% COMMENT REMOVED
% COMMENT REMOVED
% COMMENT REMOVED
% COMMENT REMOVED
\caption{
200 simulations starting at $\big( (0.5, 0.8), 1 \big)$ when real transition rate $\lambda=1$.
$x$ axis: the number of simulations;
$y$ axis: the time cost.
The labels of the dots indicates the number of transitions occurred before the boat reaches the target:
 $\times \to 0$ ,  $\plus \to 1$, $\circ \to 2$, $\Diamond \to 3$, $* \to 4$.
Green line: approximate $u^{r, p}$ (observed average cost) ;
 red line: $u^p$;
 purple line: $u^r$.
 %TO BE ADDED:
 The optimum expected cost $\approx 0.873$. The observed averages of coupled (optimal) planning, uncoupled planning and infinite-transition-rate planning are approximately 0.840, 0.882 and 1.030 respectively.
% COMMENT REMOVED
% COMMENT REMOVED
% COMMENT REMOVED
% COMMENT REMOVED
% COMMENT REMOVED
% COMMENT REMOVED
% COMMENT REMOVED
% COMMENT REMOVED
% COMMENT REMOVED
% COMMENT REMOVED
% COMMENT REMOVED
% COMMENT REMOVED
% COMMENT REMOVED
}
\label{figCostLow}
\end{figure}

%% file: Figures/Fig_Simulation_High.tex
\begin{figure}[!htbp]
\centerline{
\includegraphics[width=0.35\textwidth, height=0.265\textwidth]{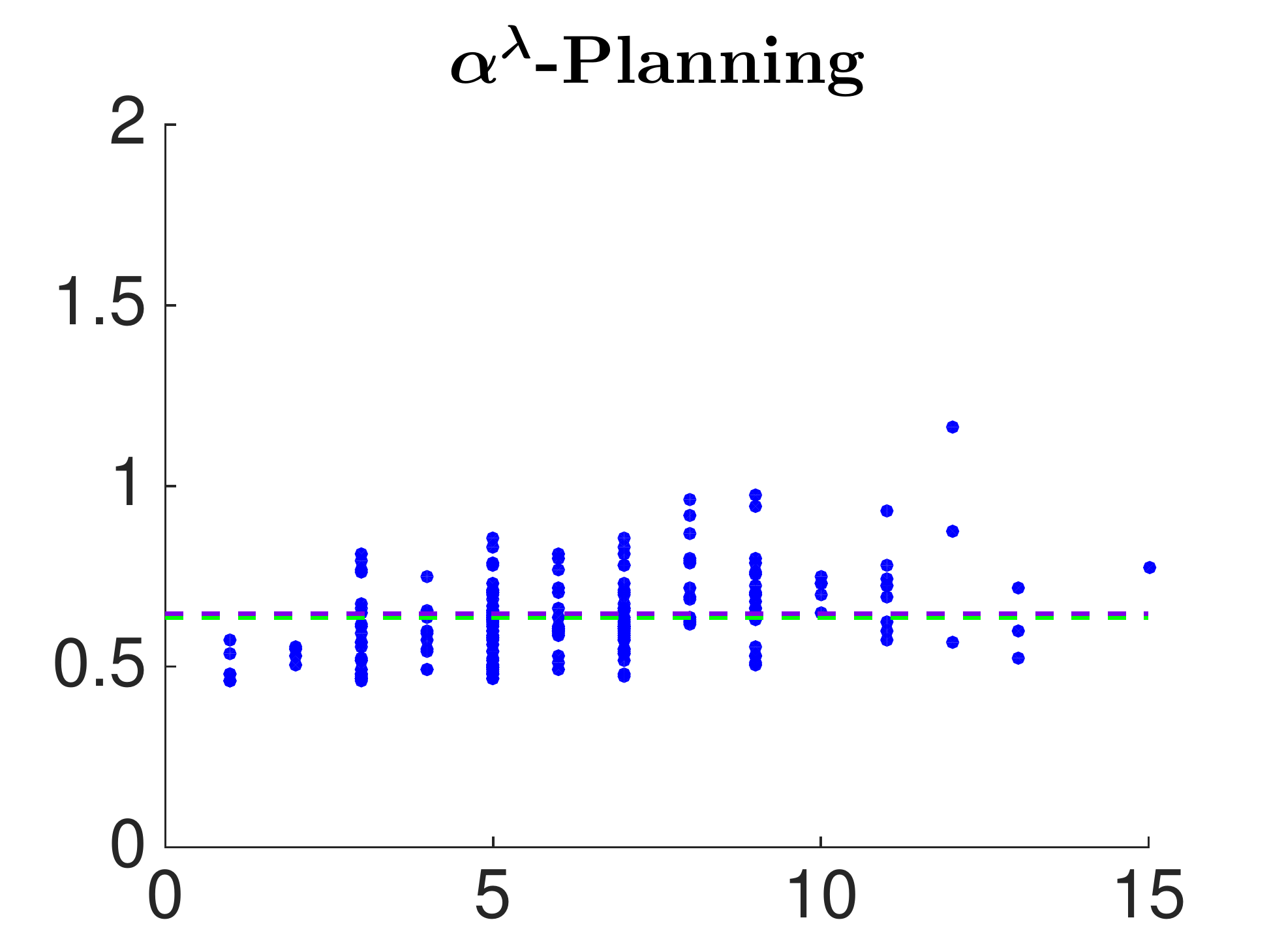}
\includegraphics[width=0.35\textwidth, height=0.265\textwidth]{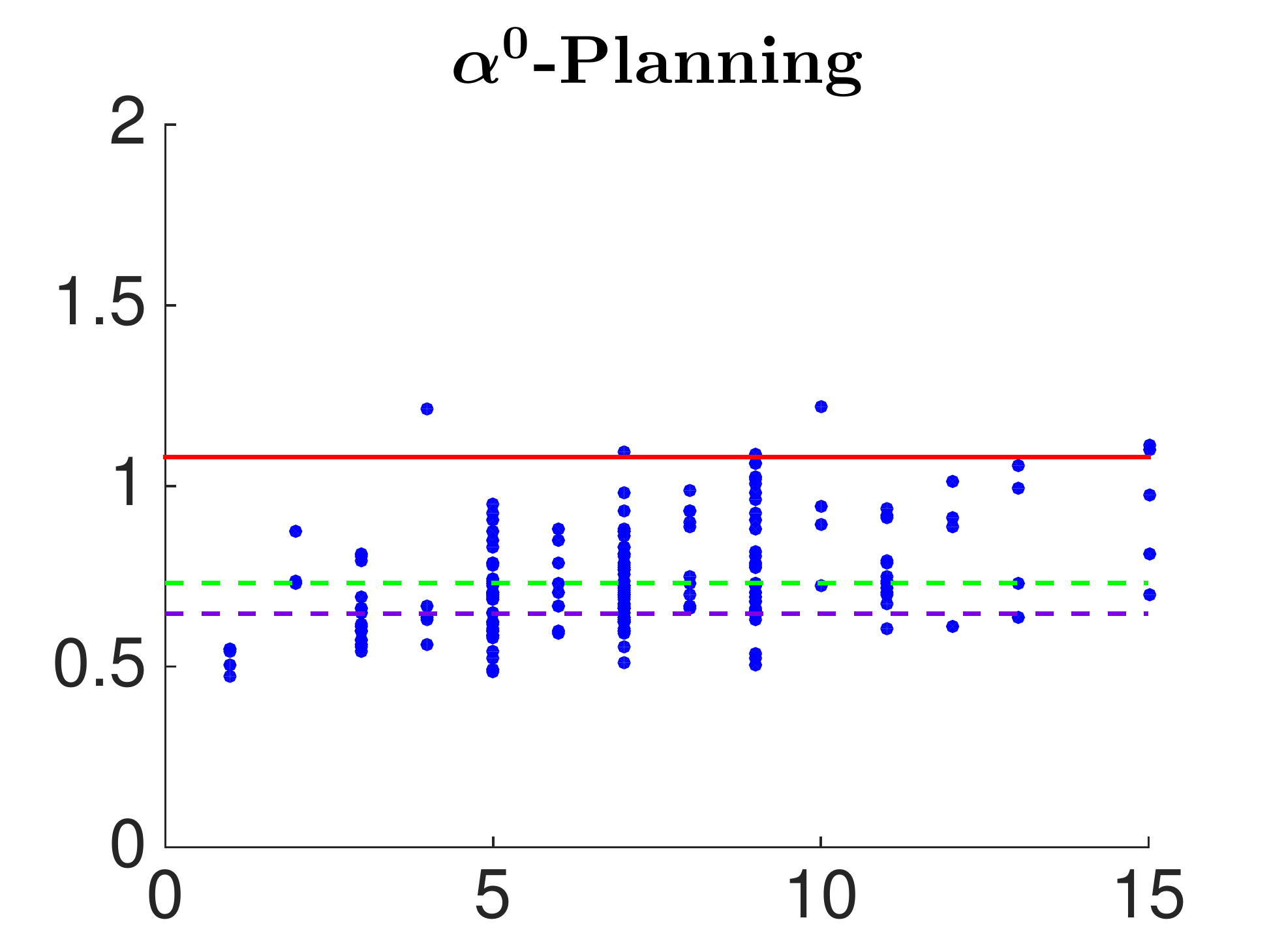}
\includegraphics[width=0.35\textwidth, height=0.265\textwidth]{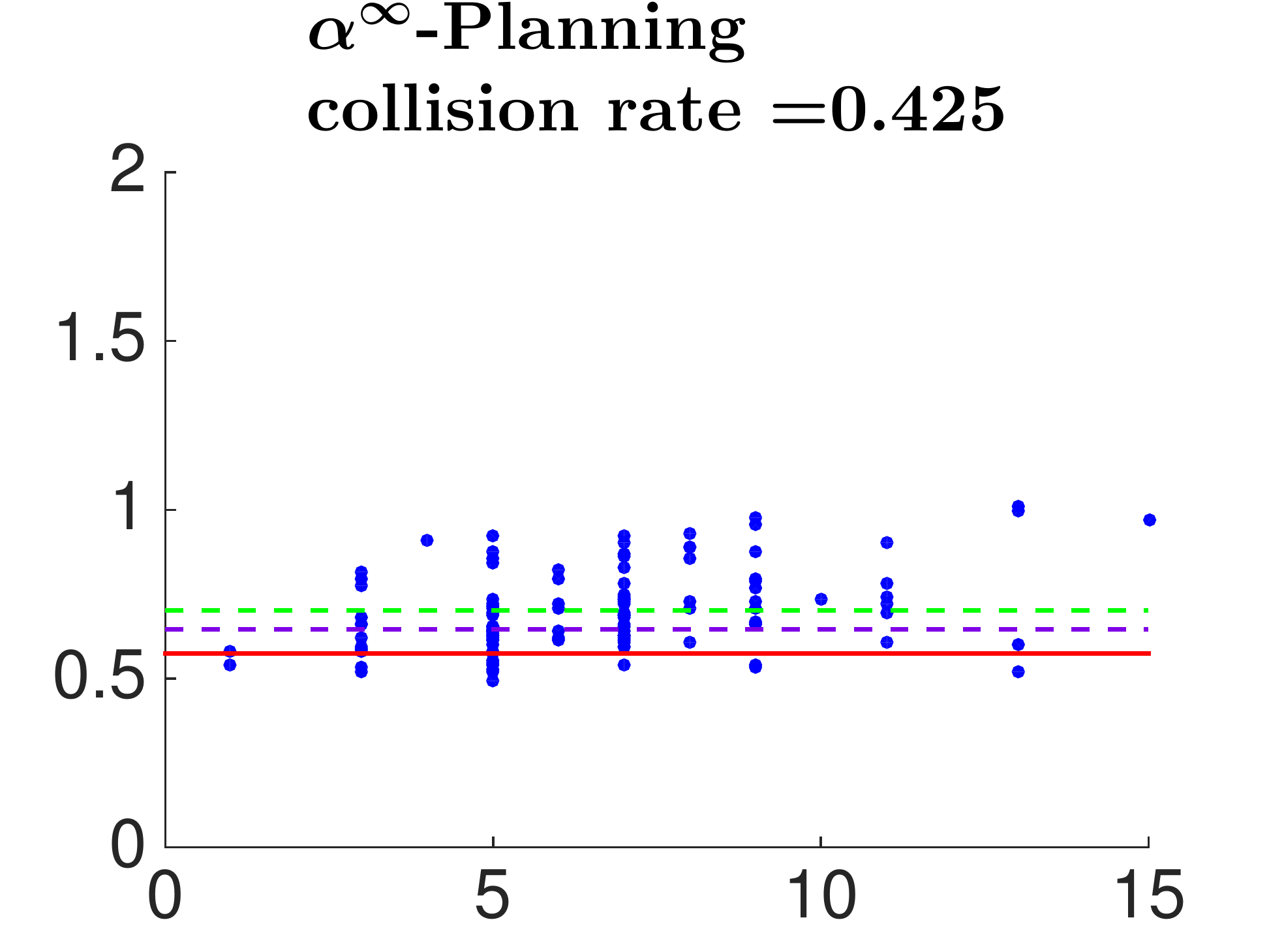}
}
% COMMENT REMOVED
% COMMENT REMOVED
% COMMENT REMOVED
% COMMENT REMOVED
% COMMENT REMOVED
\caption{
200 simulations starting at $\big( (0.5, 0.8), 1 \big)$, when real transition rate $\lambda=10$.
$x$ axis:  the number of transitions that happens in a particular simulation before the swimmer reaches the target;
 $y$ axis: the time cost.
Green line: approximate $u^{r, p}$ (observed average cost); red line: $u^p$; purple line: $u^r$.
 %TO BE ADDED:
 The optimum expected cost $\approx 0.646$. The observed averages of coupled (optimal) planning, uncoupled planning and infinite-transition-rate planning are approximately 0.636, 0.731 and 0.702 respectively.
% COMMENT REMOVED
% COMMENT REMOVED
% COMMENT REMOVED
% COMMENT REMOVED
% COMMENT REMOVED
% COMMENT REMOVED
% COMMENT REMOVED
% COMMENT REMOVED
% COMMENT REMOVED
}
\label{figCostHigh}
\end{figure} 

%% file: tex/BM.tex
% COMMENT REMOVED
\subsection{Continuous Mode Space Model}
In the rowboat problem of section \ref{sec:experiments}, we have assumed that the wind has only two possible directions.
However, in real world the wind direction is a continuous variable.
We will now consider a different (though also admittedly unrealistic) model, where the wind speed $s_w$ is constant, the wind velocity is
% COMMENT REMOVED
$
\w(\theta)=s_w(\cos \theta , \sin \theta )^T,
$
% COMMENT REMOVED
and $\theta$ undergoes a Brownian Motion on $\mathbb{R}$.
% COMMENT REMOVED
% COMMENT REMOVED
% COMMENT REMOVED
% COMMENT REMOVED
% COMMENT REMOVED
% COMMENT REMOVED
% COMMENT REMOVED
In this case, it is natural to pose the optimal control problem %in a continuous expanded state space
on an expanded state space
$\tilde{\Omega} = \Omega\times\mathbb{R}$ with the dynamics defined by
% COMMENT REMOVED
% COMMENT REMOVED
\begin{align}
\left\{
\begin{array}{ll}
\y'(t)&=\f \, \Big( \y(t),\, \theta(t),\, \ba \big( \y(t),\theta(t) \big) \Big),\\
d\theta &= \sigma dW,\\
\y(t) &= \x,\\
\theta(0)&=\theta_0,
\end{array}
\right.
\end{align}
where $W$ is a standard Wiener process, $\sigma >0$, and $(\x, \theta_0)$ encode the initial boat position and wind direction.
% COMMENT REMOVED
Assuming the running cost $C(\x, \theta, \balpha)$,
the total cost of this process is
\begin{align}
% COMMENT REMOVED
J(\x, \theta_0, \balpha)=
% COMMENT REMOVED
\int_{0}^{T}C\left(\y(t), \theta(t), \balpha\left(\y(t), \theta(t)\right) \right)dt+q(\y(T)),
\end{align}
and the value function is
% COMMENT REMOVED
% COMMENT REMOVED
$
u(\x, \theta) \; = \; \inf_{\balpha\in\mathcal{A}}\mathbb{E} [J(\x, \theta, \balpha)].
$
% COMMENT REMOVED
Starting with Bellman's optimality condition and Taylor expanding (using Ito's Lemma), we can formally obtain the HJB equation for the value function:
\begin{align}
\min_{\ba\in {A}}\{\nabla_{\x} u(\x, \theta)\cdot \f(\x, \theta, \ba)+
\frac{\sigma^2}{2} \frac{\partial ^2}{\partial \theta^2}u(\x, \theta)+C(\x, \theta, \ba)\}
\; = \;
0.
\end{align}
% COMMENT REMOVED

\noindent
Since the wind direction is periodic, %i.e., $\w(\theta) = \w(\theta + 2\pi)$,
$u(\x, \theta)$ is also periodic with respect to $\theta$:
\begin{align}\label{eq:BM_BC}
u(\x, \theta) \; = \; u(\x, \theta + 2\pi), \qquad \forall \theta \in \mathbb{R}.
\end{align}
Therefore, we just need to compute $u(\x, \theta)$ on $\Omega \times [0, 2\pi)$.
If we divide $[0, 2\pi)$ into intervals with length $\Delta \theta = \frac{2\pi}{n}$,
and introduce discrete values $\theta_i=i\Delta\theta$ for $i = 0, 1, \cdots, n - 1$,
then the second-order term, $\frac{\sigma^2}{2} \frac{\partial ^2}{\partial \theta^2}u(\x, \theta)$, can be approximated with central difference method.
We will use $\tildeu_i(\x)=\tildeu(\x, \theta_i)$ to denote this $\theta$-discretized approximation of $u(\x, i\Delta\theta)$.
% COMMENT REMOVED
% COMMENT REMOVED
% COMMENT REMOVED
These functions must satisfy the system of first-order HJB PDEs
\begin{align}
\min_{\ba\in A}\left\{ \nabla_{\x} \tildeu_i(\x)\cdot \f(\x,\theta_i,\ba)+C(\x,\theta_i, \ba) \right\}+
\frac{\sigma^2}{2(\Delta\theta) ^2}\sum_{\epsilon=\pm 1} ( \tildeu_{i+\epsilon}(\x)- \tildeu_i(\x))
\; = \;
0;
\qquad  \qquad i = 0, 1, \cdots, n-1,
\end{align}
where we have used the notation $\tildeu_n(\x) = \tildeu_0(\x)$ and $\tildeu_{-1}(\x) = \tildeu_{n-1}(\x)$.
% COMMENT REMOVED
% COMMENT REMOVED

\noindent
This can be also interpreted as a piecewise-deterministic system discussed in section \ref{sec:statement}, provided
\begin{itemize}
\item[(a)] from every mode $i=0,\cdots,(n-1)$,
the switching can happen only to its ``adjacent'' modes $(i-1)$ and $(i+1)$;
\item[(b)] the rate of switching to these adjacent modes is equal to $\frac{\sigma^2}{2(\Delta\theta)^2}=\frac{\sigma^2n^2}{8\pi^2}$.
\end{itemize}
As the number of modes grows, the original value function on $\mathbb{R}^{d+1}$ is recovered in the limit:
% COMMENT REMOVED
$$
u \left(\x,\theta \right)
 \; = \;
\lim_{n \rightarrow \infty}  \tildeu \left(\x, {\frac{2\pi}{n} \left\lfloor \frac{n\theta}{2\pi} \right\rfloor} \right).
$$

% COMMENT REMOVED
% COMMENT REMOVED
% COMMENT REMOVED

\subsection{Upper Bound on $\theta$-variability}
It is easy to derive a uniform upper bound on
$| u(\x,\theta_1) - u(\x,\theta_2) |, \; \forall \x,\theta_1,\theta_2.$
% COMMENT REMOVED
% COMMENT REMOVED
Suppose {\em the first hitting time} is $\kappa (\theta_1, \theta_2) = \min\limits_{t\geq 0}\left\{t \mid \theta(t_0+t)=\theta_2+2n\pi \text{ for some integer } n \; | \; \theta(t_0)=\theta_1\right\}$.
% COMMENT REMOVED
% COMMENT REMOVED
% COMMENT REMOVED
% COMMENT REMOVED
% COMMENT REMOVED
Since $u$ is periodic in $\theta$, we just need to consider the case $\theta_2 - 2 \pi < \theta_1 < \theta_2  < \theta_1 + 2\pi$ and
$\mathbb{E}[\kappa(\theta_2, \theta_1)] = \phi(\theta_2-\theta_1),$ where $\phi(\alpha)$ is the solution of
$$\phi''(\alpha) = -\frac{2}{\sigma^2}, \quad \forall \alpha \in (0,2\pi);  \qquad \qquad \qquad \phi(0)=\phi(2 \pi) =0.$$
Thus, $\phi(\alpha) = \alpha (2 \pi - \alpha) / \sigma^2 \leq \frac{\pi^2}{\sigma^2}.$
By symmetry the same formula is also valid for $\mathbb{E}[\kappa(\theta_1, \theta_2)].$
An argument similar to that in Appendix \ref{append:upperbound}, shows that
$$
| u(\x,\theta_1) - u(\x,\theta_2) | \; \leq \; \max\{\mathbb{E} [ \kappa (\theta_1, \theta_2)], \, \mathbb{E} [ \kappa (\theta_2, \theta_1)] \}  \leq -\frac{1}{\sigma^2} (\theta_2-\theta_1) (\theta_2-\theta_1-2\pi) \leq \frac{\pi^2}{\sigma^2}.
$$
Similar upper bounds are also easy to derive for $\| \tildeu_i - \tildeu_j \|_{\infty}$
in the $\theta$-discretized version.
% COMMENT REMOVED
% COMMENT REMOVED
% COMMENT REMOVED
% COMMENT REMOVED
% COMMENT REMOVED
% COMMENT REMOVED
% COMMENT REMOVED
% COMMENT REMOVED
% COMMENT REMOVED
% COMMENT REMOVED
% COMMENT REMOVED
% COMMENT REMOVED
% COMMENT REMOVED
% COMMENT REMOVED
% COMMENT REMOVED
% COMMENT REMOVED
% COMMENT REMOVED
% COMMENT REMOVED
% COMMENT REMOVED
% COMMENT REMOVED

\subsection{Example}
\input{Figures/Fig_BM.tex}
Returning to the problem of time-optimal trajectories (i.e., $C \equiv 1$),
we will assume that the rowing speed is $s(\x)=2$, while the wind speed is $s_w=1.5$ and the wind direction undergoes Brownian motion with $\sigma = 2$.  The resulting second-order HJB PDE is
$$
2 \| \nabla u \| - \w(\theta) \cdot \nabla u \; = \; 1 \, + \, 2 u_{\theta \theta},
$$
where
$
\w(\theta)
\; = \;
1.5 \left( \cos(\theta), \, \sin(\theta) \right)^T.
$
To treat this in the framework of section \ref{sec:experiments},
we discretize in $\theta$ using $n$ different wind modes
with
$$
\theta_i \, = \, \frac{2 i \pi}{n};
\qquad
\w_i
\, = \,
\w(\theta_i);
\qquad
\lambda_{ij} \, = \, \left\{
\begin{array}{rl}
\frac{n^2}{2 \pi^2}, & j\equiv i\pm 1 \mod n;\\
0, & \text{otherwise};\\
\end{array}
\right.
\qquad
i,j=0, \ldots, n-1.
$$
Figure \ref{fig:BM} presents the numerical results for $n=8$ wind modes.

% COMMENT REMOVED
% COMMENT REMOVED
% COMMENT REMOVED
% COMMENT REMOVED
% COMMENT REMOVED
% COMMENT REMOVED
% COMMENT REMOVED
% COMMENT REMOVED
% COMMENT REMOVED
% COMMENT REMOVED
% COMMENT REMOVED
% COMMENT REMOVED
% COMMENT REMOVED
% COMMENT REMOVED
% COMMENT REMOVED
% COMMENT REMOVED
% COMMENT REMOVED
% COMMENT REMOVED
% COMMENT REMOVED
% COMMENT REMOVED
% COMMENT REMOVED
% COMMENT REMOVED
% COMMENT REMOVED
% COMMENT REMOVED
% COMMENT REMOVED
% COMMENT REMOVED
% COMMENT REMOVED

% COMMENT REMOVED
% COMMENT REMOVED
% COMMENT REMOVED
% COMMENT REMOVED
% COMMENT REMOVED 

%% file: Figures/Fig_BM.tex
\def \FIGFIVETYPE {pdf}

\begin{figure}[!ht]
\centerline{
\includegraphics[width=0.25\textwidth, height=0.25\textwidth]{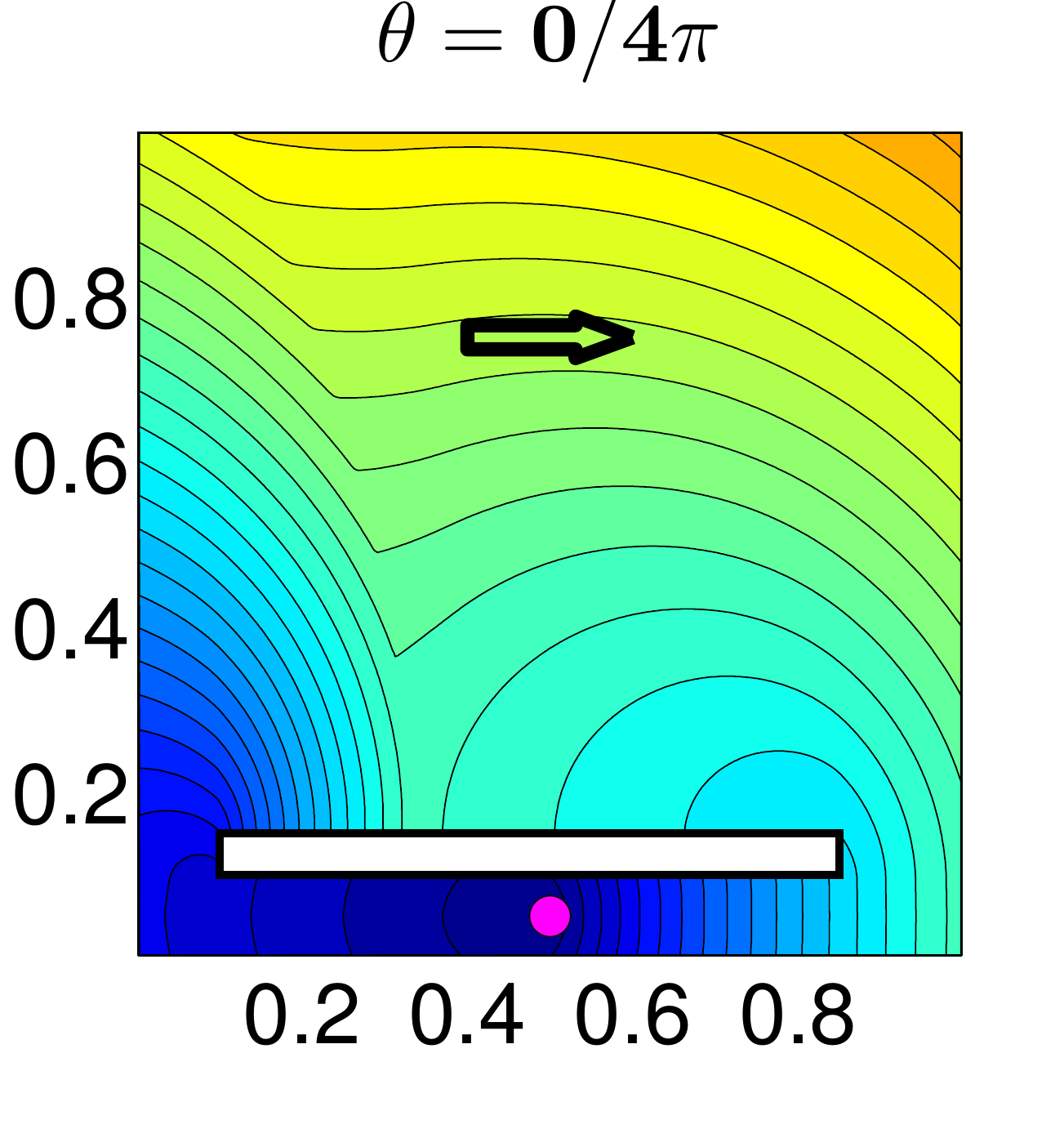}
\includegraphics[width=0.25\textwidth, height=0.25\textwidth]{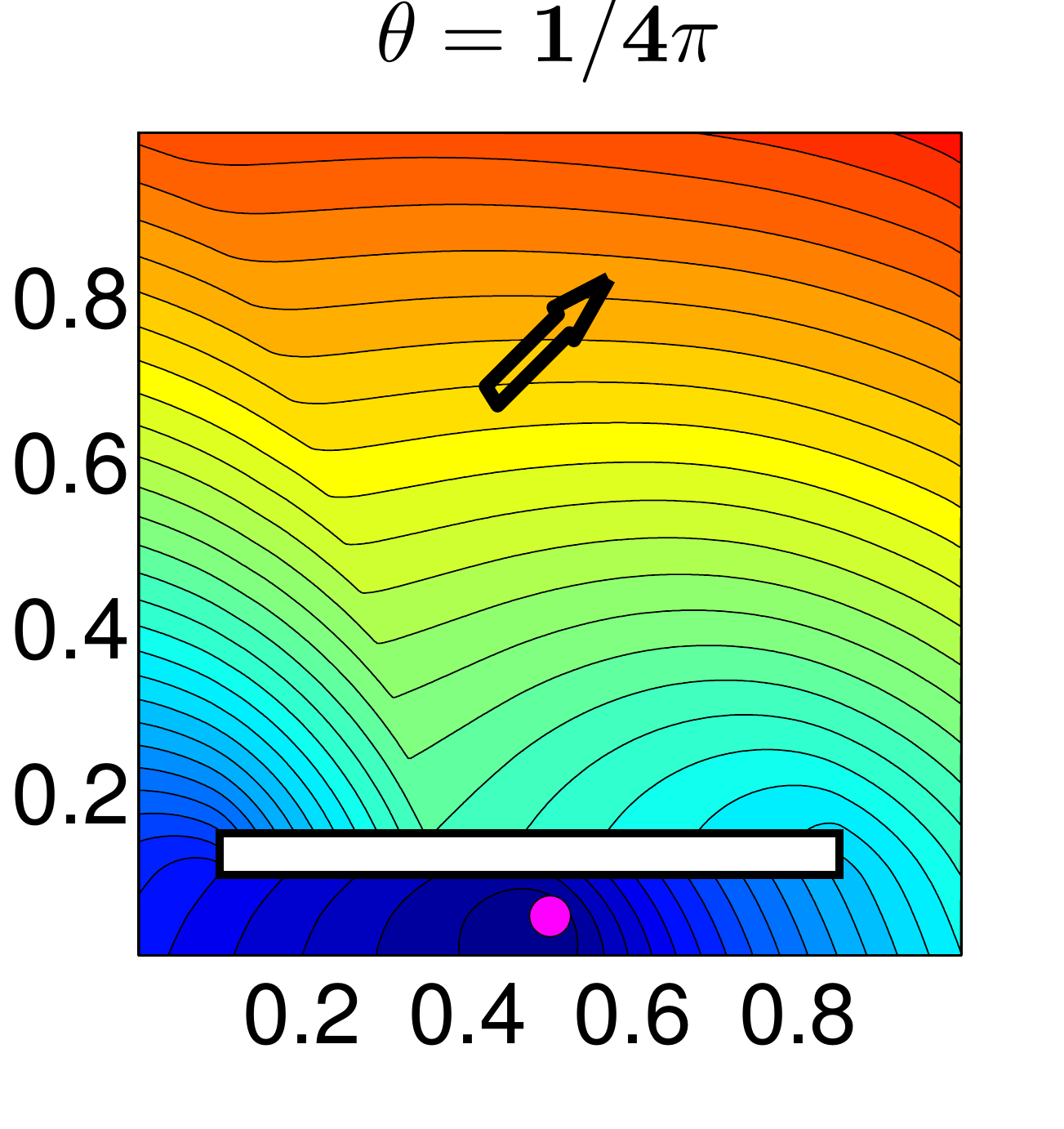}
\includegraphics[width=0.25\textwidth, height=0.25\textwidth]{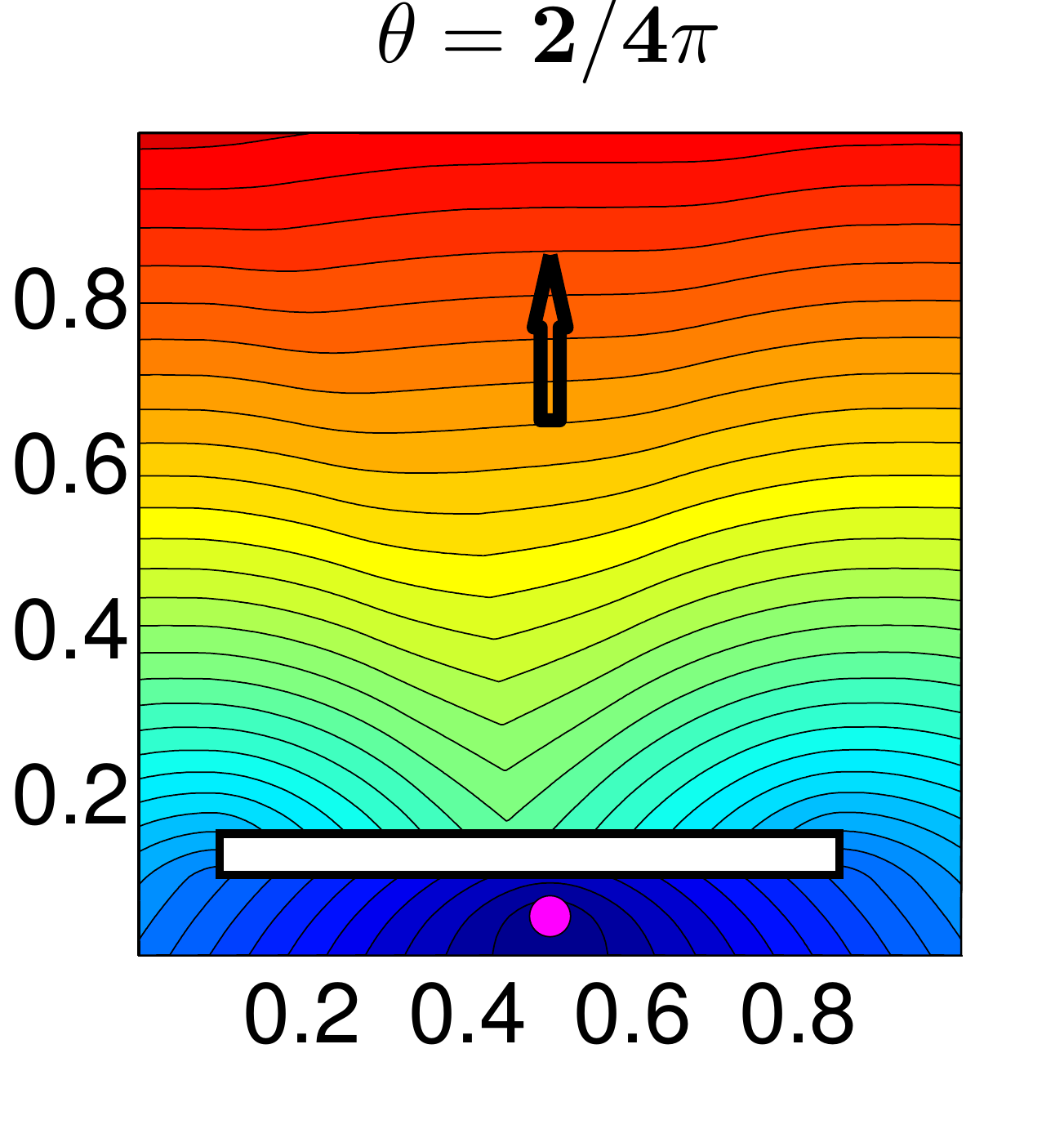}
\includegraphics[width=0.25\textwidth, height=0.25\textwidth]{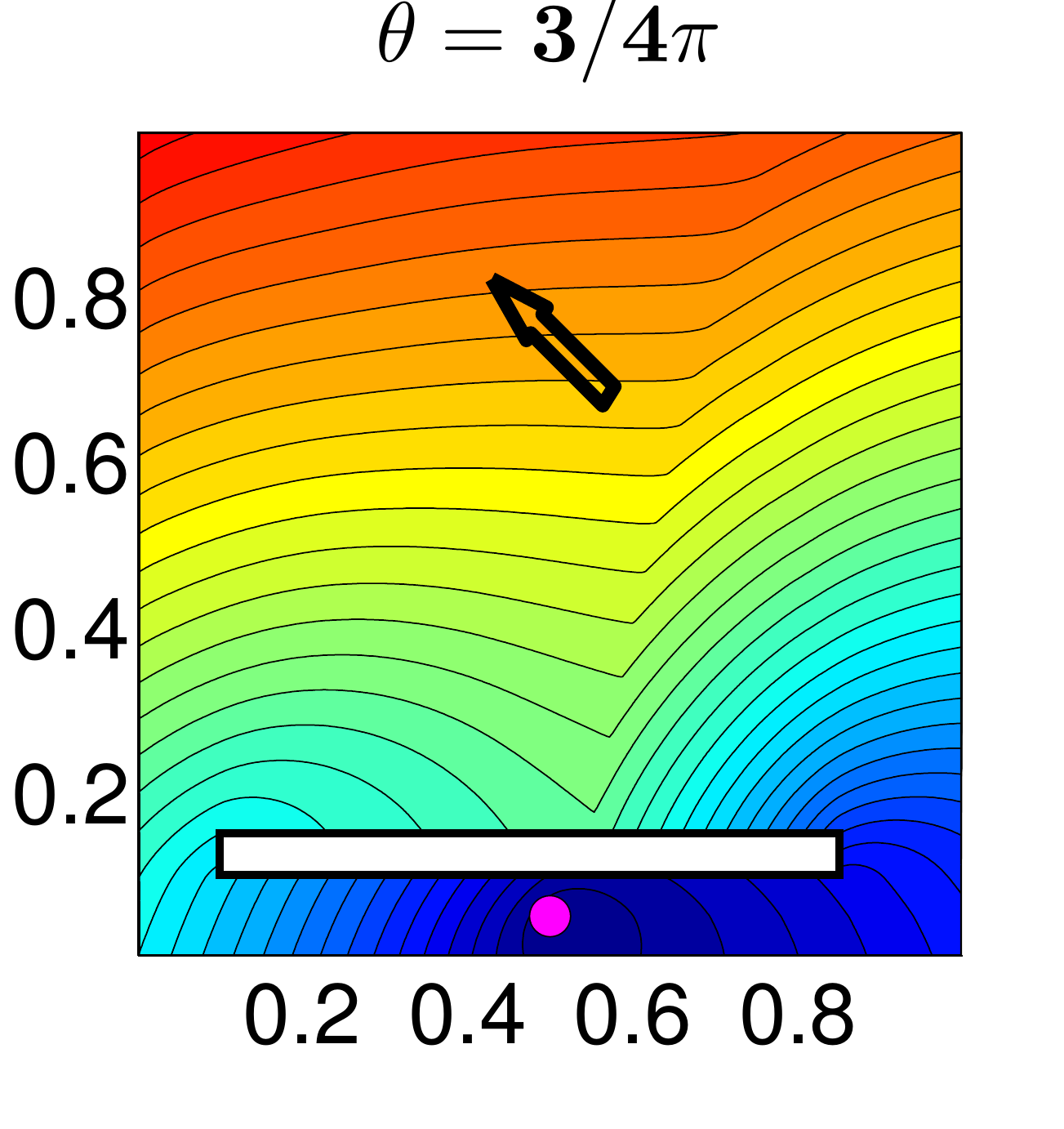}
}
\centerline{
\includegraphics[width=0.3\textwidth, height=0.01\textwidth]{Figures/blank.png}
}
\centerline{
\includegraphics[width=0.25\textwidth, height=0.25\textwidth]{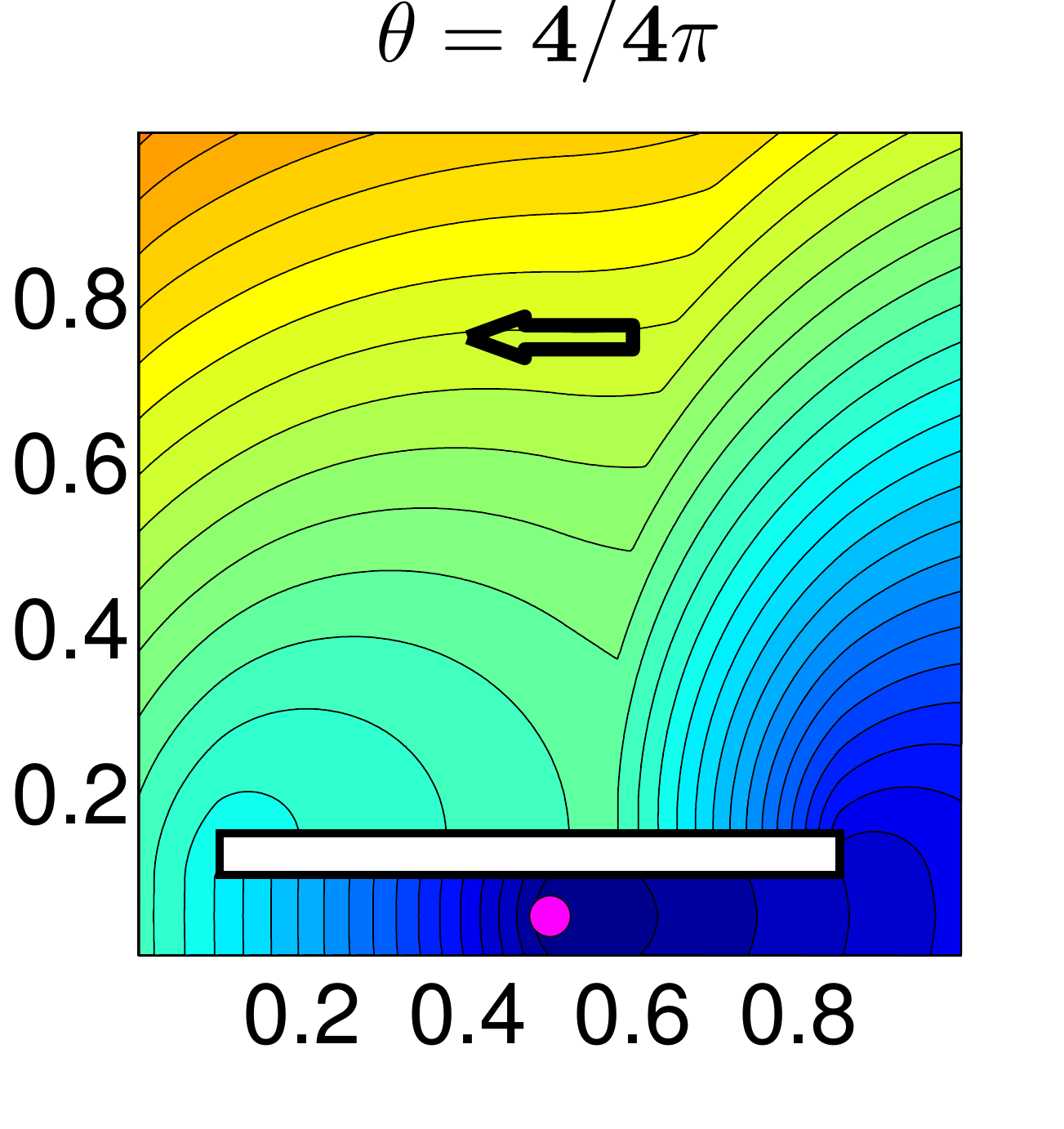}
\includegraphics[width=0.25\textwidth, height=0.25\textwidth]{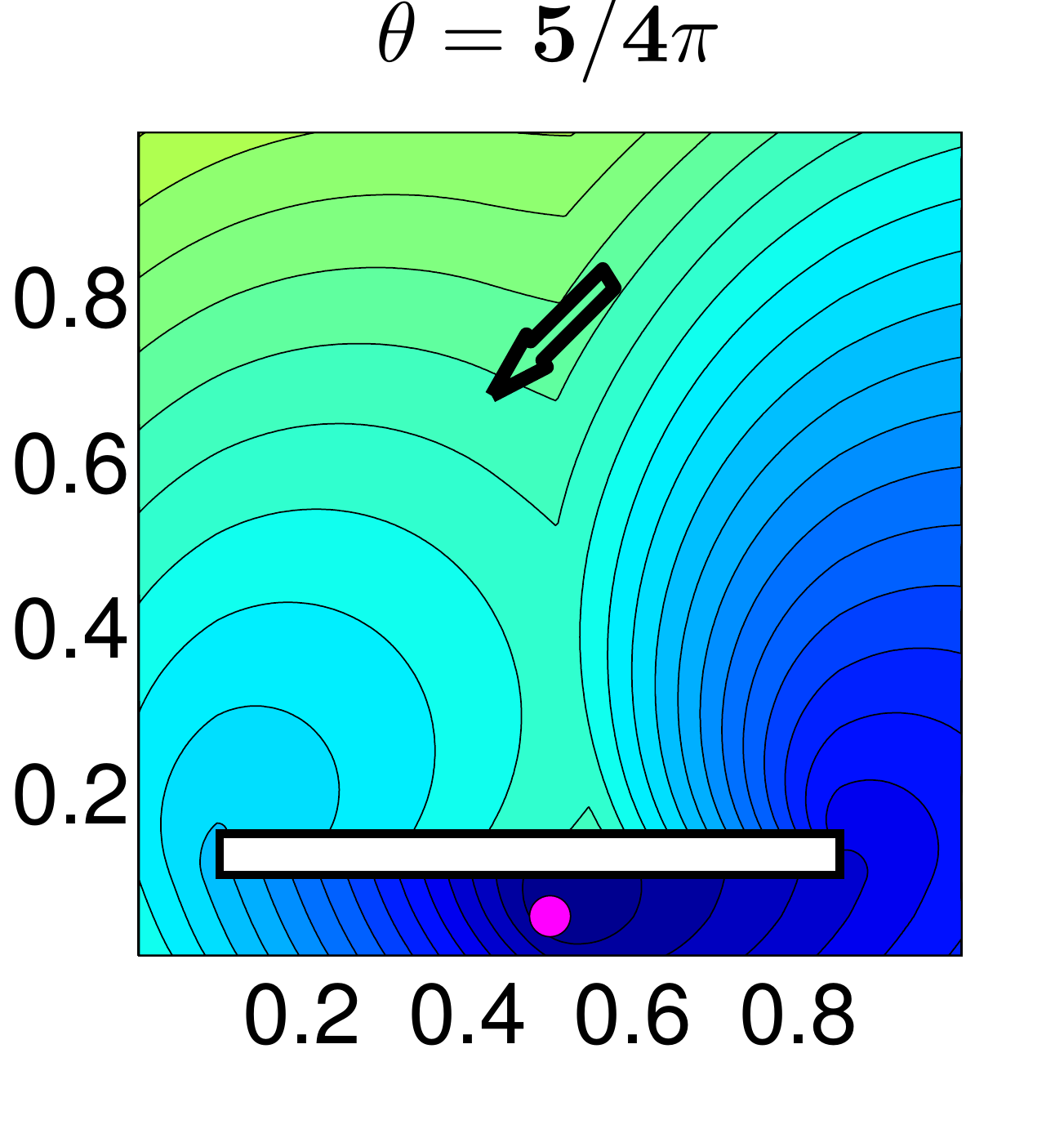}
\includegraphics[width=0.25\textwidth, height=0.25\textwidth]{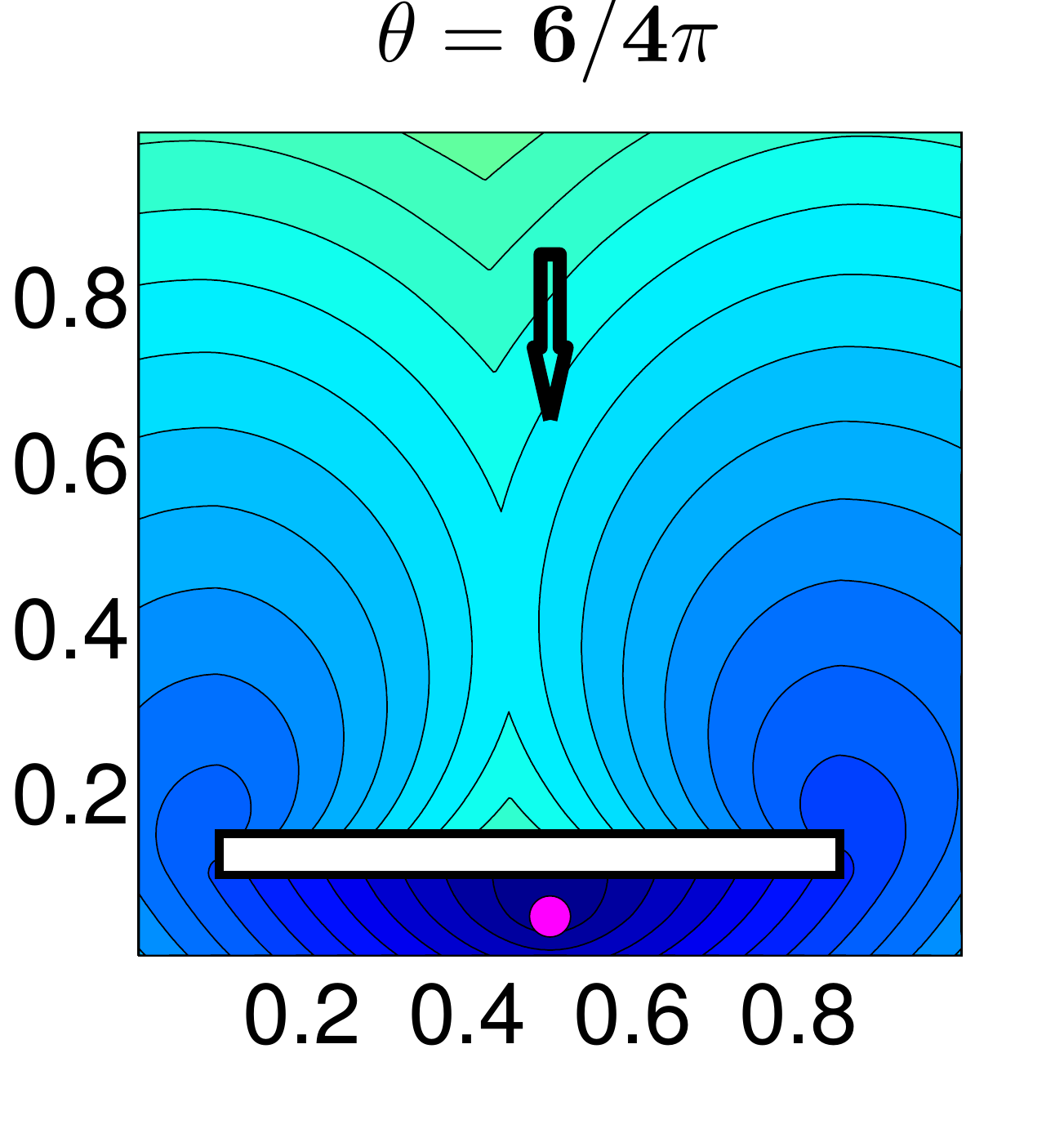}
\includegraphics[width=0.25\textwidth, height=0.25\textwidth]{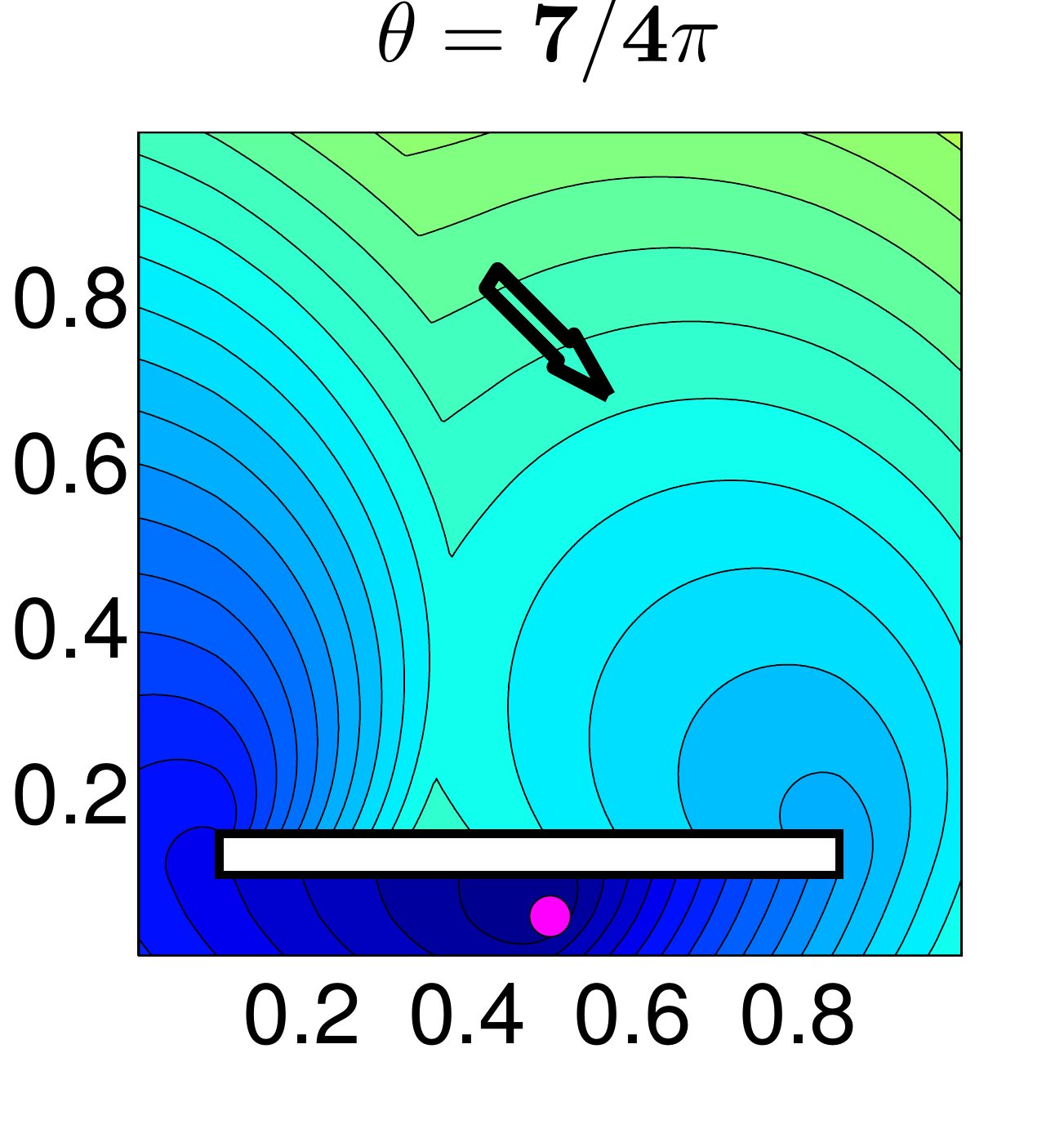}
}
\centerline{
\includegraphics[width=0.6\textwidth, height=0.2\textwidth]{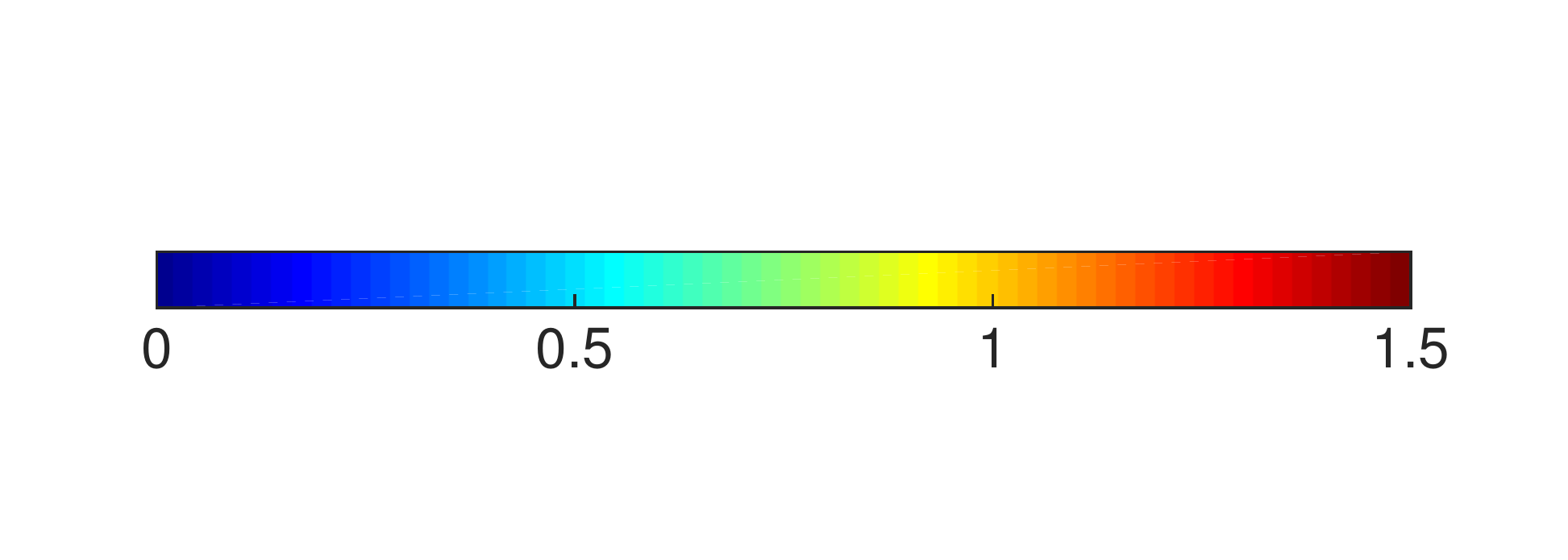}
}
\caption{
Value functions for 8 wind modes discretizing the degenerate diffusion model in section \ref{sec:BM}.
% COMMENT REMOVED
% COMMENT REMOVED
Wind directions represented by arrows.
% COMMENT REMOVED
}
% COMMENT REMOVED
% COMMENT REMOVED
% COMMENT REMOVED
% COMMENT REMOVED
% COMMENT REMOVED
% COMMENT REMOVED
\label{fig:BM}
\end{figure}

%% file: tex/conclude.tex
We have considered the problem of optimal path planning under uncertainty due to stochastic switching of {\em modes} in the system dynamics.
Such piecewise-deterministic models result in weakly-coupled systems of static HJB equations, and we have discussed both the discretization strategies and the iterative numerical algorithms for solving them.
Asymptotic approximations, obtained by letting the rates of switching tend to either zero or $\infty$, are less challenging computationally.
But our experimental evidence showed that the use of controls recovered from them results in significant performance degradation in non-asymptotic regimes.
The model problems used in our computational experiments are quite simple, but already reveal the complications and opportunities of taking stochastic switching into account.
Some generalizations (e.g., more modes, general anisotropic cost/dynamics, or spatially-dependent switching rates) would be trivial.  It would be both more challenging and interesting to incorporate some non-Markovian switching.

From practitioner's perspective, piecewise-deterministic models will become much more attractive once efficient numerical methods become available.
It will be interesting to see if the number of iterations can be significantly reduced despite the mode-coupling and the mode-dependence of optimal directions.
This is obviously doable for a  subclass of {\it structurally causal} switching systems:
if the modes can be re-numbered so that the transition matrix $\Lambda = (\lambda_{ij})$ becomes upper triangular,
the equations in \eqref{coupledHJB} can be de-coupled and solved one at a time.  Examples of this type are common in
some of the applications (e.g., \cite{haurie2006stochastic, haurie2005multigenerational, andrews2013deterministic}), but it is still unclear whether similar efficiency gains are attainable in the general case.
Another relevant practical task is to derive a priori upper bounds on performance degradation due
to using $\balpha_*^0$ or $\balpha_*^{\infty}$ instead of the correct  $\balpha_*.$  Practitioners could use this information to decide whether solving the coupled system is actually worthwhile.

Many features already available for the fully deterministic case would be similarly attractive in
the piecewise-deterministic setting.
This includes multi-criteria path planning \cite{KumarVlad} and dynamic domain restriction \cite{ClawsonChaconVlad}.  It will be similarly useful
% COMMENT REMOVED
to extend
the concept of {\em risk-sensitive} optimal controls, already popular in standard stochastic control problems \cite{fleming2006controlled}.
Another alternative is to optimize the expected cost with a constraint on the ``worst case scenario'' \cite{ErmonGSV12},
since the latter is well-defined in
stochastic switching problems.
% COMMENT REMOVED

%% file: tex/appendixB.tex
Our goal is to derive an upper bound on the difference between mode-specific value functions: $||u_i-u_j||_{\infty}$ with $i, j \in\M$.
The basic idea is that, starting from $\x$ in the mode $i$,
it is often possible to hover around $x$ long enough until we transition into a (possibly preferable) mode $j$.
% COMMENT REMOVED
% COMMENT REMOVED
% COMMENT REMOVED
We define {\em the first hitting time} $\kappa_{ij}=\min\limits_{t\geq0}\{t: m(t_0+t)=j \, | \, m(t_0)=i\}$ with $\kappa_{ii}=0, \, \forall i\in\M$, and make the following assumptions about the controlled process.

\begin{assumption}\label{assum:Cost}
The running cost $C_i(\x, \ba) \leq \bC$ for all $(\x, i, \ba) \in \Omega \times \M \times A$.
\end{assumption}

\begin{assumption}\label{assum:Lip}
$u_i(\x)$ is L-Lipschitz continuous for all $i \in \M$.
That is,
\begin{align}\label{eq:Lip}
 | u_i(\x_1) - u_i(\x_2) | \, \leq \,  L \| \x_1 - \x_2 \|, \qquad \forall \x_1, \x_2 \in \Omega, \, i \in \M
 \end{align}
\end{assumption}

\begin{assumption}\label{assum:zero_speed}
For every starting configurations $(\x,i) \in \Omega \times \M$ and every $\delta > 0$, there exists a control
$\balpha_{\delta} \in \mathcal{A}$ such that
\begin{align}
\| \y (t) - \x \| < \delta, \quad \forall t > 0,
\end{align}
where $\y$ satisfies equation \eqref{eq:ODE} with $\balpha = \balpha_{\delta}$.
\end{assumption}
\begin{theorem}\label{thm:upperbound}
If the assumptions  \ref{assum:Cost}-\ref{assum:zero_speed} hold, then
\begin{align}\label{eq:upper_bound}
|u_i(\x)-u_j(\x)| \; \leq \; \bC \, \max\{ \mathbb{E}[\kappa_{ij}], \, \mathbb{E}[\kappa_{ji}] \},  \qquad \forall i, j \in \M.
\end{align}
\end{theorem}

\begin{proof}
Starting from $\x\in\Omega$ in mode $i$, we choose a number $\delta > 0$ small enough such that $\delta < dist(\x, Q)$,
and introduce the following hybrid strategy (denoted $\tilde{\balpha}_{\delta}$).
We use $\balpha_{\delta}(\x, m(t))$ described in assumption \ref{assum:zero_speed} until the mode switches to $j$
at some point $\tilde{\x} = \y(\kappa_{ij})$.  From there on, we switch to the optimal strategy corresponding to $(\tilde{\x},j)$.

 While $t \leq \kappa_{ij}$, we have $\| \y(t) - \x \| < \delta < dist(\x, Q)$; so the trajectory will not reach $Q$ before the mode switches to $j$.
% COMMENT REMOVED
% COMMENT REMOVED
% COMMENT REMOVED
% COMMENT REMOVED
 Therefore,
 \begin{align}\label{eq:ui_less_than}
 u_i(\x) \; \leq \;
 \mathbb{E}\left[ J(\x, i, \tilde{\balpha}_{\delta}) \right]
\;  \leq \;
 \bC\, \mathbb{E} \left[ \kappa_{ji} \right] \, + \,
  \mathbb{E}\left[ J(\tilde{\x}, j, \tilde{\balpha}_{\delta}) \right]
\;  =  \;
 \bC \, \mathbb{E} \left[ \kappa_{ij} \right] \, + \,  u_j(\tilde{\x}) .
  \end{align}
Since $\|\tilde{\x} - \x\| \; < \; \delta$, inequalities \eqref{eq:Lip} and \eqref{eq:ui_less_than} yield
% COMMENT REMOVED
 \begin{align}\label{eq:i-j}
u_i(\x) \; \leq \; \bC \, \mathbb{E} \left[ \kappa_{ij} \right] \, + \,  u_j(\x)  \, + \, L\delta.
 \end{align}
Reversing the roles of $i$ and $j$, we also have
\begin{align}\label{eq:j-i}
u_j(\x) \; \leq \; \bC \, \mathbb{E} \left[ \kappa_{ji} \right] \, + \,  u_i(\x)  \, + \, L\delta.
 \end{align}
 To finish the proof, we combine \eqref{eq:i-j} and \eqref{eq:j-i}, and then let  $\delta \to 0^+$.
% COMMENT REMOVED
% COMMENT REMOVED
% COMMENT REMOVED
% COMMENT REMOVED
% COMMENT REMOVED
% COMMENT REMOVED
% COMMENT REMOVED
% COMMENT REMOVED
% COMMENT REMOVED
% COMMENT REMOVED
% COMMENT REMOVED
\end{proof}

Suppose the transition rate matrix is $c\Lambda$ for $c>0$ and a fixed irreducible matrix $\Lambda$.
If we define
% COMMENT REMOVED
$
 \rho_{ij} \, = \, \max \{ \mathbb{E}[\kappa_{ij}], \, \mathbb{E}[\kappa_{ji}] \},
$
 %\end{align}
a standard argument from the theory of Markov processes \cite{norris1998markov} shows that $\rho_{ij} \propto c^{-1}$.
% COMMENT REMOVED
% COMMENT REMOVED
% COMMENT REMOVED
% COMMENT REMOVED
Therefore, if we send $c \to \infty$,
$|u_i(\x) - u_j(\x)| = O(c^{-1})$
uniformly in $\Omega$ for all $i\neq j$.